\documentclass[11pt]{amsart}
\usepackage[utf8x]{inputenc}
\usepackage[english]{babel}
\usepackage[T1]{fontenc}
 
\usepackage{graphicx}
\usepackage{caption}
\usepackage{subcaption}
\usepackage{amssymb}
\usepackage{amsmath}
\usepackage{mathrsfs}
\usepackage{mathtools}
\usepackage{enumerate}
\usepackage[hidelinks]{hyperref}
\usepackage{indentfirst}
\usepackage{enumerate}
\usepackage{appendix}
\usepackage{latexsym}
\usepackage{url}
\usepackage{color}
\usepackage{accents}
\usepackage{setspace}
\usepackage{pdfpages}
\usepackage{stmaryrd}
\usepackage{xfrac}
\usepackage{catchfile}
\usepackage{amsrefs}

\usepackage[margin=2.5cm]{geometry}
\allowdisplaybreaks

\BibSpec{article}{%
	+{}{\PrintAuthors} {author}
	+{,}{ } {title}
	+{, }{\textit } {journal}
	+{}{ \parenthesize} {date}
		+{,  }{no. } {volume}
	+{,}{ } {pages}
	+{,}{ } {note}
}

\ExplSyntaxOn
\makeatletter

\newcommand{\gobblefirst}[1]{\expandafter\@gobble#1}
\makeatother

\DeclarePairedDelimiter\ceil{\lceil}{\rceil}
\DeclarePairedDelimiter\floor{\lfloor}{\rfloor}
\ExplSyntaxOff
\BibSpec{book}{%
	+{}{\PrintAuthors}  {author}
	+{. }{}{title}
	+{,}{ }{series}
	+{,}{ vol.~}{volume}
	+{. }{\textit}{publisher}
	+{,}{ ISBN \gobblefirst}{isbn}
}
\BibSpec{collection.article}{%
	+{}{\PrintAuthors}{author}
	+{, }{}{title}
	+{, }{\textit}{booktitle}
	+{, }{ \DashPages}{pages}
	+{,}{ }{series}
	+{, }{}{volume}
	+{, }{\textit}{publisher}
	+{,}{ }{date}
}

\mathcode`l="8000
\begingroup
\makeatletter
\lccode`\~=`\l
\DeclareMathSymbol{\lsb@l}{\mathalpha}{letters}{`l}
\lowercase{\gdef~{\ifnum\the\mathgroup=\m@ne \ell \else \lsb@l \fi}}%
\endgroup

\def\XXint#1#2#3{{\setbox0=\hbox{$#1{#2#3}{\int}$ }
		\vcenter{\hbox{$#2#3$ }}\kern-.6\wd0}}

\newtheorem{prop}{Proposition}
\newtheorem{thm}[prop]{Theorem}
\newtheorem{lem}[prop]{Lemma}
\newtheorem{coro}[prop]{Corollary}

\newtheorem{rema}[prop]{Remark}

\title[Tangential limits of stable minimal capillary surfaces]{ 
	Tangential limits of stable minimal capillary surfaces}
\author{Michael Eichmair}
\address{
	\textnormal{Michael Eichmair \newline  \indent
		University of Vienna \newline \indent
		Faculty of Mathematics  \newline \indent
		Oskar-Morgenstern-Platz 1 \newline \indent
		1090 Vienna, 	Austria  \newline\indent 
		\href{https://orcid.org/0000-0001-7993-9536}{https://orcid.org/0000-0001-7993-9536} \newline\indent	
		\href{mailto:michael.eichmair@univie.ac.at}{michael.eichmair@univie.ac.at}}
}

\author{Thomas Koerber}
\address{\textnormal{Thomas Koerber  \newline \indent
		University of Vienna \newline \indent
		Faculty of Mathematics  \newline \indent
		Oskar-Morgenstern-Platz 1 \newline \indent 1090 Vienna,	Austria \newline\indent 
		\href{https://orcid.org/0000-0003-1676-0824}{https://orcid.org/0000-0003-1676-0824} \newline \indent
		\href{mailto:thomas.koerber@univie.ac.at}{thomas.koerber@univie.ac.at}}
}

\begin{document}
		\date{\today}
	\maketitle
\begin{abstract}
We characterize all compact embedded stable minimal capillary surfaces with capillary angle close to either $0$ or $\pi$ that are supported on a complete embedded minimal surface with finite total curvature that is not an affine plane. Moreover, we characterize all compact embedded weakly stable minimal capillary surfaces with capillary angle close to either $0$ or $\pi$ that are supported on a closed surface whose mean curvature is positive and has no degenerate maxima. An important ingredient in our work are curvature estimates for sequences of  weakly stable minimal capillary surfaces with capillary angles tending to $0$ or $\pi$ that enable us to analyze the tangential limits of such sequences at suitable scales.

\end{abstract}
	\section{Introduction}
	Let $S\subset \mathbb{R}^3$ be a properly embedded minimal surface with positive reach  such that $\mathbb{R}^3\setminus S$ has  two components. We call $S$ a support surface. Let $D(S)$ be one of these components and $\Sigma \subset \operatorname{cl}D(S)$  a properly embedded surface with $\partial \Sigma\subset S$. $\Sigma$ is called a minimal capillary surface if its mean curvature  vanishes and if $S$ and $\Sigma$ intersect at a constant angle  $0<\theta<\pi$ along $\partial \Sigma$. Such surfaces model the steady-state interface between two incompressible fluids in the absence of gravity; see, e.g., \cite{Taylor}.   Variationally,  they arise as the critical points of the free energy
	\begin{align} \label{free energy}
	|\Sigma|+\cos(\theta)\,|S(\Sigma)|
	\end{align} 
	where $S(\Sigma)$ is the wetting surface of $\Sigma$, i.e., the region in $S$ enclosed by $\partial \Sigma$ that intersects $\Sigma$ at an acute angle. In particular, minimal capillary surfaces pass the first derivative test for least area given the area of their wetting surface. This variational problem is known as the capillary problem. It is arguably the most basic setting   involving minimal surfaces after Plateau's problem.  In recent years, minimal capillary surfaces have been used to establish rigidity results in geometry such as in C.~Li's proof \cite{Li} of a polyhedron comparison conjecture due to M.~Gromov \cite{Gromov} or the authors' proof \cite{eichmair2024penrose} of the Penrose inequality in extrinsic geometry conjectured by G.~Huisken \cite{volkmann}. The latter implies an extrinsic characterization of the catenoid among all complete embedded minimal surfaces with finite total curvature. We also point to the works of D.~Ko and X.\,Yao \cite{ko1,ko2} and of Y.~Wu \cite{Wu}, who have used minimal capillary surfaces to study the interplay between  the scalar curvature of a Riemannian manifold and the mean curvature of its  boundary, as well as to several recent contributions to the regularity theory of minimal capillary surfaces such as, e.g.,  \cite{DePhillipis,HongSaturnino,LiZhouZhu,dePhilippis2,chodoshedelenli,DeMasietal}.
	
From an analytic point of view, the first variation of \eqref{free energy} given by 
\begin{align} \label{first variation intro} 
\int_{\Sigma}H(\Sigma)\,f-\int_{\partial\Sigma} \sqrt{\frac{\nu(S)\cdot \nu(\Sigma)}{1-(\nu(S)\cdot \nu(\Sigma))^2}}\,f
+\cos(\theta)\int_{\partial \Sigma} \sqrt{\frac{1}{1-(\nu(S)\cdot \nu(\Sigma))^2}}\,f
\end{align}  degenerates as the capillary angle $\theta$ tends to $0$ or to $\pi$. Here, $\nu(S)$ is the unit normal of $S$ pointing out of $D(S)$, $\nu(\Sigma)$ the unit normal of $\Sigma$, $H(\Sigma)$ the mean curvature of $\Sigma$ with respect to $\nu(\Sigma)$, and $f\in C^\infty(\Sigma)$ the normal speed of the variation.  It is a  fundamental problem to understand how sequences of minimal capillary surfaces behave in this setting. The degeneracy of \eqref{first variation intro} suggests that this should only be possible under additional geometric assumptions.  To this end, we say that a compact minimal capillary surface $\Sigma \subset \mathbb{R}^3$ supported on $S$ with capillary angle $0<\theta<\pi$ is stable for the free energy if 
\begin{align} \label{stable intro} 
\int_\Sigma |\nabla^\Sigma f|^2\geq 		\int_\Sigma |h(\Sigma)|^2\,f^2-\int_{\partial \Sigma} k(\Sigma)\cdot\mu(\Sigma)\,f^2+\frac{1}{\sin\theta}\,\int_{\partial \Sigma} H(S)\,f^2
\end{align} 
for all $f\in C^\infty(\Sigma)$. Here, $\nabla^\Sigma f$ is the gradient of $f$, $h(\Sigma)$ the second fundamental form of $\Sigma$ with respect to $\nu(\Sigma)$, $\mu(\Sigma)$ the outward co-normal of $\partial \Sigma$, $k(\Sigma)$ the  curvature vector of $\partial \Sigma$ in $\mathbb{R}^3$, and $H(S)$ the mean curvature of $S$ with respect to $\nu(S)$; see, e.g., \cite{HongSaturnino}. We say that $\Sigma$ is weakly stable for the free energy if \eqref{stable intro} holds for all $f\in C^\infty(\Sigma)$ with 
$$
\int_{\partial \Sigma} f=0.
$$ We note that $\Sigma$ is  stable for the free energy if and only if $\Sigma$ passes the second derivative test for \eqref{free energy}. Likewise, $\Sigma$ is weakly  stable for the free energy if and only if $\Sigma$ passes the second derivative test for \eqref{free energy} among variations that preserve the area of the wetting surface.
 Our main contribution in this paper is twofold:
\begin{itemize}
	\item[$\circ$] In the case where $S$ is a complete embedded minimal surface with finite total curvature that is not an affine plane, we characterize all compact minimal capillary surfaces $\Sigma \subset \mathbb{R}^3$ supported on $S$ with capillary angle close to either $0$ or $\pi$ that are stable for the free energy.
		\item[$\circ$] In the case where $S$ is a closed surface whose mean curvature is positive and has no degenerate maxima, we characterize all compact minimal capillary surfaces $\Sigma\subset \mathbb{R}^3$ supported on $S$ with capillary angle close to either $0$ or $\pi$ that are weakly stable for the free energy.
\end{itemize} 
We note that there are no compact minimal capillary surfaces supported on an affine plane. In particular, our results are not perturbative. 
The class of complete embedded minimal surfaces with finite total curvature contains infinitely many geometrically distinct examples such as the plane, the catenoid,  Costa's minimal surface, or more complicated surfaces that do not exhibit symmetries; see, e.g.,  \cite{Costa,Traizet}. Our recent proof of the Penrose inequality in extrinsic geometry \cite{eichmair2024penrose} shows that minimal capillary surfaces are well-suited to study the geometry and topology of such surfaces.
 Among closed surfaces with positive mean curvature, the class of surfaces whose mean curvature has no degenerate maxima is generic. 
  We note that the assumptions on the mean curvature of $S$ and the (weak) stability for the free energy of $\Sigma$ cannot be dropped in our results; see Remarks \ref{counter 1} and \ref{counter 2}.

\subsection*{Outline of our results} To describe our contributions, note that by reversing the orientation, a minimal capillary surface $\Sigma\subset \mathbb{R}^3$  supported on a support surface $S\subset \mathbb{R}^3$ with  capillary angle $0<\theta<\sfrac{\pi}2$ may also be viewed  as a minimal capillary surface $\Sigma\subset \mathbb{R}^3$  supported on   $S\subset \mathbb{R}^3$ with  capillary angle $\pi-\theta$. In the remainder of this paper, we will therefore only  consider minimal capillary surfaces $\Sigma \subset \mathbb{R}^3$ with capillary angle $\sfrac{\pi}2<\theta<\pi.$ In the case where $\theta=\sfrac{\pi}2$, \eqref{free energy} is simply the area functional, which has been studied extensively in the literature. 

\subsubsection*{Complete embedded minimal surfaces with finite total curvature}
 
Let $S\subset \mathbb{R}^3$ be a complete embedded minimal surface with Gauss curvature $K(S)$. Recall that $S$ is said to be of finite total curvature if 
$$
-\int_S K(S)<\infty,
$$
in which case there is $\lambda>0$ such that $S\setminus B_\lambda(0)$ has finitely many components $S^1,\,\ldots,\,S^m\subset \mathbb{R}^3$ called the ends of $S$, each of which is contained in the graph of a function with asymptotically constant gradient. By a rotation, we may arrange that the ends of $S$ are horizontal and ordered in terms of their height.  We choose $\nu(S)$ to be pointing downward on the top end $S^1$ and let $D(S)$ be the component of $\mathbb{R}^3\setminus S$ that $\nu(S)$ points out of. For example, in the case where $S$ is a catenoid, $m=2$ and $D(S)$ is the simply connected component of $\mathbb{R}^3\setminus S$.  

Our first main result proves the existence of a  canonical family of minimal capillary surfaces supported on $S$.

\begin{thm} \label{THM A}
	Let $S\subset \mathbb{R}^3$ be a complete embedded minimal surface with finite total curvature that is  not an affine plane. Let $S^1,\,\ldots,\, S^m$ be the ends of $S$. There are $\lambda>1$,  $\sfrac{\pi}2<\theta_0<\pi$, and minimal capillary surfaces $\Sigma(\theta)\subset \mathbb{R}^3$ supported on $S$ with capillary angle $\theta_0<\theta<\pi$ with the following properties.
	\begin{align*}
		&\circ  \qquad \text{There holds $\partial \Sigma(\theta)\subset S^1$ for  all $\theta_0<\theta<\pi$.}
		\\&\circ \qquad \{y\in S^1:|y|> \lambda\}\subset \{x\in \partial \Sigma(\theta):\theta_0<\theta<\pi\} 
		\\	&\circ \qquad \text{Each $\Sigma(\theta)$ is stable for the free energy \eqref{free energy}.}
		\\&\circ \qquad \text{The surfaces }\sin(\theta)\,\Sigma(\theta)\text{ converge smoothly to a centered round disk as $\theta\nearrow \pi$.}
		\\ &\circ \qquad \text{The surfaces $\{\Sigma(\theta):\theta_0<\theta<\pi\}$ form  a smooth foliation.}
	\end{align*}
\end{thm}
Our second main result characterizes all compact minimal capillary surfaces supported on $S$ with capillary angle close to $\pi$ that are stable for the free energy  as being part of the family of surfaces given by Theorem \ref{THM A}. For the statement of this characterization, we distinguish between the cases where the number $m$ of ends of $S$ is odd or even. Note that $\nu(S)$ points in the same direction on $S^1$ and $S^m$ if $m$ is odd.
\begin{thm} \label{THM B} Let $S\subset \mathbb{R}^3$ be a complete embedded minimal surface with finite total curvature that is not an affine plane. Let $S^1,\,\ldots,\, S^m$ be the ends of $S.$ Assume that $m$ is odd. Let $\Sigma\subset \mathbb{R}^3$ be a   minimal capillary surface supported on $S$ with capillary angle $\theta_0<\theta<\pi$ that   is stable  for the free energy \eqref{free energy}. There holds $\Sigma=\Sigma(\theta)$ where $\Sigma(\theta)$ is as in Theorem \ref{THM A}.
\end{thm} 
In the case where $m$ is even, $\nu(S)$ points in opposite directions on $S^1$ and $S^m$. The construction from Theorem \ref{THM A} can therefore be applied with $S^m$ in place of $S^1$. We denote the family from Theorem \ref{THM A} by $\{\Sigma^1(\theta):\theta_0<\theta<\pi\}$ and the family obtained by applying Theorem \ref{THM A} with $S^m$ in place of $S^1$ by $\{\Sigma^m(\theta):\theta_0<\theta<\pi\}$; see Figure \ref{four ends}.
	\begin{figure}\centering
\begin{subfigure}{.5\linewidth}	\includegraphics[width=\linewidth]{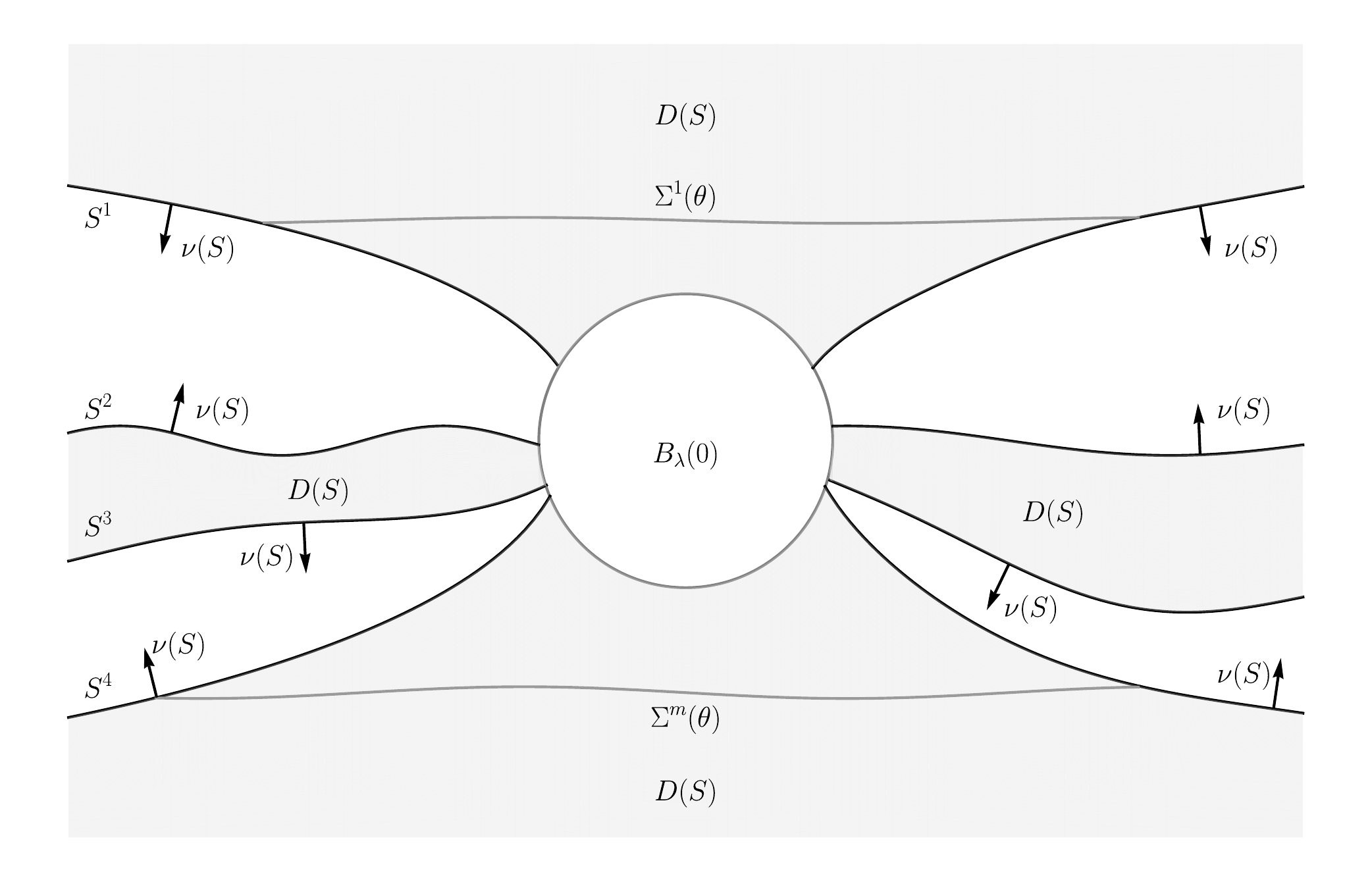}
\end{subfigure}%
\begin{subfigure}{.5\linewidth}
	\includegraphics[width=\linewidth]{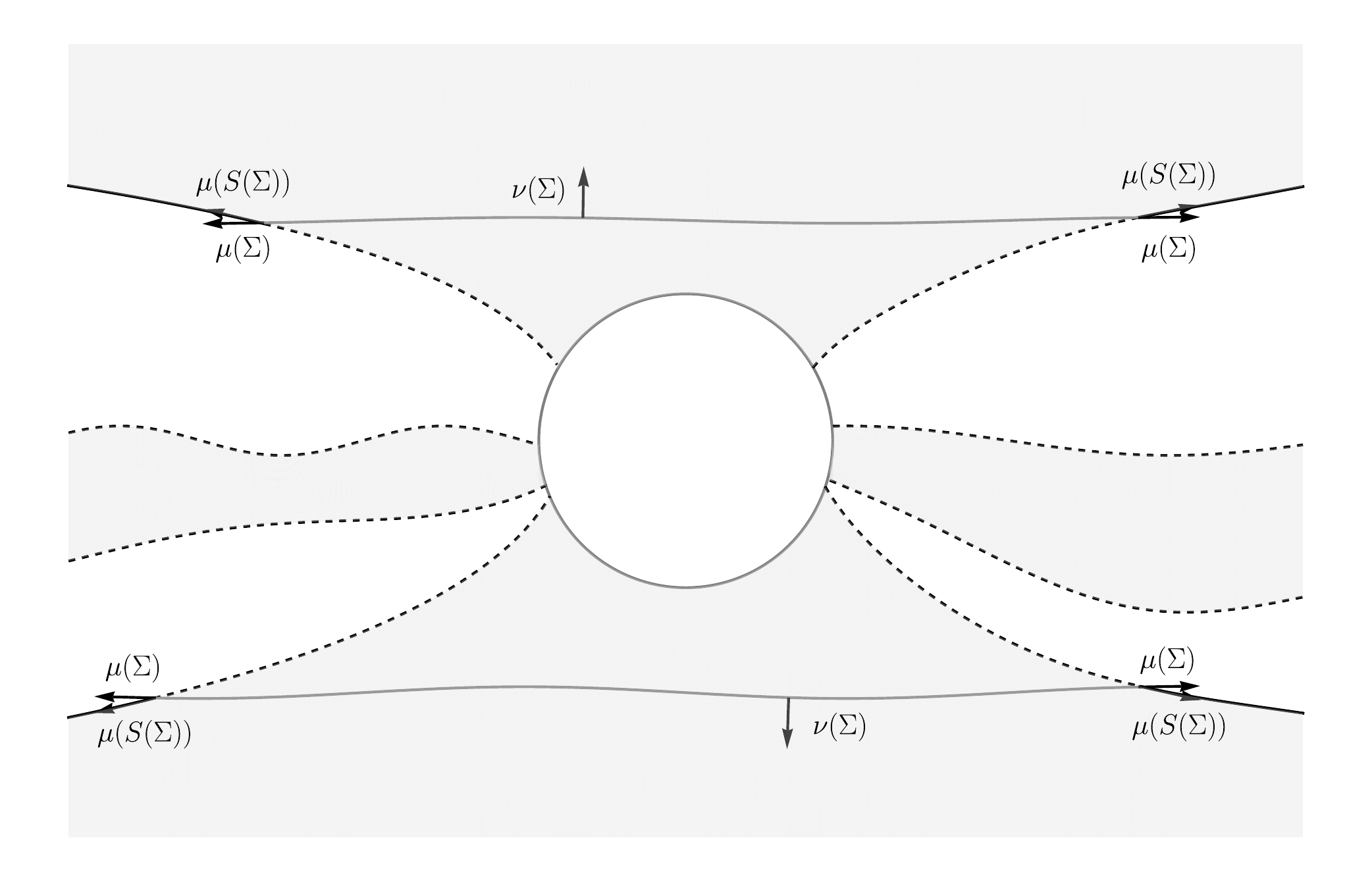}
\end{subfigure}
	\caption{An illustration of a complete embedded minimal surface $S\subset \mathbb{R}^3$ with four ends outside the ball $B_\lambda(0)$ and the minimal capillary surface $\Sigma=\Sigma^1(\theta)\cup\Sigma^m(\theta)\subset \mathbb{R}^3$ supported on $S$. The shaded region indicates the domain $D(S)$ bounded by $S$. The dotted line on the right indicates the wetting surface of $\Sigma$ 
	}
	\label{four ends}
\end{figure}
\begin{thm} \label{THM B 2} Let $S\subset \mathbb{R}^3$ be a complete embedded minimal surface with finite total curvature. Let $S^1,\,\ldots,\,S^m$ be the ends of $S$. Assume that $m$ is even. Let $\Sigma\subset \mathbb{R}^3$ be a  minimal capillary surface supported on $S$ with capillary angle $\theta_0<\theta<\pi$  that  is stable  for the free energy \eqref{free energy}. Then either $\Sigma=\Sigma^1(\theta)$, $\Sigma=\Sigma^m(\theta)$, or $\Sigma=\Sigma^1(\theta)\cup \Sigma^m(\theta)$. 
\end{thm} 
\begin{rema} \label{counter 1}
The assumption that $\Sigma\subset \mathbb{R}^3$ is stable for the free energy in Theorem \ref{THM B 2} cannot be dropped. Indeed, in the case of the catenoid
$
S=\{(y,t)\in \mathbb{R}^2\times \mathbb{R}:\cosh t=|y|\},
$ 	
the sequence of minimal capillary surfaces $$\Sigma_k=\left(\frac{k-1}{k}\,S\right)\cap \operatorname{cl}D(S)$$
supported on $S$ converges smoothly to $\{(y,t)\in \mathbb{R}^2\times [-1,1]:\cosh t=|y|\}$. 
\end{rema}
\subsubsection*{Closed support surfaces with positive mean curvature}
Let $S\subset \mathbb{R}^3$ be a closed surface whose mean curvature is positive and $z\in S$  a nondegenerate local maximum  of $H(S)$. Our third  main result proves the existence  of a  canonical family of minimal capillary surfaces supported on $S$  that are contained in a small neighborhood of  $z$.
\begin{thm} \label{THM D}
	Let $S\subset \mathbb{R}^3$ be a closed surface with positive mean curvature $H(S)$. Assume that $\nabla^SH(S)(z)=0$ and $\nabla^S\nabla^S H(S)(z)<0$ for some $z\in S$.	There are $\varepsilon>0$, $\sfrac{\pi}2<\theta_0<\pi$, and  minimal capillary surfaces $\Sigma(\theta) \subset \mathbb{R}^3$ supported on $S$ with capillary angle $\theta_0<\theta<\pi$ with the following properties.
	\begin{align*}
		&\circ \qquad \{y\in S:0<\operatorname{dist}_S(y,z)<\varepsilon\}\subset \{x\in\partial \Sigma(\theta):\theta_0<\theta<\pi\}.
		\\		&\circ \qquad \text{Each $\Sigma(\theta)$ is weakly stable for the free energy \eqref{free energy}.}
		\\&\circ  \qquad \text{The surfaces  }\sfrac{1}{\sin\theta}\,(\Sigma(\theta)-z)\text{ converge smoothly to a centered round disk as $\theta\nearrow \pi$.}
			\\ &\circ \qquad \text{The surfaces $\{\Sigma(\theta):\theta_0<\theta<\pi\}$ form  a smooth foliation.}
	\end{align*} 
\end{thm}
Our fourth main result characterizes all compact minimal capillary surfaces supported on $S$ with capillary angle close to $\pi$ that are weakly stable for the free energy  as being part of the families of surfaces given by  Theorem \ref{THM D}.
\begin{thm} \label{THM C}
	Let $S\subset \mathbb{R}^3$ be a closed surface whose mean curvature $H(S)$ is positive and has no degenerate maxima. Let $\Sigma\subset \mathbb{R}^3$ be  a minimal capillary surface supported on $S$ with capillary angle $\theta_0<\theta<\pi$ that is weakly stable for the free energy. There is $z\in S$ where $H(S)$ achieves a local maximum such that $\Sigma=\Sigma(\theta)$ where $\Sigma(\theta)$ is as in Theorem \ref{THM D}. 
\end{thm}
\begin{rema} \label{counter 2}
	The assumption that the mean curvature of $S$ has no degenerate maxima in Theorem \ref{THM C} cannot be dropped. Indeed, if $S=\{x\in \mathbb{R}^3:|x|=1\}$ is the unit sphere and $\Sigma\subset \mathbb{R}^3$ a minimal capillary surface supported on $S$, then, for every rotation $R:\mathbb{R}^3\to \mathbb{R}^3$, $R(\Sigma)$ is also a minimal capillary surface supported on $S$ with the same capillary angle.
\end{rema}
\begin{rema} 
	The proof of Theorem \ref{THM C} also shows the following: Let $S\subset \mathbb{R}^3$ be a closed surface with positive mean curvature $H(S)$ and $\Sigma_k\subset \mathbb{R}^3$  minimal capillary surfaces supported on $S$ with capillary angles $\sfrac{\pi}{2}<\theta_k<\pi$. Assume that each $\Sigma_k$ is  weakly stable for the free energy and that $\theta_k\nearrow \pi$. Passing to a subsequence, there is $z\in S$ with $\nabla^SH(S)=0$ and $\nabla^S\nabla^S H(S)\leq 0$ such that $\sfrac{1}{\sin\theta_k}\,(\Sigma_k-z)$ converges smoothly to  a centered round disk.
\end{rema}

\subsection*{Outline of related results}
In the case where $S=\mathbb{R}^2\times \{0\}$, H.~Hong and A.~Saturnino \cite{HongSaturnino} and, independently, C.~Li, X.~Zhou, and J.~Zhu \cite{LiZhouZhu} have shown that every minimal capillary surface supported on $S$ that is weakly stable for the free energy is a half-space. This result may be viewed as  the stable Bernstein theorem \cite{Bernstein1,Bernstein2} for minimal capillary surfaces. In the case where $S=\{x\in \mathbb{R}^3:|x|=1\}$, G.~Wang and C.~Xia \cite{WangXia} have characterized all minimal capillary surfaces supported on $S$ that are weakly stable for the free energy as round disks. They have given a similar characterization for constant mean curvature surfaces that intersect $S$ at a constant angle, which may be viewed as a variant of a classical result of J.~Barbosa and M.~do Carmo \cite{BarbosadoCarmo}; see also \cite{DeRosa}.

The proofs of the results in the previous paragraph  exploit the  symmetries of the support surface $S$. They do not involve  an asymptotic analysis of the capillary problem as the capillary angle degenerates.    Recently, in the case where $S$ is a plane, O.~Chodosh, N.~Edelen, and C.~Li \cite{chodoshedelenli} have made rigorous an analytic connection between the capillary problem   and the so-called one-phase Bernoulli problem \cite{AltCafarelli}. Building on this connection, they have characterized as half-spaces all minimal capillary cones in ambient dimension $4$  with capillary angle  close to $\pi$ that minimize the 4-dimensional analog of the free energy \eqref{free energy}; see also, e.g.,  \cite{pacati2025some,chodosh2025weiss,Tsiamis1,Tsiamis2} for subsequent developments in this direction. As we will explain below, our arguments draw inspiration from their approach in some places. 

We mention that an ambitious program initiated by the seminal work of G.~Huisken and S.-T.~Yau \cite{HuiskenYau} has led to the complete characterization of large weakly stable constant mean curvature spheres in asymptotically flat Riemannian 3-manifolds under suitable assumptions on the scalar curvature; see \cite{Brendle,EichmairBrendle,EichmairChodosh,cmc} and the references therein. Likewise, small weakly stable constant mean curvature surfaces have been characterized in compact Riemannian 3-manifolds near nondegenerate local maxima of the scalar curvature; see, e.g., \cite{Ye,Pacard}. Finally, similar programs have been initiated and partially completed to characterize  large area-constrained Willmore spheres in asymptotically flat Riemannian 3-manifolds and small area-constrained Willmore spheres in compact Riemannian 3-manifolds; see, e.g., \cite{LMS1,LMS2,largearea,HuiskenYauUniqueness}. Our contributions here may be viewed as mirroring these results of intrinsic geometry in extrinsic geometry. The methods to establish uniqueness are very different, however.

In our recent proof of the Penrose inequality in extrinsic geometry \cite{eichmair2024penrose}, we construct minimal capillary surfaces $\Sigma_\theta \subset \mathbb{R}^3$ of capillary angle $\sfrac{\pi}2<\theta<\pi$ supported on a complete embedded minimal surface $S\subset \mathbb{R}^3$ with finite total curvature as minimizers of the free energy \eqref{free energy}. There, we only obtain crude estimates on the geometry of $\Sigma_\theta$ and rely on the optimal isoperimetric inequality to compute $\sfrac{1}{\sin \theta}\,\sqrt{|\Sigma_\theta|}$ as $\theta \nearrow \pi$; see \cite[Lemma 26 and Proposition 37]{eichmair2024penrose} and \cite{Carleman,Brendleiso,EichmairBrendleIso}. Theorem \ref{THM A} and Theorem \ref{THM B} respectively Theorem \ref{THM B 2} provide an alternative  way to arrive at this limit.

\subsection*{Outline of our arguments}
	Let $\tilde S_k\subset \mathbb{R}^3$ be support surfaces and $\tilde \Sigma_k\subset \mathbb{R}^3$  minimal capillary surfaces supported on $\tilde S_k$ with capillary angle $\sfrac{\pi}2<\theta_k<\pi$. Assume that each $\tilde \Sigma_k$ is weakly stable for the free energy \eqref{free energy}, that $\theta_k\nearrow \pi$, and that $\tilde S_k$ converges locally smoothly to a support surface $\tilde S\subset \mathbb{R}^3$. The starting point for our arguments is the following heuristic: Since $\tilde \Sigma_k$  and $\tilde S_k$ intersect almost tangentially along $\partial \tilde \Sigma_k$, we expect that for $k$ sufficiently large, $\tilde \Sigma_k$ is contained in the normal graph of a nonnegative function $ u_k\in C^\infty(\tilde S_k(\tilde \Sigma_k))$ that satisfies  $u_k=0$  and $|\nabla^{\tilde S_k}u_k|=-\tan\theta_k$ on $\partial \tilde \Sigma_k$ and such that  $u_k=O(1)\sin\theta_k$ in $\tilde S_k(\tilde \Sigma_k)$. Moreover, since $\tilde \Sigma_k$ is minimal, the Jacobi equation suggests that   
	$$
	-\Delta^{\tilde S_k} u_k-|h(\tilde S_k)|^2u_k=-H(\tilde S_k)+o(1)\sin\theta_k 
	$$
	where $\Delta^{\tilde S_k}$ is the nonpositive Laplace-Beltrami operator of $\tilde S_k$ and $h(\tilde S_k)$  the second fundamental form of $\tilde S_k$ with respect to $\nu(\tilde S_k)$. 
If $H(\tilde S_k)=O(1)\,\sin\theta_k$, $(-\cot \theta_k)\,u_k$ should converge to a function $v:\tilde \Omega\to \mathbb{R}$ that satisfies an overdetermined partial differential equation where 
$$
\tilde \Omega=\lim_{k\to\infty} S(\tilde \Sigma_k)
$$ is the tangential limit of $\tilde \Sigma_k$. We expect this overdetermined equation to determine both $\tilde \Omega$ and $v$ and to use this information to conclude on the geometry of $\tilde \Sigma_k$.

\textbf{Step 1.} We make the idea of taking such tangential limits precise. Specifically,  if $0\in \partial \tilde \Sigma_k$ for each $k$, we prove that for $k$ sufficiently large, there is $u_k\in C^\infty(\tilde S_k(\tilde \Sigma_k))$  such that \begin{align} \label{graphicality 0} u_k+|\nabla^{\tilde S_k} u_k|+|\nabla^{\tilde S_k}\nabla^{\tilde S_k}u_k|=O(1)\sin\theta_k\end{align}  and
\begin{align} \label{graphicality}
	 \{y-\nu(\tilde S_k)(y)\,u_k(y):y\in B_1(0)\cap \tilde S_k(\tilde \Sigma_k)\}\subset \tilde \Sigma_k;
	\end{align} see Proposition \ref{curvature estimates}. This result is optimal in that neither the assumption that $H(S_k)=O(1)\,\sin\theta_k$ nor the assumption that $h(S_k)=O(1)$ may be dropped; see Remarks \ref{counter 3} and \ref{counter 4}. Moreover, we show that, passing to a subsequence, $S_k(\Sigma_k)\setminus \partial \Sigma_k$ converges locally in $C^{1,\alpha}$ to a so-called generalized domain $\tilde \Omega$ in $\tilde S$ with generalized boundary $\tilde \partial \Omega$ where $0<\alpha<1$. If 
$$
L(\tilde S)=\lim_{k\to\infty} \frac{H(\tilde S_k)}{\sin \theta_k}
$$
exists, we also obtain that $(-\cot\theta_k)\,u_k$ converges in $C_{loc}^{1,\alpha}$ to a nonnegative function $v:\tilde \Omega\to \mathbb{R}$ satisfying 
\begin{align} \label{PDEintro}
	\begin{dcases}
	-\Delta^{\tilde S} v-|h(\tilde S)|^2\,v=L(\tilde S)\qquad&\text{in $\tilde\Omega\setminus \tilde \partial \Omega$}\\
		v=0&\text{on $\tilde \partial \Omega$, and}\\
		|\nabla^{\tilde S} v|=1&\text{on $\tilde \partial  \Omega;$}
	\end{dcases}
\end{align}
see Proposition \ref{convergence}. The proof of these facts hinges on several  ingredients. First, we prove that graphicality in the sense of \eqref{graphicality 0} and \eqref{graphicality} propagates from a small neighborhood of $\partial \Sigma_k$ to a larger neighborhood of $\partial \Sigma_k$ provided that $k$ is sufficiently large; see Proposition \ref{graphicality 2}. To this end, we show that the distance function to $\tilde S_k$ satisfies a suitable partial differential equation in $\tilde \Sigma_k$  and thus, in view of the stability of $\tilde \Sigma_k$, a Harnack-type inequality; see  Lemmas \ref{distance super harmonic} and \ref{harnack inequality}. This guarantees that for $k$ sufficiently large,  $B_1(0)\cap\tilde \Sigma_k$ is close to $B_1(0)\cap\tilde S_k$ on the scale of $\sin\theta_k$; see Lemma \ref{harnack iteration} and Lemma \ref{harnack inequality}. Using that $\tilde \Sigma_k$ is embedded, we rule out that $\tilde \Sigma_k$ has multiple sheets close to $\tilde S_k$; see Lemma \ref{global lift}. Second, using that $\tilde \Sigma_k$ and $\tilde S_k$ intersect almost tangentially, we prove local bounds on the perimeter of $\tilde S_k(\tilde \Sigma_k)$ and the curvature of $\partial \tilde \Sigma_k$ as well as a one-sided estimate for the injectivity radius of $\partial \tilde \Sigma_k\subset \tilde S_k(\tilde \Sigma_k)$ that depend on the geometry of $\tilde S_k$ and the constant on the right side of \eqref{graphicality 0}; see Proposition \ref{estimates} and Lemma \ref{non collapsing}. An iterative blowup argument then shows that the failure of Proposition \ref{curvature estimates} would imply the existence of a nonaffine solution $v$  of \eqref{PDEintro} with $\tilde S$ an affine plane. By contrast, using the weak stability of $\tilde \Sigma_k$ for the free energy, we show that $v$ is stable for the one-phase Bernoulli problem and thus affine;  see, e.g., \cite{KamburovWang}.

\textbf{Step 2.} We take tangential limits as  described in Step 1 to perform an asymptotic analysis of minimal capillary surfaces on certain support surfaces. We focus on the case where $S$ is a complete embedded minimal surface with finite total curvature that is horizontal and not an affine plane and where  $\Sigma_k\subset \mathbb{R}^3$ are  compact minimal capillary surfaces supported on $S$ with capillary angle $\sfrac{\pi}2<\theta_k<\pi$. We assume that $\theta_k\nearrow \pi$ and that each $\Sigma_k$ is stable for the free energy. Let  $$r(\Sigma_k)=\sup\{|x|:x\in \Sigma_k\}.$$ If $r(\Sigma_k)=O(1)$, then we apply the results from Step 1 with $\tilde S=\tilde S_k=S$ and $\tilde \Sigma_k=\Sigma_k$  and obtain a solution of \eqref{PDEintro} with $L(\tilde S)=0$ and $\tilde \Omega$ compact. This turns out to be in contradiction with the stability of $\Sigma_k$ for the free energy; see Proposition \ref{no bounded components}. Thus, $r(\Sigma_k)\to \infty$. Next, we  apply the results from Step 1 with $\tilde S_k=\sfrac{1}{r(\Sigma_k)}\, S$, $\tilde  \Sigma_k=\sfrac{1}{r(\Sigma_k)}\,\Sigma_k$, and $\tilde S=\mathbb{R}^2\times \{0\}$. This leads to complications as the curvature of $\tilde S_k$ blows up near the origin and the distance between the ends of $\tilde S_k$ becomes small as $k\to \infty$. By an argument based on the stability of $\Sigma_k$, we show that we may virtually ignore the presence of all ends but the uppermost $S^1$, see  Lemma \ref{sandwich}, and obtain a solution of \eqref{PDE} where $\tilde \Omega$ is bounded and has a point singularity at the origin. Using iterative estimates for harmonic functions, we rule out that the origin is an accumulation point of $\tilde \partial \Omega$, see Lemma \ref{extension}, and, by an argument based on M.~B\^ocher's theorem for harmonic functions, conclude that $\Omega$ is a centered round disk and $v$ the fundamental solution of Laplace's equation. In conjunction with an argument based on the Gauss-Bonnet theorem, it follows that $\tilde \Sigma_k$ converges smoothly to a centered round disk; see Lemma \ref{disk} and Proposition \ref{blowdown odd}. In the case where $S$ is a closed support surface with positive mean curvature and $\Sigma_k\subset \mathbb{R}^3$ are compact minimal capillary surfaces with capillary angles $\sfrac{\pi}2<\theta_k<\pi$ that are weakly stable for the free energy and with $\theta_k\nearrow \pi$,  we show instead that, passing to a subsequence, there is $z\in S$ such that $\sfrac{1}{\sin \theta_k}\,(\Sigma_k-z)$ converges smoothly to a centered round disk; see Proposition \ref{blowup}. The resolution of J.~Serrin's overdetermined problem stands in for Proposition \ref{angst} in this setting; see Proposition \ref{positive curvature limit}.

\textbf{Step 3}. We construct minimal capillary surfaces on certain support surfaces as small perturbations of round disks by a linear analysis. In the case where $S$ is a complete embedded minimal surface with finite total curvature that is not an affine plane, we use the inverse function theorem to construct a family $\{\Sigma(\theta):\theta\in(\theta_0,\pi)\}$ of minimal capillary surfaces $\Sigma(\theta)\subset \mathbb{R}^3$ supported on $S$ with capillary angles $\theta$  close to $\pi$, $\partial \Sigma(\theta)\subset S^1$, and $\sfrac{1}{\sin \theta}\,\Sigma(\theta)$  close to a round disk; see Proposition  \ref{IFT existence}.  These surfaces are obtained by perturbations of the round disks that intersect the catenoid at a constant angle. This analysis is complicated by the fact that, for a variation $\{\Sigma(t):-\varepsilon<t<\varepsilon\}$ of a surface $\Sigma\subset \mathbb{R}^3$ that intersects $S$ almost tangentially along $\partial \Sigma$ by surfaces $\Sigma(t)\subset \mathbb{R}^3$ with $\partial \Sigma(t)\subset S$, the magnitude of the tangential motion  is necessarily much larger than the normal speed. We address this issue by careful rescalings and show that  the linearized problem of performing these perturbations is related to a certain Steklov-type problem on the disk, which turns out to be invertible; see Lemma \ref{isomorphism}. The case where $S$ is a compact support surface with positive mean curvature is substantially more involved. Indeed, the application of the inverse function theorem  is obstructed  by the presence of a 2-dimensional kernel in the linearized problem. We overcome this difficulty by  performing a Lyapunov-Schmidt reduction, which shows that minimal capillary surfaces with capillary angle close to $\pi$ exist near nondegenerate local maxima of the mean curvature.

Theorems \ref{THM A} and \ref{THM D} follow from the linear analysis described in Step 3. The asymptotic analysis described in \text{Step 2} shows that sequences of minimal capillary surfaces as  described in \text{Step 2} are captured by the linear analysis described \text{Step 3}. Theorems \ref{THM B}, \ref{THM B 2}, and \ref{THM C} then follow from the local uniqueness of the inverse function theorem.

Finally, we note that some of our arguments to extract tangential limits described in Step 1  draw some inspiration from the work of   O.~Chodosh, N.~Edelen, and C.~Li \cite{chodoshedelenli}, specifically the iterative application of the Harnack inequality in Lemma \ref{harnack iteration} and the iterative blowup argument that leads to Proposition \ref{curvature estimates}; see also \cite{CafarelliFriedman}. In our setting, the fact that $S$  is not flat and that $\Sigma$ does not minimize the free energy leads to substantial new technical and conceptual difficulties that need to be addressed; cp., e.g., Lemma \ref{global lift} and \cite[Lemma 4.7]{chodoshedelenli}.

\subsection*{Acknowledgments.}  This research was  funded in whole or in part by the
Austrian Science Fund (FWF) [\href{https://www.fwf.ac.at/en/research-radar/10.55776/PAT9828924}{10.55776/PAT9828924}, \href{https://www.fwf.ac.at/en/research-radar/10.55776/Y963}{10.55776/Y963}].
	\section{Graphical propagation}
	Recall from Appendix \ref{appendix support surfaces} the definition of a support surface $S\subset \mathbb{R}^3$ with unit normal $\nu(S)$ and of the domain $D(S)\subset \mathbb{R}^3$ bounded by $S$.  Recall that $\Pi_S:\mathbb{R}^3\to S$ is the nearest point projection and that $\operatorname{dist}(\,\cdot\,,S)$ is the signed distance function to $S$ that is positive in $D(S)$. Moreover, recall the definition of the reach $i_S$ of $S$ and our conventions for the second fundamental form $h(S)$, the shape operator $A(S)$, and the mean curvature $H(S)$ of $S$.

	Recall from Appendix \ref{appendix minimal capillary surfaces}  the  definitions of a minimal capillary surface $\Sigma \subset \mathbb{R}^3$ supported on a support surface $S\subset \mathbb{R}^3$  and that of the wetting surface $S(\Sigma)\subset S$ of $\Sigma$. In particular, recall that  $\Sigma\setminus \partial \Sigma \subset D(S)$. We use $\operatorname{dist}_\Sigma$ to denote the intrinsic distance function of $\Sigma$. Given $x\in \Sigma$ and $r>0$, we agree that $B_r^\Sigma(x)=\{z\in \Sigma:\operatorname{dist}_\Sigma(z,x)<r\}$. 
	
	In this section, we consider a support surface $S\subset \mathbb{R}^3$  with unit normal $\nu(S)$
	 and a minimal capillary surface $\Sigma\subset \mathbb{R}^3$ supported on $S$ with capillary angle $\sfrac{\pi}{2}<\theta<\pi$.

Let $r_2>0$ be such that $S_{r_2}(0)$ intersects $S$, $\Sigma$, and $\partial \Sigma$ transversely. 

Let  $\Lambda>1$ be such that 
	\begin{align} \label{geometry bound} 
	1\leq \Lambda\,i_S(y)
	\end{align} 
for all $y\in B_{r_2}(0)\cap S$ and thus, by  \eqref{reach upper bound},  
	\begin{align*}
|h(S)(y)|\leq \sqrt{2}\,\Lambda.
	\end{align*}
	
	Assume that 
	\begin{align} \label{H bound} 
	|H(S)(y)|\leq \Lambda\sin\theta
	\end{align} 
	for all $y\in B_{r_2}(0)\cap S$ and that $$\text{$B_{r_2}(0)\cap (\Sigma\setminus \partial \Sigma)$ is stable.}$$ We emphasize that we do not require $\Sigma$ to be  stable for the free energy in $B_{r_2}(0)$.

Let $\varepsilon>0$ and $\delta>0$  be such that 
\begin{equation} \label{epsilon, delta} 
\begin{aligned}
	&\circ \qquad 2\,\varepsilon<r_2,\\
	&\circ\qquad 4\,\delta<\varepsilon,\text{ and}\\
	&\circ\qquad 4\,\delta\,\Lambda<1.
\end{aligned}
\end{equation} 
Let  $r_0>0$ and $r_1>0$ be such that
\begin{equation} \label{r0, r1}
	\begin{aligned} 
&\circ \qquad\varepsilon<r_1<r_2-\varepsilon\text{ and}\\
&\circ \qquad r_0<r_1-\varepsilon
	\end{aligned} 
\end{equation}
 and such that  both $S_{r_0}(0)$ and  $S_{r_1}(0)$ intersect  $S$, $\Sigma$, and $\partial \Sigma$ transversely. 

Finally, we  assume that
\begin{align} \label{assumption 0}
	\frac{1}{2}\,(-\tan\theta)\operatorname{dist}_\Sigma(x, \partial \Sigma)\leq \operatorname{dist}(x,S)\leq 2\,(-\tan\theta)\operatorname{dist}_\Sigma(x, \partial \Sigma)
\end{align}
for all $x\in B_{r_2}(0)\cap \Sigma$ with $\operatorname{dist}_\Sigma(x, \partial \Sigma)<\delta $. 

In this section, we prove that, provided that $\theta$ is sufficiently close to $\pi$, there is a function with small $C^1$-norm on $B_{r_0}(0)\cap S(\Sigma)$  whose normal graph is contained in $\Sigma$.
\begin{lem} \label{lift} 
	Let $\beta:[0, T]\to S$ be a smooth curve such that $\beta(0)\in \partial \Sigma$ and $\beta(t)\in S(\Sigma)\setminus \partial \Sigma$ for all sufficiently small $t\in (0, T]$. There is $T^{\max}\in(0, T]$ and  a unique smooth curve $\alpha:[0, T^{\max})\to \Sigma$ such that 
	\begin{equation} \label{lift properties} 
	\begin{aligned}
		&\circ \qquad 	\Pi_S(\alpha(t))=\beta(t),\\
		&\circ \qquad  \operatorname{dist}(\alpha(t),S)<i_S(\beta(t)),\text{ and}\\
		&\circ \qquad \nu(\Sigma)(\alpha(t))\cdot \nu(S)(\beta(t))< 0
	\end{aligned} 
	\end{equation} 
for all $t\in[0,T^{\max})$ and either 
	\begin{align}
	&	\circ\qquad \text{$\limsup_{t\nearrow T^{\max}}\operatorname{dist}(\alpha(t),S)=i_S(\beta(T^{\max}))$ or  $\limsup_{t\nearrow T^{\max}}\, \nu(\Sigma)(\alpha(t))\cdot \nu(S)(\beta(t))=0$} \label{alternative A}
\end{align}
or  one of the following two alternatives holds.
\begin{align} 
	&	\circ\qquad \text{$T= T^{\max}$ and $\alpha$ extends to a smooth curve $\alpha:[0,T]\to \Sigma$.} \label{alternative B}\\
	&	\circ\qquad \text{$T^{\max}< T$ and $\alpha$ extends to a smooth curve $\alpha:[0,T^{\max}]\to \Sigma$ with $\alpha(T^{\max})\in \partial \Sigma$.} \label{alternative C}
	\end{align}

\end{lem}
\begin{proof}
Note that $\nu(\Sigma)\cdot \nu(S)< 0$ on $\partial \Sigma$. By \eqref{first projection},    
	$$
	(D\Pi_S)(\beta(0))(\upsilon)=\upsilon-  (\upsilon\cdot \nu(S)(\beta(0)))\,\nu(S)(\beta(0))\neq 0	$$
	for all $\upsilon\in T_{\beta(0)}\Sigma$ with $\upsilon\neq 0$. 
	By a standard argument based on the inverse function theorem, using that $\mu(\Sigma)\cdot \mu(S(\Sigma))>0$ on $\partial \Sigma$, there is $T^{\max}\in(0, T]$ maximal and  a unique smooth curve $\alpha:[0, T^{\max})\to \Sigma$ that satisfies \eqref{lift properties} for all $t\in[0,T^{\max})$. Assume that \eqref{alternative A} does not hold. Using \eqref{first projection} and that $\beta$ is a smooth curve, we see that 
	$$
	\sup\{|\dot \alpha(t)|:t\in(0,T^{\max})\}<\infty.
	$$
	It follows that the limit $\alpha^{\max}=\lim_{t\nearrow T^{\max}}\alpha(t)$ exists. Moreover, there holds $\Pi_S(\alpha^{\max})=\beta(T^{\max})$,  $\operatorname{dist}(\alpha^{\max},S)<i_S(\beta(T^{\max}))$, and $\nu(\Sigma)(\alpha^{\max})\cdot \nu(S(\beta(T^{\max})))<0$. Using \eqref{first projection} once again, we see that 
	$
	(D\Pi_S)(\alpha^{\max})(\upsilon)\neq 0
	$
	for all $\upsilon\in T_{\alpha^{\max}}\Sigma$ with $\upsilon\neq 0$.
	In view of the inverse function theorem,  either \eqref{alternative B} or \eqref{alternative C} holds.
	
	This completes the proof of the lemma.
\end{proof}
Given $\alpha$ and $\beta$ as in Lemma \ref{lift}, we call $\alpha$ the lift of $\beta$.
\begin{rema} \label{uniqueness of lifts}
 Let $\beta_1,\,\beta_2:[0, T]\to S$  be two smooth curves as in Lemma \ref{lift} with respective lifts $\alpha_1:[0,T^{\max}_1)\to \Sigma$ and $\alpha_2:[0,T^{\max}_2)\to \Sigma$. Assume that there is $0<T_0<\min\{T^{\max}_1,T^{\max}_2\}$ such that $\beta_1(t)=\beta_2(t)$ for all $t\in [T_0,T]$ and $\alpha_1(T_0)=\alpha_2(T_0)$. 	The proof of Lemma \ref{lift} shows that $T^{\max}_1=T^{\max}_2$ and $\alpha_1(t)=\alpha_2(t)$ for all $t\in[T_0,T^{\max}_1)$.   
\end{rema}

The proof of the following lemma is in part inspired by that of \cite[Lemma 4.8]{chodoshedelenli}; see also \cite{Simon}.

	\begin{lem}\label{harnack iteration}
Let   $L>1$.  There are $\sfrac{\pi}2<\theta_0<\pi$ and $c>1$ depending only on $\Lambda$, $\varepsilon$, $\delta$,  and $L$ such that the following holds. Let $\theta_0<\theta<\pi$ and
  $\alpha:[0,T]\to B_{r_1}(0)\cap\Sigma$ be a smooth curve of length at most $L$ with $\alpha(0)\in \partial \Sigma$. For all $t\in[0,T]$, there holds $$2\operatorname{dist}(\alpha(t),S)<\min\{\varepsilon,\,i_S(\Pi_S(\alpha(t)))\}$$ and 
\begin{align} \label{conclusion} 
\operatorname{dist}(\alpha(t),S)\leq c\,(-\tan\theta)\,s(t)
\end{align} 
where $s(t)=\min\{\operatorname{dist}_\Sigma(\alpha(t),\partial \Sigma),\sfrac{\delta}2\}$.
Moreover, for each $t\in[0,T]$, there is  $u\in C^\infty(S)$ positive such that
\begin{align} \label{graph concl 1}
\left\{z\in \Sigma:2\operatorname{dist}_\Sigma(\alpha(t),z)\leq s(t)\right\}\subset \{\Phi_S(u)(y):y\in S\}
\end{align} and
\begin{align} \label{graph concl2}
	|\nabla^Su(y)|+s(t)\,|\nabla^S\nabla^Su(y)|\leq c\,(-\tan\theta)
\end{align}
for all $y\in S$ with $\Phi_S(u)(y)\in \left\{z\in \Sigma:4\operatorname{dist}_\Sigma(\alpha(t),z)\leq s(t)\right\}$.
	\end{lem} 
	\begin{proof}
By \eqref{geometry bound} and \eqref{epsilon, delta},		 for all $y\in B_{r_2}(0)\cap S$,
		\begin{align} \label{delta 0}
			4\,\delta < \min\{\varepsilon,\,i_S(y)\}.
		\end{align}
		
		Let  $c_1>1$ be the constant from Lemma \ref{harnack inequality} applied with $A=\Lambda\,\delta$. 	Let  $c_2>1$ be the constant from Lemma \ref{slab lemma} applied with $B=\sfrac{\delta\,\Lambda}{4}$. Let $0<\varepsilon_2<1$ be such that Lemma \ref{slab lemma} holds for all $\varepsilon\in(0,\varepsilon_2)$ for this choice of $B$. Let $c_3=(1+\sfrac{\delta\,\Lambda}{4})\,c_1$.
		
		We choose  $\sfrac{\pi}2<\theta_0<\pi$ such that 
				\begin{align}  
			\label{theta small 0}
			4\,(-\tan\theta_0)\,(2+c_3^{\sfrac{4\,L}{\delta}})<\varepsilon_2.		
		\end{align}	
		
		It suffices to prove the assertion of the lemma for $t=T$. We may and will assume that $\alpha(T)\notin \partial \Sigma$.
			
		Assume first that $2\operatorname{dist}_\Sigma(\alpha(T),\partial \Sigma)< \delta$. 
		
		By \eqref{assumption 0}, \eqref{conclusion} holds with $c=2$.  Let $z\in \Sigma$ be such that
		$$
		\operatorname{dist}_\Sigma(\alpha(T),z)<\operatorname{dist}_\Sigma(\alpha(T),\partial \Sigma).
		$$ 
		 Note that $z\notin \partial \Sigma$. By \eqref{r0, r1} and \eqref{delta 0}, \begin{align} \label{z} 2\,|z|\leq 2\,|\alpha(T)|+2\,|\alpha(T)-z|<2\,r_1+\delta<2\, r_2.\end{align}
Thus, $z\in B_{r_2}(0)\cap \Sigma$. Moreover,	
		 $$
		 \operatorname{dist}_\Sigma(z,\partial \Sigma)\leq  \operatorname{dist}_\Sigma(z,\alpha(T))+ \operatorname{dist}_\Sigma(\alpha(T),\partial \Sigma)<2\operatorname{dist}_\Sigma(\alpha(T),\partial \Sigma)<\delta.
		 $$
		By \eqref{assumption 0} and \eqref{theta small 0}, 
		\begin{align} \label{close to boundary 0} 
		\operatorname{dist}(z,S)\leq 2\,(-\tan\theta)\operatorname{dist}_\Sigma(z,\partial \Sigma)<4\,(-\tan\theta)\operatorname{dist}_\Sigma(\alpha(T),\partial \Sigma)<2\,\delta.
		\end{align} 
		By \eqref{r0, r1}, \eqref{delta 0},  \eqref{theta small 0}, \eqref{z}, and \eqref{close to boundary 0}, 
		$$
		|\Pi_S(z)|\leq |z|+\operatorname{dist}(z,S)\leq r_1+3\,\delta<r_1+\varepsilon<r_2.
		$$
		Thus, $\Pi_S(z)\in B_{r_2}(0)\cap S$.
	By \eqref{delta 0}  and \eqref{close to boundary 0},  $2\operatorname{dist}(z,S)< \min\{\varepsilon,\,i_S(\Pi_S(z))\}$. By \eqref{geometry bound},
		$$
		\frac12\operatorname{dist}_\Sigma(\alpha(T),\partial \Sigma)<\frac{\delta}4=\frac{B}{\Lambda}\leq B\,i_S(y)
		$$
		for all $y\in B_{r_2}(0)\cap S$. By \eqref{theta small 0} and \eqref{close to boundary 0},  $$2\operatorname{dist}(z,S)\leq\varepsilon_2\operatorname{dist}_\Sigma(\alpha(T),\partial \Sigma).$$
		By Lemma \ref{slab lemma}, there is a  $u\in C^\infty(S)$ positive that satisfies \eqref{graph concl 1} and \eqref{graph concl2} with $c=8\,c_2$.

		Assume now that $2\operatorname{dist}_\Sigma(\alpha(T),\partial \Sigma)\geq \delta$.

By assumption,  the length of $\alpha:[0,T]\to B_{r_1}(0)\cap \Sigma$ is at most $L$. By \eqref{r0, r1} and \eqref{delta 0}, there are an integer $k\geq 1$ with 
		\begin{align} \label{k bound 0}
			(k+1)\,\delta\leq  4\,L
		\end{align}
		and  $0<t_1<\ldots <t_{k+1}=T$ such that 
		\begin{align} \label{p -1}
		\sfrac\delta4<	\operatorname{dist}_\Sigma(\alpha(t_1), \partial \Sigma)<\sfrac\delta2 		
		\end{align} 
	 and, for each $1\leq \ell \leq k$, 
		\begin{align}
			&\circ\qquad  \text{$\alpha(t_{\ell})\in B^\Sigma_{\sfrac{\delta}{4}}(\alpha(t_{\ell+1}))$,} \label{connect 0} \\
			&\circ \qquad \text{$B^\Sigma_{\sfrac{\delta}{2}}{(\alpha(t_{\ell+1}))}\cap \partial \Sigma=\emptyset$, and} \label{sufficiently large ball 1 0}	  \\
			&\circ \qquad \text{$B^\Sigma_{4\,\delta}{(\alpha(t_{\ell+1}))}\subset B_{r_2}(0)\cap\Sigma$.}	\label{sufficiently large ball 2 0}	 		
		\end{align}
		We claim that 
		\begin{align}
			&\circ \qquad \text{$2\operatorname{dist}(z,S)<\min\{\varepsilon,\,i_S(\Pi_S(z))\}$ for all  $z\in B^\Sigma_{\sfrac{\delta}{2}}(\alpha(t_\ell))$ and} \label{to show 1 0}\\
			&\circ\qquad  
			\operatorname{dist}(z,S)\leq c_3^{\ell-1}\,\delta\,(-\tan\theta)  \text{ for all  $z\in B^\Sigma_{\sfrac{\delta}{4}}{(\alpha(t_\ell))}$} \label{to show 2 0} 
		\end{align}			
		for each $2\leq \ell \leq k+1$.
		To see this, we proceed by induction. 
		
		First, if $\ell=2$ and $z\in B^\Sigma_{\sfrac{\delta}{2}}(\alpha(t_2))$, by  \eqref{assumption 0},  \eqref{theta small 0}, \eqref{p -1},  and \eqref{connect 0}, 
		$$
		\operatorname{dist}(z,S)\leq |z-\alpha(t_2)|+|\alpha(t_2)-\alpha(t_1)|+\operatorname{dist}(\alpha(t_1),S)<  	\frac{\delta}{2}+\frac{\delta}{4}	
		+(-\tan\theta)\,\delta<2\,\delta.
		$$
		In conjunction with \eqref{r0, r1} and \eqref{delta 0}, we have 
		$$
		|\Pi_S(z)|\leq|z|+\operatorname{dist}(z,S)\leq r_1+3\,\delta<r_1+\varepsilon<r_2.
		$$
		Using \eqref{delta 0} again, we obtain
		$$ 
			2\operatorname{dist}(z,S)< \min\{\varepsilon,\,i_S(\Pi_S(z))\}.
			$$
		In view of this, \eqref{H bound}, \eqref{connect 0}, \eqref{sufficiently large ball 1 0}, and \eqref{sufficiently large ball 2 0}, we may use Lemma \ref{distance super harmonic} and Lemma \ref{harnack inequality} with $f=\operatorname{dist}(\,\cdot\,,S)$ and $r=\sfrac{\delta}{4}$ to conclude that
			\begin{align*} 
			\operatorname{dist}(z,S)&\leq c_1\operatorname{dist}(\alpha(t_1),S)+\frac{\delta^2}{16}\,c_1\,\Lambda\,\sin\theta
			\end{align*} 
		 for all $z\in B^\Sigma_{\sfrac{\delta}{4}}(\alpha(t_2))$.  By \eqref{assumption 0} and \eqref{p -1}, 
		\begin{align*} 
	\operatorname{dist}(z,S)\leq c_1\,\delta\,(-\tan\theta)+\frac{\delta^2}{16}\,c_1\,\Lambda\,\sin\theta
	< c_3\,\delta\,(-\tan\theta).
		\end{align*} 
		Assume now that $2\leq \ell\leq k$ and that both \eqref{to show 1 0} and \eqref{to show 2 0} have already been shown for $\ell$. We obtain, for  $z\in B^\Sigma_{\sfrac{\delta}{2}}(\alpha(t_{\ell+1}))$, using  \eqref{delta 0}, \eqref{theta small 0}, \eqref{k bound 0}, \eqref{connect 0}, and \eqref{to show 2 0},  
		$$
		\operatorname{dist}(z,S)\leq |z-\alpha(t_{\ell+1})|+|\alpha(t_{\ell+1})-\alpha(t_{\ell})|+\operatorname{dist}(\alpha(t_{\ell}),S)\leq \frac\delta2+\frac\delta 4+c_3^{\ell-1}\,\delta\,(-\tan\theta)\leq 2\,\delta 
		$$
		As above, this implies that $|\Pi_S(z)|<r_2$ and $2\operatorname{dist}(z,S)\leq \min\{\varepsilon,\,i_S(\Pi_S(z))\}$.  
		We may use Lemma \ref{distance super harmonic}, Lemma \ref{harnack inequality}, and \eqref{to show 1 0} to show that \eqref{to show 2 0} also holds for $\ell+1$. 
		
		Specifying \eqref{to show 2 0} to the case where $\ell=k+1$ and using \eqref{k bound 0}, we obtain \eqref{conclusion} with $c=2\,c_3^{\sfrac{4\,L}{\delta}}$. Likewise, specifying \eqref{to show 2 0} to the case where $\ell=k+1$, using Lemma \ref{slab lemma}, \eqref{k bound 0}, and \eqref{theta small 0}, we see that there is  $u\in C^\infty(S)$ positive satisfying \eqref{graph concl 1} and \eqref{graph concl2} with $c=8\,c_2\,c_3^{\sfrac{4\,L}{\delta}}$.

		This completes the proof of the lemma.

	\end{proof}
	\begin{coro} \label{normal coro}
	Assumptions as in Lemma \ref{harnack iteration}. For all $t\in(0,T]$, there holds
			\begin{align*} 
			1+\nu(S)(\Pi_S(\alpha(t)))\cdot \nu(\Sigma)(\alpha(t))\leq c\,(\tan\theta)^2.
		\end{align*} 
	\end{coro}
	\begin{proof}
	This follows from	Lemma \ref{harnack iteration}, Lemma \ref{graph},  and Taylor's theorem, using that $\nu(S)\cdot \nu(\Sigma)<0$ on $\partial \Sigma$.
		
	\end{proof}
	\begin{lem} \label{bounded length lift}
		Let   $L>1$.  There is $\sfrac{\pi}2<\theta_0<\pi$  depending only on $\Lambda$, $\delta$, $\varepsilon$, and $L$ such that the following holds. Assume that $\theta_0<\theta<\pi$. 		Let $\beta:[0, T]\to B_{r_0}(0)\cap S$ be a smooth curve of length less than $\sfrac{L}{2}$ such that $\beta(0)\in \partial \Sigma$ and $\beta(t)\in S(\Sigma)\setminus \partial \Sigma$ for all sufficiently small $t\in (0, T)$. Let $\alpha:[0, T^{\max})\to  \Sigma$  be the lift of $\beta$. For all $t\in[0,T^{\max})$, there holds 
		\begin{align*} 
			&\circ \qquad  \alpha(t)\in B_{r_1}(0)\cap \Sigma,
			\\&\circ \qquad 2\,\operatorname{dist}(\alpha(t),S)<i_S(\beta(t)),\text{ and}\\
			&\circ \qquad \nu(\Sigma)(\alpha(t))\cdot \nu(S)(\beta(t))\leq -\sfrac12.
		\end{align*} 
		 In particular, \eqref{alternative A} does not occur.
	\end{lem}
	\begin{proof} 
		Let $\sfrac{\pi}2<\theta_0<\pi$ and $c>1$ be the constants from Lemma \ref{harnack iteration} applied with $\varepsilon$, $\delta$, and $L$. 
		Let $0<T_0\leq  T^{\max}$ be largest such that $\alpha|_{[0,T_0)}$ has length less than $L$ and such that  
		$$\nu(S)(\beta(t))\cdot \nu(\Sigma)(\alpha(t))\leq -\sfrac12\qquad\text{and}\qquad  2\operatorname{dist}(\alpha(t),S)\leq \min\{i_S(\beta(t)),\varepsilon\}$$ for all $t\in[0,T_0)$. 
		 By \eqref{r0, r1},
		$$
		|\alpha(t)|\leq |\beta(t)|+\operatorname{dist}(\alpha(t),S)< r_0+\varepsilon<r_1.
		$$  
for all $t\in[0,T_0)$.		By Lemma \ref{harnack iteration}, 
		\begin{align} \label{prelim alpha}
		2\operatorname{dist}(\alpha(t),S)< \min\{i_S(\beta(t)),\varepsilon\}
		\end{align}  for all $t\in(0,T_0)$. By Corollary \ref{normal coro}, increasing $\theta_0$ if necessary, there holds 
		$$
		2\,\nu(S)(\beta(t))\cdot \nu(\Sigma)(\alpha(t))< -1
		$$
for all $t\in(0,T_0)$. Using that $\alpha(t)=\Phi_S(u)(\beta(t))$ where $u\in C^\infty(S)$ is as in Lemma \ref{harnack iteration}, we have
		 $$
		 \dot\alpha(t)=\dot\beta(t)-\dot\beta(t)\cdot \nabla^Su(\beta(t))\,\nu(S)(\beta(t))-\operatorname{dist}(\alpha(t),S)\,A(S)(\beta(t))\dot\beta(t).
		 $$
		 Using 		 \eqref{graph concl2}, \eqref{prelim alpha}, and \eqref{reach upper bound} and increasing $\theta_0$, if necessary, it follows that
		$$
		2\,|\dot \alpha(t)|\leq (3+2\,c\,(-\tan\theta))\,|\dot \beta(t)|<4.
		$$
		Thus,  $T_0=T^{\max}$.
		
		This completes the proof of the lemma.
	\end{proof}
		\begin{lem} \label{global lift}
			There is $\sfrac{\pi}2<\theta_0<\pi$  depending only on $\Lambda$, $\varepsilon$, and $\delta$ such that the following hold.	Assume that $\theta_0<\theta<\pi$.
			
	Let $\beta_1:[0, T_1]\to B_{r_0}(0)\cap S(\Sigma)$ be a smooth curve such that $\beta_1(0)\in \partial \Sigma$ and $\beta_1(t)\in S(\Sigma)\setminus \partial \Sigma$ for all  $t\in (0, T_1)$. Let $\alpha_1:[0,T^{\max}_1)\to \Sigma$ be the lift of $\beta_1$. Then \eqref{alternative B} holds. In particular,  $T^{\max}_1= T_1$. Moreover, if $\beta_1(T_1)\in \partial \Sigma$, then $\alpha_1(T_1)=\beta_1(T_1)$. 
	
	Let $\beta_2:[0,T_2]\to  B_{r_0}(0)\cap S(\Sigma)$ be a second smooth curve such that $\beta_2(0)\in \partial \Sigma$ and $\beta_2(t)\in S(\Sigma)\setminus \partial \Sigma$ for all  $t\in (0,T_2)$. Let $\alpha_2:[0,T_2]\to \Sigma$ be the lift of $\beta_2$. Assume that  $\beta_2(T_2)=\beta_1(T_1)$.  Then $\alpha_1(T_1)=\alpha_2(T_2)$.
	\end{lem}
	\begin{proof}
		Let $\bar L$ be the constant from Lemma \ref{escape} applied with the radii $r_2$ and $r_0$. Let $\sfrac{\pi}2<\theta_0<\pi$ be the constant from Lemma \ref{bounded length lift} with $L=8\, \bar L$. 
		
		For the following argument,  note that Lemma \ref{bounded length lift} also holds  for continuous  curves that are piecewise smooth, with the same proof.
		
		To prove the lemma, we argue by induction over the smallest integer $k\geq 1$  such that the lengths of  $\beta_1$ and $\beta_2$ are both at most $(k+1)\,\bar L$. 
		
	First,   assume that the lengths of $\beta_1$ and $\beta_2$ are both less than $2\,\bar L= \sfrac{L}4$. 
	
 By Lemma \ref{bounded length lift}, $T^{\max}_1=T_1$ and \eqref{alternative B} holds.  Suppose, for a contradiction, that $\beta_1( T_1)\in \partial\Sigma$ but $\alpha_1(T_1)\notin\partial \Sigma$. Let $ \beta_3:[0, T_1]\to B_{r_0}(0)\cap S(\Sigma)$ be given by $ \beta_3(t)=\beta_1(T_1-t)$. Note that $\beta_3$ has the same length as $\beta_1$ so that, by Lemma \ref{bounded length lift}, \eqref{alternative B} holds for the lift $ \alpha_3:[0,T_1]\to \Sigma$ of $\beta_3$.  
  Since $\Sigma$ is embedded and \eqref{alternative A} holds neither for $\alpha_1$ nor for $\alpha_3$, we have $0\leq \operatorname{dist}(\alpha_3( T_1-t),S)<\operatorname{dist}(\alpha_1(t),S)$ for all $t\in[0, T_1]$. Since $\alpha_1(0)\in S$, this is a contradiction.
		
			Likewise, by Lemma \ref{bounded length lift}, $T^{\max}_2=T_2$.  Suppose, for a contradiction, that $\operatorname{dist}(\alpha_2(T_2),S)<\operatorname{dist}(\alpha_1(T_1),S)$. Note that the curve $\beta_4:[0, T_1+ T_2]\to B_{r_0}(0)\cap S(\Sigma)$ given by 
			$$\beta_4(t)=\begin{dcases} \beta_2(t)\qquad&\text{if $t\in[0, T_2]$ and}\\ \beta_1(T_1+T_2-t)&\text{if $t\in[T_2,T_1+T_2]$}
			\end{dcases}
			$$ is continuous and piecewise smooth. The  length of $\beta_4$ is less than $\sfrac{L}{2}$. By Lemma \ref{bounded length lift}, \eqref{alternative B} holds for the lift $ \alpha_4:[0,T_1+T_2]\to \Sigma$  of $ \beta_4$.  By the uniqueness of lifts, $\alpha_4(T_2)=\alpha_2(T_2)$. Since $\Sigma$ is embedded and \eqref{alternative A} holds neither for $\alpha_1$ nor for $\alpha_4$, we have $0\leq \operatorname{dist}(\alpha_4(T_2+t),S)<\operatorname{dist}(\alpha_1(T_1-t),S)$ for all $t\in[0,T_1]$. Since $\alpha_1(0)\in S$, this is a contradiction.
			
		Next, assume that $k\geq 2$ and that the assertion of the lemma has already been proven for all curves of length at most $k\,\bar L$. 
		
	By Lemma \ref{escape},
	 there is a smooth curve $ \tilde \beta_1:[0,T_1]\to B_{r_0}(0)\cap S(\Sigma)$ of length less than $(k-1)\,\bar L$ with $\tilde \beta_1(t)\notin \partial \Sigma$ for all $t\in(0,T_1)$  and $T_0\in(0, T_1)$  such that $ \tilde \beta_1(t)=\beta_1(t)$ for all $t\in[T_0,T_1]$ and $\tilde \beta_1|_{[0,T_0]}$ has length less than $2\,\bar L$. By the induction hypothesis, the lift $\tilde \alpha_1:[0,T_1]\to \Sigma$ of $\tilde \beta_1$ exists. Using the induction hypothesis again, we have	
	 $ \alpha_1(T_0)=\tilde \alpha_1(T_0).$ By Remark \ref{uniqueness of lifts}, $T^{\max}_1= T_1$ and $\alpha_1(T_1)=\tilde \alpha_1(T_1)$. If $\beta_1(T_1)\in \partial \Sigma,$ then  $\tilde \alpha_1(T_1)\in \partial \Sigma$ by the induction hypothesis. In particular, $\alpha_1(T_1)\in \partial \Sigma$. 
	 
	 A similar argument shows $T^{\max}_2=T_2$ and that $\alpha_1(T_1)=\alpha_2(T_2)$ if $\beta_1(T_1)=\beta_2(T_2)$.
	 
	 This completes the proof of the lemma. 
	\end{proof}

			\begin{prop} \label{graphicality 2}
		There is $\sfrac{\pi}2<\theta_0<\pi$  depending only on $\Lambda$, $\delta$, and $\varepsilon$ such that the following hold.
		Assume that $\theta_0<\theta<\pi$ and  $0\in \partial \Sigma$. There is  $ u\in C^{0,1}(S)$ nonnegative with the following properties.
			\begin{align}
				\label{u 0.5}	&\circ\qquad \{\Phi_S(u)(y):y\in  B_{r_0}(0)\cap S(\Sigma)\}\subset \Sigma\\ 
				\label{u 1.5}	&\circ\qquad \text{$u=0$ in $S\setminus S(\Sigma)$} \\ 
				\label{u 2}	&\circ\qquad \text{$ u\in C^\infty(B_{r_0}(0)\cap S(\Sigma))$}\\ 
				\label{u 3}	&\circ\qquad \text{$|\nabla^S u|\leq c\,(-\tan\theta)$ in $ B_{r_0}(0)\cap  S(\Sigma)$}\\ 
				\label{u 5}	&\circ\qquad \text{$|\nabla^S u|=-\tan\theta$ on $B_{r_0}(0)\cap \partial \Sigma $}
			\end{align}
			Moreover, there holds
			\begin{align}
				\label{u 1.75} 
				u(y)\leq c\,(-\tan\theta)\operatorname{dist}_\Sigma(\Phi_S(u)(y),\partial \Sigma)
			\end{align}
			for all $y \in B_{r_0}(0)\cap \Sigma$.
		\end{prop}
		\begin{proof}
Let $y\in B_{r_0}(0)\cap S$.	If $y\notin S(\Sigma)\setminus \partial \Sigma$, let $u(y)=0$. If $y\in S(\Sigma)\setminus \partial \Sigma$, since $0\in \partial \Sigma$, there is a smooth curve $\beta:[0, T]\to B_{r_0}(0)\cap S(\Sigma)$ such that $\beta(0)\in \partial \Sigma$, $\beta(T)=y$, and $\beta(t)\notin \partial \Sigma$ for all $t\in(0,T)$.  Let $\alpha:[0,T^{\max}]\to \Sigma$ be the lift of $\beta$. By  Lemma \ref{global lift}, $T^{\max}= T$ provided that $\sfrac{\pi}2<\theta_0<\pi$ is sufficiently close to $\pi$.  By  Lemma \ref{global lift}, $\alpha(T)$ is independent of the choice of $\beta$. We define $u(y)=\operatorname{dist}(\alpha(T),S)$. By Lemma \ref{escape}, we may assume that the length of $\beta$ is less than $\bar L$ where $\bar L$ is the constant from Lemma \ref{escape} applied with the radii $r_2$ and $r_0$. By  Lemma \ref{bounded length lift}, the length of $\alpha$ is at most $2\,\bar L$ and the image of $\alpha$ is contained in $B_{r_1}(0)\cap \Sigma$. Moreover, there holds $u(y)<i_S(y)$ and $\nu(\Sigma)(u(y))\cdot \nu(S)<0$. Applying the inverse function theorem as  in the proof of Lemma \ref{lift}, \eqref{u 0.5}, \eqref{u 1.5}, and \eqref{u 2} follow.  Applying Lemma \ref{harnack iteration} with $L=2\,\bar L$, we obtain \eqref{u 3} and \eqref{u 1.75} provided that $\sfrac{\pi}2<\theta_0<\pi$ is sufficiently close to $\pi$. Since $\Sigma$ is a minimal capillary surface supported on $S$ with capillary angle $\theta$, we obtain \eqref{u 5}.

This completes the proof of the proposition. 
		\end{proof}

		\section{Stable minimal capillary surfaces close to a support surface}

	Recall from Appendix \ref{appendix support surfaces} the definition of a support surface $S\subset \mathbb{R}^3$ with unit normal $\nu(S)$ and of the domain $D(S)\subset \mathbb{R}^3$ bounded by $S$.  Recall that $\operatorname{dist}(\,\cdot\,,S)$ is the signed distance function to $S$ that is positive in $D(S)$. Recall  the definition of the reach $i_S$ of $S$ and our conventions for the second fundamental form $h(S)$, the shape operator $A(S)$, and the mean curvature $H(S)$ of $S$. Given $\psi \in C^\infty(S)$, recall the definition \eqref{Phi S} of $\Phi_S(\psi)$.

Recall from Appendix \ref{appendix minimal capillary surfaces}  the  definition of a minimal capillary surface $\Sigma \subset \mathbb{R}^3$ supported on a support surface $S\subset \mathbb{R}^3$  and that of its  wetting surface $S(\Sigma)\subset S$. In particular, recall that  $\Sigma\setminus \partial \Sigma \subset D(S)$.  Moreover, recall the concepts of  stability and weak stability for the free energy  of a minimal capillary surface.
 
Given $U\subset \mathbb{R}^3$ open, recall from Appendix \ref{appendix: generalized domain} the definition of a generalized domain $\Omega$ in $U\cap S$ and of its associated Riemannian 2-manifold $(\tilde \Omega,g)$. Given $L(S)\in C^\infty(S)$, recall the concepts of  stability and weak stability of a generalized domain  with respect to $L(S)$.  Moreover, recall the definitions of convergence of a sequence of domains to a generalized domain and of convergence of a sequence of functions defined on the domains of such a sequence to a function defined on the limiting generalized domain.

	In this section,  we consider   support surfaces $S_k\subset \mathbb{R}^3$, minimal capillary surfaces $\Sigma_k$ supported on $S_k$ with capillary angle $\sfrac{\pi}2<\theta_k<\pi$, and  $u_k\in C^{0,1}(S_k)$ nonnegative, $k\geq 1$.

	 We assume that $\theta_k\nearrow \pi$.
	 
	 Let $U_\ell\subset \mathbb{R}^3$, $\ell \geq 1$, be smooth  open sets   such that $U_\ell\Subset U_{\ell+1}$ and   $\partial U_\ell$ intersects $S_k$, $\Sigma_k$, and $\partial \Sigma_k$ transversely for all $k\geq 1$. In our applications, the sets $U_\ell$ will either be concentric balls of increasing radii or concentric spherical shells of increasing outer radii and decreasing inner radii. 
	 
	 We also assume that there are a support surface $S$ with $H(S)=0$ and $\psi_k\in C^\infty(S)$ such that, for each fixed $\ell\geq1$,  the following hold.
	 \begin{align}
	 	\label{S assumption 3}&\circ\qquad \sup\left\{|H(S_k)(y)|:y\in U_\ell\cap S_k\right\}=O(1)\sin\theta_k\text{ as $k\to\infty$.}\\
	 	\label{S assumption 1} &\circ \qquad U_\ell\cap S_k \subset \{\Phi_S(\psi_k)(y):y\in U_{\ell+1}\cap S\}\text{ for all $k$ sufficiently large}.\\
	 	\label{S assumption 2}&\circ \qquad \sup\{|\psi_k(y)|+|\nabla^S\psi_k(y)|+|\nabla^S\nabla^S\psi_k(y)|+|\nabla^S\nabla^S\nabla^S\psi_k(y)|:y\in U_{\ell+1}\cap S\}\\&\qquad\qquad\notag=o(1)\text{ as $k\to\infty$.} 
	 \end{align}
	 In particular, there are  $\delta_\ell>0$ such that, for all $k$ sufficiently large and $y\in U_\ell \cap S_k$,
	 \begin{align}
i_{S_k}(y)\geq \delta_\ell	\qquad\text{and thus}\qquad  \delta_\ell\,|h(S_k)(y)|\leq \sqrt{2}.
	 \end{align}

	We also assume that there is $L(S)\in C^1(S)$ such that, for each fixed $\ell \geq 1$,
	 \begin{equation}\label{S assumption 4}
	  \begin{aligned}
	 & 	 \sup\left\{|H(S_k)(\Phi_S(\psi_k)(y))-(-\tan\theta_k)\,L(S)(y)|:y\in U_{\ell+1}\cap S\right\}=o(1)\sin\theta_k\text{ and}
	 \\
	 &\sup\left\{|\nabla ^S( H(S_k)\circ\Phi_S(\psi_k))(y)-(-\tan\theta_k)\,\nabla^S L(S)(y)|:y\in U_{\ell+1}\cap S\right\}
 =o(1)\sin\theta_k
	  \end{aligned}
	  \end{equation} 
	  as $k\to\infty$.

	Finally, we assume that, for each fixed $\ell \geq 1$,  the following hold for all $k$  sufficiently large.
	\begin{align}
		\label{u assumption 1}	&\circ\qquad  \text{$\{\Phi_{S_k}(u_k)(y):y\in U_{\ell}\cap  S_k(\Sigma_k)\}\subset U_{\ell+1}\cap \Sigma_k$}\hspace{3cm}\\ 		
			\label{u assumption 2}	&\circ\qquad \text{$u_k=0$ in $S_k\setminus S_k(\Sigma_k)$} \\
			&\circ\qquad \text{$ u_k\in C^\infty(U_{\ell}\cap S_k(\Sigma_k))$}
				\label{u assumption 5}	 \text{ and $u_k=O(1)\sin\theta_k$, $\nabla^{S_k} u_k=O(1)\sin\theta_k$, and }
				\\\notag &\qquad \quad \text{$\nabla^{S_k}\nabla^{S_k} u_k=O(1)\sin\theta_k$ in $U_{\ell}\cap  S_k(\Sigma_k)$}\\ 
		\label{u assumption 6}	&\circ\qquad \text{$|\nabla^{S_k} u_k|=-\tan\theta_k$ on $U_{\ell}\cap \partial \Sigma_k $}
	\end{align}

In this section, we study the geometry of the wetting surface $S_k(\Sigma_k)$ and the behavior of $u_k$ as $k\to\infty$. 

Let $$\Sigma_k^\ell=\left\{\Phi_{S_k}(u_k)(y):y\in U_{\ell}\cap  S_k(\Sigma_k)\right\}.$$

	\begin{prop} Fix $\ell \geq 1$. \label{estimates} 
		There holds, as $k\to\infty$,
		\begin{align}
			|\Sigma^\ell_k|=O(1)\label{area}
		\end{align}
		and
			\begin{align} 
			|U_\ell\cap\partial \Sigma_k|=O(1) \label{length}.
		\end{align}  
	Moreover,   
		\begin{align}
			|k( \partial \Sigma_k)(x)|=O(1). \label{g curvature}
		\end{align} 
 for all    $x\in U_\ell\cap \partial \Sigma_k$.	
	\end{prop}
	\begin{proof} 
		\eqref{area} follows from \eqref{S assumption 1}, \eqref{S assumption 2}, and \eqref{u assumption 5}.
		
		To prove \eqref{length}, let $\eta \in C^\infty(\mathbb{R}^3)$  nonnegative be such that $\eta(x)=1$ if $x\in   U_{\ell}$  and $\eta(x)=0$ if $x\notin U_{\ell+1}$. Let $\eta_k\in C^\infty(\Sigma_k)$ be given by $\eta_k(x)=\eta(x)$ if $x\in \Sigma^{\ell+1}_k$ and $\eta_k(x)=0$ else. By Lemma \ref{graph}, using  \eqref{S assumption 1}, \eqref{S assumption 2},  \eqref{u assumption 5}, \eqref{first distance}, and Taylor's theorem, there holds
		\begin{align} \label{distance vs normal 1} 
		1+D\operatorname{dist}(\,\cdot\,,S_k)\cdot \nu(\Sigma_k)=O(1)\,(\sin\theta_k)^2
		\end{align} 
		 in  $\Sigma^{\ell+1}_k$ so that
		\begin{align} \label{distance vs normal 2} 
		\nabla^{\Sigma_k}\operatorname{dist}(\,\cdot\,,S_k)=O(1)\sin\theta_k.
		\end{align} 
		 By \eqref{u assumption 1} and \eqref{u assumption 5}, 
		\begin{align} \label{close}
			\sup\{\operatorname{dist}(x,S_k):x\in \Sigma^{\ell+1}_k\}=O(1)\sin\theta_k.
		\end{align}
	 Given $t\in(0,\delta_{\ell+2})$, let $$S^t_k=\{x\in \mathbb{R}^3:\operatorname{dist}(x,S_k)=t\}.$$ Note that $U_{\ell+2}\cap S^t_k$ is a smooth hypersurface for all $k$  sufficiently large.
	 
	  Assume that $S^t_k\cap \Sigma^{\ell+1}_k\neq \emptyset$.  By \eqref{close}, $t=O(1)\sin\theta_k$. In conjunction with \eqref{S assumption 3}, \eqref{S assumption 1}, and \eqref{S assumption 2},  we obtain  
	 $$ 	\sup\{|H(S^t_k)(y)|:y\in U_{\ell+2}\cap S^t_k\}=O(1)\sin\theta_k.$$  Since $\Delta^{S^t_k}\operatorname{dist}(\,\cdot\,,S_k)=0$ on $U_{\ell+2}\cap S^t_k$ and $D\operatorname{dist}(\,\cdot\,,S_k)$ is a unit normal of $S^t_k$, it follows that
		$$
		\Delta \operatorname{dist}(\,\cdot\,,S_k)-D^2\operatorname{dist}(\,\cdot\,,S_k)(D\operatorname{dist}(\,\cdot\,,S_k),D\operatorname{dist}(\,\cdot\,,S_k))=O(1)\sin\theta_k
		$$
in $\Sigma^{\ell+1}_k$. 		In conjunction with \eqref{u assumption 1}, \eqref{distance vs normal 1} and \eqref{distance hessian estimate}, using that $\Sigma_k$ is minimal, we conclude that  $$\Delta^{\Sigma_k}\operatorname{dist}(\,\cdot\,,S_k)=O(1)\sin\theta_k$$ in $\Sigma^{\ell+1}_k$. Using also \eqref{distance vs normal 2}, we obtain
		\begin{align} \label{divergence estimate} 
		\operatorname{div}_{\Sigma_k}\left(\eta_k\,\nabla^{\Sigma_k}\operatorname{dist}(\,\cdot\,,S_k)\right)=O(1)\sin\theta_k.
		\end{align} 
	By Lemma \ref{angles}, 	
		 on $\partial \Sigma_k$, 
		$$
		\sin\theta_k=\mu(\Sigma_k)\cdot\nu(S_k) =-\mu(\Sigma_k)\cdot \nabla^{\Sigma_k}\operatorname{dist}(\,\cdot\,,S_k).
		$$
		Thus, 
	$$
	\sin(\theta_k)\,|U_\ell\cap \partial \Sigma_k|\leq \sin(\theta_k)\,\int_{\partial  \Sigma_k}\eta_k=-\int_{\partial  \Sigma_k}\eta_k\,\mu(\Sigma_k)\cdot \nabla^{\Sigma_k}\operatorname{dist}(\,\cdot\,,S_k).
	$$
	By \eqref{divergence estimate} and the divergence theorem,
	$$
	-\int_{\partial  \Sigma_k}\eta_k\,\mu(\Sigma_k)\cdot\nabla^{\Sigma_k} \operatorname{dist}(\,\cdot\,,S_k)=O(1)\sin(\theta_k)\,|\Sigma_k^{\ell+1}|.
	$$
	In conjunction with \eqref{area} with $\ell$ replaced by $\ell+1$, we obtain \eqref{length}. 
	
		 	By Corollary \ref{graph 2}, \eqref{u assumption 1}, and \eqref{u assumption 2}, in $U_\ell\cap\partial \Sigma_k$, there holds
	\begin{align*} 
	(1+|\nabla^{S_k}u_k|^2)^{\sfrac12}\,h(\Sigma_k)(e( \Sigma_k),e( \Sigma_k))&=-\nabla^{S_k}\nabla^{S_k}u_k(e( \Sigma_k),e( \Sigma_k))-h(S_k)(e(\Sigma_k),e( \Sigma_k)).
	\end{align*} 
	In conjunction with \eqref{u assumption 6}, we obtain 
		\begin{equation} \label{hessian difference} 
		\begin{aligned}  
				&h(\Sigma_k)(e( \Sigma_k),e( \Sigma_k))-\cos(\theta_k)\,h(S_k)(e(\Sigma_k),e( \Sigma_k))\\
		&\qquad =\cos(\theta_k)\,\nabla^{S_k}\nabla^{S_k}u_k(e( \Sigma_k),e( \Sigma_k)).
	\end{aligned} 
	\end{equation} 
	By \eqref{curvature} and Lemma \ref{angles},
	\begin{align*} 
	k( \Sigma_k)&=\frac{1}{\sin\theta_k}\,\big(h(\Sigma_k)(e(\Sigma_k),e( \Sigma_k))-\cos(\theta_k)\,h(S_k)(e(\Sigma_k),e( \Sigma_k))\big)\,\mu(S(\Sigma_k))
	\\&\qquad +h(S_k)(e(\Sigma_k),e( \Sigma_k))\,\nu(S_k).
	\end{align*}
	Using this  \eqref{S assumption 1}, \eqref{S assumption 2}, \eqref{u assumption 5},    
	and \eqref{hessian difference},  we obtain \eqref{g curvature}.
	
	This completes the proof of the lemma.
\end{proof}
\begin{rema} \label{bounded curvature preparation}
	Let $x\in U_\ell\cap S$ and $x_k\in U_\ell\cap\partial \Sigma_k$ be such that $x_k\to x$. 
	The proof of Proposition \ref{estimates} shows that
		\begin{align*}
	\limsup_{k\to\infty}	|k( \partial \Sigma_k)(x_k)|\leq \limsup_{k\to\infty} \frac{|\nabla^{S_k}\nabla^{S_k}u_k(x_k)|}{\sin\theta_k}+|h(S)(x)|.
	\end{align*} 
\end{rema}
	\begin{figure}\centering
	\includegraphics[width=.9\linewidth]{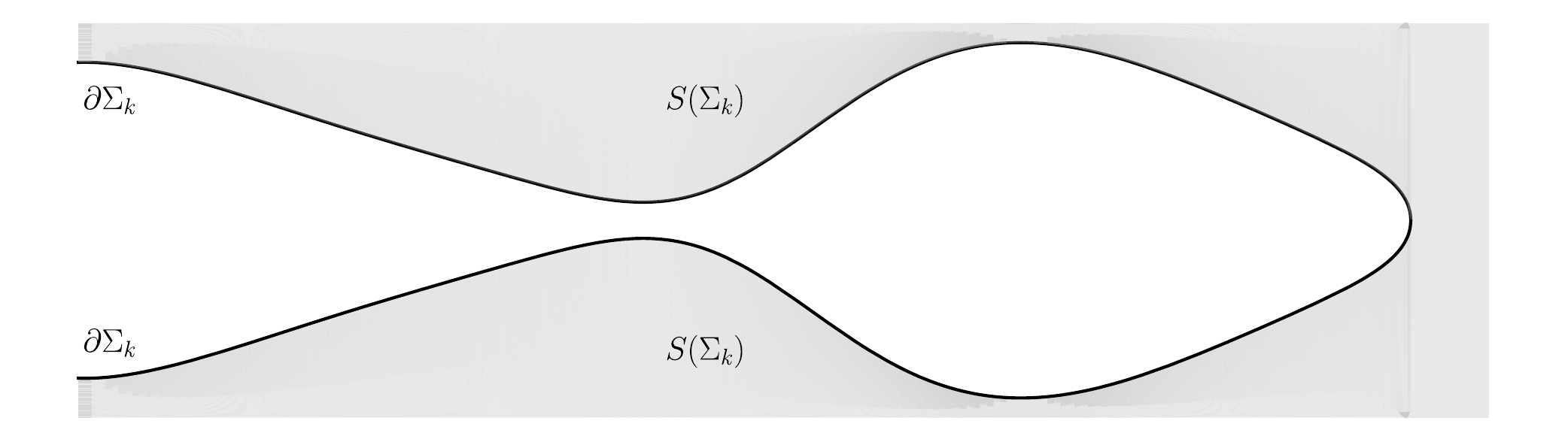}
	\caption{An illustration of the statement of Lemma \ref{non collapsing} where $S_k(\Sigma_k)$ is indicated by the shaded region. In the depicted example, the reach of $\partial \Sigma_k\subset S$ tends to zero while the collar radius of $\partial \Sigma_k\subset S_k(\Sigma_k)$ is bounded from below by a positive constant.
	}
\end{figure}
\begin{lem} \label{non collapsing}
Fix $\ell \geq1$. There is $s(\ell)\in(0,\infty)$ 
 such that $$\exp^{S_k}_x(-s\,\mu(S_k(\Sigma_k))(x))\in  S_k(\Sigma_k)\setminus \partial \Sigma_k$$ for  all $k$ sufficiently large,   $x\in U_{\ell}\cap \partial \Sigma_k,$ and $s\in (0,s(\ell))$. 
\end{lem}
\begin{proof}
	Suppose not. Passing to a subsequence, there are  
	$s_k\in(0,\infty)$ with $s_k=o(1)$  and $x_k\in U_\ell\cap \partial \Sigma_k$ such that
	$$
	z_k=\exp^{S_k}_{x_k}(-s_k\,\mu(S_k(\Sigma_k))(x_k))\in U_{\ell+1}\cap \partial \Sigma_k 
	$$ and, for all $s\in(0,s_k)$,  $$\exp^{S_k}_{x_k}(-s\,\mu(S_k(\Sigma_k))(x_k))\in  U_{\ell+1}\cap S_k(\Sigma_k)\setminus \partial \Sigma_k.$$ 	
	 By \eqref{S assumption 1}, \eqref{S assumption 2}, and \eqref{g curvature},  
  $\mu(S_k(\Sigma_k))( x_k)\cdot \mu(S_k(\Sigma_k))(z_k)\leq-1+ o(1).$
	In particular, $$|\mu(S_k(\Sigma_k))( x_k)- \mu(S_k(\Sigma_k))(z_k)|\geq 1$$ for all $k$ sufficiently large. Using Lemma \ref{angles}, we conclude that
	$$|\nu( \Sigma_k)(x_k)^{\top(S_k)}-\nu(\Sigma_k)(z_k)^{\top(S_k)}|\geq \sin\theta_k.
	$$ 
	By contrast, using \eqref{u assumption 5},  
	Lemma \ref{graph}, and integration, we obtain
	\begin{align*} 
|\nu( \Sigma_k)(x_k)^{\top(S_k)}-\nu(\Sigma_k)(z_k)^{\top(S_k)}|=O(1)\,|\nabla^{S_k}u_k(x_k)-\nabla^{S_k}u_k(z_k)|=o(1)\sin\theta_k.
	\end{align*}
 
This contradiction completes the proof of the lemma.
\end{proof}

Let
$$
U=\bigcup_{\ell=1}^\infty U_\ell.
$$

Let $v_k=(-\cot\theta_k)\, u_k$.

\begin{prop} \label{convergence}
There is a generalized domain $\Omega$ in $U\cap S$ and  $ v\in C^\infty(\Omega)\cap C_{loc}^{1,1}(\tilde \Omega)$
nonnegative such that 
\begin{align}
		&\circ \qquad \text{$-\Delta^S v-|h(S)|^2\,v=L(S)$ in $\Omega$,} \label{PDE}\\
	&\circ\qquad \text{$v=0$ on $\tilde \partial \Omega$, and} \label{v}\hspace{8cm}\\
	&\circ \qquad \text{$|\nabla ^S v|=1$  on $\tilde \partial  \Omega$.} \label{grad v}
\end{align}
 Passing to a subsequence,
 \begin{align}
 	&\circ \qquad \text{$S(\Sigma_k)\setminus \partial \Sigma_k\to \Omega$ in $C_{loc}^{1,\alpha}$ in $U\cap S$,} \hspace{6.1cm} \label{generalized domain convergence} \\
 	&\circ \qquad \text{$v_k\to v$  in $C_{loc}^{1,\alpha }(\tilde \Omega),$  and} \label{v convergence} \\
 	&\circ \qquad \text{$v_k\to v$ in $C^{2,\alpha}_{loc}(\Omega)$} \label{v convergence 1} 
 \end{align}
for all $0<\alpha<1$. Moreover,
 \begin{align}
&\circ\qquad  	\text{$|h(\Sigma_k)\circ \Phi_{S_k}(u_k)|\to|h(S)|$ in $C_{loc}(\Omega)$.   }\label{second fundamental form convergence} \hspace{4.3cm}
 \end{align}
 
\end{prop}

\begin{proof}
The existence of $\Omega$ and \eqref{generalized domain convergence} follows from Proposition \ref{estimates}, Lemma \ref{non collapsing}, and Lemma \ref{compactness}.

 By \eqref{S assumption 1}, \eqref{S assumption 2}, \eqref{u assumption 2},  \eqref{u assumption 5}, and \eqref{u assumption 6},   there is $ v\in C_{loc}^{1,1}(\tilde \Omega)$ that satisfies \eqref{v}, \eqref{grad v}, and \eqref{v convergence}. By the Jacobi equation, using \eqref{S assumption 1}, \eqref{S assumption 2}, \eqref{u assumption 1},   \eqref{u assumption 5}, and that $\Sigma_k$ is minimal, there holds 
$$
H(S_k)=-\Delta^{S_k}u_k-|h(S_k)|^2\,u_k+O(1)\,(\sin\theta_k)^2
$$  
uniformly on compact subsets of $S_k(\Sigma_k)$. 
In conjunction with \eqref{S assumption 4},
$$
L(S)=-\Delta^{S_k}v_k-|h(S_k)|^2\,v_k+o(1)
$$  
uniformly on compact subsets of $\Omega$. Here, the estimates for the error terms are in $C^{0,\alpha}_{loc}(\Omega)$.  By the Schauder estimates for differential operators in divergence form, using \eqref{S assumption 2} and \eqref{S assumption 4}, $v_k$ is bounded in $C_{loc}^{2,\alpha}(\Omega)$; see, e.g., \cite[Theorem 3.1 and Theorem 3.13]{HanLin}. From this, \eqref{S assumption 1}, \eqref{S assumption 2},   \eqref{u assumption 5}, and Lemma \ref{graph},  
we obtain \eqref{PDE}, \eqref{v convergence 1},   and \eqref{second fundamental form convergence}. By \eqref{PDE} and standard elliptic theory,  $v\in C^\infty(\Omega)$. This completes the proof of the proposition.
\end{proof}
\begin{coro} \label{bounded geometry}
	Assume that $|h(S)|_{L^\infty(U\cap S)}<\infty$ and 
	$$
\limsup_{\ell\to\infty}	\limsup_{k\to\infty}\frac{1}{	\sin\theta_k}\,\operatorname{sup}\{|\nabla^{S_k}\nabla^{S_k}u_k(y)|:y\in U_{\ell}\cap \partial \Sigma_k\}<\infty.
	$$
	Then $|k(\Omega)|_{L^\infty(\tilde \partial \Omega)}<\infty$. 
\end{coro}
\begin{proof}
	This follows from Remark \ref{bounded curvature preparation} and Proposition \ref{convergence}. 
\end{proof}

\begin{lem} \label{stability lem} Assume that $\Sigma_k$ is stable for the free energy in $U_\ell$ for all $\ell,\,k\geq 1$.
 Then $(\tilde \Omega,g)$ is stable with respect to $L(S)$.  
\end{lem}
\begin{proof}
	Let $f\in C^1(\tilde \Omega)$ have compact support.
 Using Proposition \ref{convergence},  we may choose $\ell \geq 1$ and functions $f_k\in C^\infty(S(\Sigma_k))$ with compact support in $U_\ell\cap S(\Sigma_k)$ such that $f_k\to f$ in $C^1(\tilde \Omega)$. Let $\tilde f_k:\Sigma^{\ell}_k\to \mathbb{R}$ be given by $\tilde f_k=f_k\circ \Pi_{S_k}$. Using \eqref{S assumption 1}, \eqref{S assumption 2}, \eqref{S assumption 4}, \eqref{u assumption 1}, \eqref{u assumption 5}, \eqref{second fundamental form convergence},  Proposition \ref{convergence}, and the dominated convergence theorem, we see that, as $k\to\infty$,
\begin{align*} 
		\int_{\Sigma_k}|\nabla^{\Sigma_k}\tilde f_k|^2&\to \int_{\Omega}|\nabla^S f|^2,\\
		\int_{\Sigma_k}|h( \Sigma_k)|^2\,\tilde f_k^2&\to \int_{\Omega} |h(S)|^2\,f^2,\\
	\int_{\partial \Sigma_k}k( \Sigma_k)\cdot \mu(\Sigma_k)\,\tilde f_k^2&\to \int_{\tilde \partial \Omega}k( \Omega)\cdot \mu(\Omega)\,f^2,\text{ and}\\
	\frac{1}{\sin\theta_k}\,\int_{ \partial \Sigma_k}H(S_k)\,\tilde f_k^2&\to \int_{\tilde \partial \Omega} L(S)\,f^2.
	\end{align*}
The assertion follows from this and the stability of $\Sigma_k$ for the free energy in $U_\ell$.	
\end{proof}
\begin{lem} \label{weak stability lem} Assume that $\Sigma_k$ is weakly stable for the free energy in $U_\ell$ for all $\ell,\,k\geq 1$.
Then $(\tilde \Omega,g)$ is weakly stable with respect to $L(S)$.  
\end{lem}
\begin{proof}
This follows as in the proof of Lemma \ref{stability lem}. The only difference is that we approximate $f$  by functions $f_k\in C^\infty(S(\Sigma_k))$ with 
$$
\int_{\partial \Sigma_k}f_k=0.
$$ 
\end{proof}
\begin{lem} \label{hodo graph}
	Assume that there is $0<\alpha<1$ such that $\Omega$ is an ordinary $C^{2,\alpha}$-domain in $U$. Then $S(\Sigma_k)\setminus \partial \Sigma_k\to\Omega$ in $C^{1,\alpha}_{loc}$ in $U$ and $v_k\to v$ in $C^{2,\alpha}_{loc}(\tilde \Omega)$. 
\end{lem}
\begin{proof}
	This follows from elliptic regularity as employed in the proof of \cite[Proposition 4.11]{chodoshedelenli}, using Proposition \ref{convergence} and that $S_k\to S$ in $C^{2,\alpha}_{loc}$ in $U$ for every $0<\alpha<1$.
\end{proof}
\section{Stable and weakly stable generalized domains}
	
	Recall from Appendix \ref{appendix support surfaces} the definition of a support surface $S\subset \mathbb{R}^3$ with unit normal $\nu(S)$  and of the domain $D(S)\subset \mathbb{R}^3$ bounded by $S$.  Moreover, recall our conventions for the second fundamental form $h(S)$, the shape operator $A(S)$, and the mean curvature $H(S)$ of $S$.
	
Given $U\subset \mathbb{R}^3$ open, recall from Appendix \ref{appendix: generalized domain} the definition of a generalized domain $\Omega$ in $U\cap S$ and of its associated Riemannian 2-manifold $(\tilde \Omega,g)$. Given $L(S)\in C^\infty(S)$, recall the concepts of stability and weak stability of a generalized domain  with respect to $L(S)$.
	
	Let $S\subset \mathbb{R}^3$ be a support surface, $U\subset \mathbb{R}^3$ open, and $L(S)\in C^\infty(S)$  nonnegative.

Let $\Omega$ be a generalized domain in $U\cap S$ and $(\tilde \Omega,g)$ its associated Riemannian 2-manifold. 

Assume that there is  $ v\in C^\infty(\Omega)\cap C^{1,1}(\tilde \Omega)$ nonnegative
such that 
\begin{align}  
		&\circ \qquad \text{$-\Delta^S v-|h(S)|^2\,v=L(S)$ in $\Omega$, }\label{v assumptions 1}\\ 
	&\circ\qquad \text{$v=0$ on $\tilde \partial \Omega$, and} \label{v assumptions 2}\\
	&\circ \qquad \text{$|\nabla ^S v|=1$  on $\tilde \partial  \Omega$.} \label{v assumptions 3}
\end{align} 
By the strong maximum principle, $v>0$ or $v=0$ in every component of $\Omega$. In particular, $v>0$ in every component of $\Omega$ that has nonempty boundary. Moreover, there holds $\nabla^Sv=-\mu(\Omega)$ on $\tilde \partial \Omega$. 

In this section, we characterize such configurations in certain cases.

\begin{prop} \label{no bounded components}
	Assume that $U=\mathbb{R}^3$, that $L(S)= 0$, and that $(\tilde \Omega,g)$ is stable with respect to $L(S)$. Then every component of $\Omega$ is unbounded. 
\end{prop}
\begin{proof}
	
	We may and will assume that $\Omega$ is connected.
	
	Suppose, for a contradiction, that $\Omega$ is bounded. By Lemma \ref{complete},  $(\tilde \Omega,g)$ is a compact Riemannian 2-manifold with nonempty $C^{1,1}$ boundary.	By the divergence theorem, using \eqref{v assumptions 1} and \eqref{v assumptions 2}, we have
	$$
	\int_\Omega |h(S)|^2\,v^2= \int_\Omega |\nabla^S v|^2.
	$$  
	Using \eqref{v assumptions 2} and \eqref{Euler-Lagrange}, we obtain $\nabla^Sv=0$  on $\tilde \partial \Omega$.
	This contradicts \eqref{v assumptions 3}.
\end{proof}

The proof of Lemma \ref{gradient bound} below is inspired by that of \cite[Proposition A.2]{KamburovWang}.

\begin{lem} \label{gradient bound}
	Assume that $U=\mathbb{R}^3$,  that $S$ is an affine plane, and that $L(S)=L$ is a nonnegative constant. Then
	$$
	\sup\{|\nabla^S v(y)|:y\in \Omega\}<6.
	$$
\end{lem}
\begin{proof}
	In the case where $L=0$, the assertion follows as in the proof of \cite[Proposition A.2]{KamburovWang}.
	
	Assume that  $L>0$. 	By scaling, we may assume that $L=1$. Moreover, we may assume that $S=\mathbb{R}^2\times\{0\}$.  
	
	Assume first that $\tilde \partial \Omega=\emptyset$. Then $v=0$ by the Liouville theorem and the assertion  follows.
	
 Assume second that $\tilde \partial \Omega\neq \emptyset$.	Let $y\in \Omega$ and $r=\operatorname{dist}_g(y,\tilde\partial \Omega)$. We claim that 
	\begin{align} \label{claim 1}
	r\leq 2.
	\end{align}
	
	 To see this, let 
	$$
	B=\{z\in \tilde \Omega:|z-y|\leq r\}
	$$
	be the largest ball in $\tilde \Omega$ with center at $y$ and
	 $z_0\in \tilde \partial \Omega$  such that $|y-z_0|=r$. Let
 $w: B\to \mathbb{R}$ be  given by $$w(z)=\frac{r^2}4-\frac{|z-y|^2}{4}.$$ Note that $w=0$ on $\partial B$. In particular, $v(z)\geq w(z)$ for all $z\in \partial B$ with equality if $z=z_0$. Also note that $v-w$ is harmonic in $B\setminus \partial B$. By the maximum principle, $v(z)\geq w(z)$ for all $z\in B$.  It follows that 
	$$
	0\geq \mu(\Omega)(z_0)\cdot (\nabla^{S}v(z_0)-\nabla^{S}w(z_0))=-1+\frac{|z_0-y|}{2}.
	$$
	This estimate gives \eqref{claim 1}.
	
	Next, we claim that  
	\begin{align} \label{claim 2}
	v(y)-w(y)< 3\,r.
	\end{align} 
	
	To see this, let $$\kappa=\inf\left\{v(z)-w(z):z\in B\text{ and }|z-y|=\sfrac{r}2\right\}.$$
	By the Harnack inequality in dimension 2, using that $v-w$ is harmonic and nonnegative,  
	\begin{align} \label{kappa} v(y)-w(y)\leq 3\,\kappa;
		\end{align}
		see \cite[(1)]{SerrinHarnack}.  Consider the annulus $$A=\left\{z\in \tilde \Omega:\sfrac{r}2\leq |z-y|\leq  r\right \}$$ and the function $h:A\to \mathbb{R}$  given by
	$$
	h(z)=-\frac{\kappa}{\log 2}\,\left(\log|z-y|-\log r\right).
	$$
	Note that  $h$ is harmonic in $A\setminus \partial A$, that  $h(z)=\kappa$ if $|z-y|=\sfrac{r}2,$ and that $h(z)=0$ if $|z-y|=r$. By the maximum principle,  $v(z)-w(z) \geq h(z)$ for all $z\in A$ with equality at $z_0$. It follows that
	$$
	0\geq  \mu(\Omega)(z_0)\cdot \left(\nabla^S(v-w)(z_0) -\nabla^Sh(z_0)\right)=-1+\frac{r}{2}+\frac{\kappa}{r\,\log2}.
	$$
	In particular, $r\, \log2\geq \kappa$. In combination with \eqref{kappa}, using that $1> \log 2$, this gives  \eqref{claim 2}.
	
By gradient estimates for positive harmonic functions, see \cite[Theorem 2.5 and Theorem 2.10]{GilbargTrudinger},  $$|\nabla^S (v-w)(y)|\leq \frac{2\,(v(y)-w(y))}{r}< 6.$$  
Since $\nabla^Sw(y)=0$, the assertion follows.
\end{proof}

\begin{lem} \label{weak stability stability}
Assume that $U=\mathbb{R}^3$,  that $S$ is an affine plane, that $\partial \Omega$ is unbounded,  that $L(S)\in C^\infty(S)$ is a nonnegative constant,  that   $|k(\Omega)|_{L^\infty(\tilde \partial \Omega)}<\infty$, and that  $(\tilde\Omega,g)$ is weakly stable with respect to $L(S)$. Then $(\tilde \Omega,g)$ is stable with respect to $L(S)$.
\end{lem}
\begin{proof}

	Note that, almost everywhere on $\tilde \partial \Omega$,
	$$
	\Delta^S v=\Delta^{\tilde \partial \Omega}v+\nabla^S\nabla^S v(\mu( \Omega),\mu(\Omega))+k(\Omega)\cdot \mu(\Omega)\,\mu(\Omega)\cdot \nabla^S v.
	$$
	In conjunction with \eqref{v assumptions 1}, \eqref{v assumptions 2}, and \eqref{v assumptions 3},   
	\begin{align} \label{prep 1} 
\mu( \Omega)\cdot \nabla^S|\nabla^S v|^2=2\,L(S)-2\,k(\Omega)\cdot \mu(\Omega)
	\end{align} 
almost everywhere on $\tilde \partial \Omega$.	Using \eqref{v assumptions 1}, the  Bochner formula, and that $L(S)$ is a constant, it follows that
	\begin{align} \label{prep 2} 
	\Delta^S|\nabla^Sv|^2=2\,|\nabla^S\nabla^S v|^2.
	\end{align} 
	Let $\eta \in C^\infty(\Omega)\cap C^{0,1}(\tilde \Omega)$ have compact support. By  \eqref{prep 1}, \eqref{prep 2}, and the divergence theorem,
	$$
	\frac12\,\int_\Omega \nabla^S |\nabla^S v|^2\cdot \nabla^S \eta^2=-\int_{\Omega} |\nabla^S\nabla^Sv|^2\,\eta^2+\int_{\tilde \partial \Omega} L(S)\,\eta^2-\int_{\tilde \partial \Omega} k(\Omega)\cdot \mu(\Omega)\,\eta^2.
	$$ 
	Thus,  using also \eqref{v assumptions 3} and that $|\nabla ^Sv|\in C^{0,1}(\tilde \Omega)$,
	\begin{align*} 
	&\int_\Omega |\nabla^S (|\nabla^S v|\,\eta)|^2+\int_{\tilde \partial \Omega} k(\Omega)\cdot \mu(\Omega)\,|\nabla^S v|^2\,\eta^2-\int_{\tilde \partial \Omega}L(S)\,|\nabla^S v|^2\,\eta^2
	\\&\qquad =-\int_{\Omega} |\nabla^S\nabla^Sv|^2\,\eta^2+\int_{\Omega}|\nabla^S|\nabla^Sv||^2\,\eta^2+\int_{ \Omega} |\nabla^S v|^2\,|\nabla^S \eta|^2.
	\end{align*} 
	According to Lemma \ref{gradient bound}, $|\nabla ^Sv|^2\leq 36$. Using that $|\nabla ^S\nabla^S v|\geq |\nabla^S |\nabla^Sv||$ almost everywhere, we conclude that
	\begin{align} \label{error bound} 
		&\int_\Omega |\nabla^S (|\nabla^S v|\,\eta)|^2+\int_{\tilde \partial \Omega} k(\Omega)\cdot \mu(\Omega)\,|\nabla^S v|^2\,\eta^2-\int_{\tilde \partial \Omega}L(S)\,|\nabla^S v|^2\,\eta^2
			 \leq 36\,\int_{\Omega} |\nabla^S \eta|^2.
	\end{align} 
	
Since $|k(\Omega)|_{L^\infty(\tilde \partial \Omega)}<\infty$,  there is  $\delta>0$ with the following property. For every  $y\in \tilde \partial \Omega$, there is $\eta_y\in C^\infty(\Omega)\cap C^1(\tilde \Omega)$ with compact support in $\{z\in \tilde \Omega:\operatorname{dist}_g(y,z)<1\}$ such that
	\begin{align} \label{bump functions} 
	\int_{\tilde \partial \Omega} \eta_y=\delta\qquad\text{and}\qquad \int_{\Omega} |\nabla^S \eta_y|^2 \leq 1.
	\end{align}

 Let $f\in C^1(\tilde \Omega)$ have compact support. Assume that 
 	$$\int_{\tilde \partial \Omega} f=1.
	$$
	Given $0<\varepsilon<\delta$, let $k\geq 1$ be an integer and  $\tilde \varepsilon\in(0,\varepsilon)$ such that $(k\,\varepsilon+\tilde \varepsilon)\,\delta=1.$ Note that 
	\begin{align} \label{point bound} 
	k\leq \frac{1}{\delta\,\varepsilon}.
	\end{align} 
Since  $\partial \Omega$ is unbounded,	there are $y_1,\,\,\ldots, \, y_{k+1}\in \tilde \partial \Omega$ such that the functions  $f,\,\eta_{y_1},\, \ldots,\, \eta_{y_{k+1}}$  have disjoint support. Let 
$$
\tilde f=f-\varepsilon\,\sum_{\ell=1}^k|\nabla^Sv|\,\eta_{y_\ell}-\tilde \varepsilon\,|\nabla^Sv|\,\eta_{y_{k+1}}.
$$ 
By \eqref{v assumptions 3} and \eqref{bump functions},
	$$
	\int_{\tilde \partial \Omega} \tilde f=0.
	$$
	By \eqref{error bound}, \eqref{bump functions}, and \eqref{point bound},
	$$
	\int_{\Omega} |\nabla^S (f-\tilde f)|^2+ \int_{\tilde \partial \Omega} k(\Omega)\cdot \mu(\Omega)\,(f-\tilde f)^2-\int_{\tilde \partial \Omega}L(S)\,(f-\tilde f)^2\leq 36\left(\frac{\varepsilon}{\delta}+\tilde \varepsilon^2\right)\leq 36\left(\frac{\varepsilon}{\delta}+\varepsilon^2\right).
	$$
	Since $(\tilde \Omega,g)$ is weakly stable with respect to $L(S)$, 
	$$
0\leq 	\int_{\Omega} |\nabla^S\tilde f|^2+ \int_{\tilde \partial \Omega} k(\Omega)\cdot \mu(\Omega)\,\tilde f^2-\int_{\tilde \partial \Omega}L(S)\,\tilde f^2.
	$$ 
Using that $f$ and $f-\tilde f$ have disjoint support, 
	$$
	-36\,\left(\frac{\varepsilon}{\delta}+\varepsilon^2\right)\leq 	\int_{\Omega} |\nabla^S f|^2+ \int_{\tilde \partial \Omega} k(\Omega)\cdot \mu(\Omega)\, f^2-\int_{\tilde \partial \Omega}L(S)\, f^2.
	$$
Since $\varepsilon>0$ is arbitrary, the assertion follows.
\end{proof}
\begin{rema}
	The assumption in Lemma \ref{weak stability stability} that $\partial \Omega$ is unbounded and  $|k(\Omega)|_{L^\infty(\tilde \partial \Omega)}<\infty$ can be replaced by the assumption that $\partial \Omega$ has an unbounded component. 
\end{rema}

\begin{prop} \label{half space}
	Assume that $U=\mathbb{R}^3$, that $S$ is congruent to an affine plane, and that $L(S)=0$. Moreover, assume that either
	\begin{align*}
		&\circ \qquad \text{$(\tilde\Omega,g)$ is stable with respect to $L(S)$ or that}\\
		&\circ \qquad \text{$(\tilde\Omega,g)$ is weakly stable with respect to $L(S)$ and  $|k(\Omega)|_{L^\infty(\tilde \partial \Omega)}<\infty$.} 
	\end{align*}
	  Then $\Omega$ has either one or two components. Each of these components is congruent to $$\{y\in \mathbb{R}^2:e_2\cdot y> 0\}\times\{0\}$$  where  $v=e_2\cdot y$. 
\end{prop}
\begin{proof}

	By Lemma \ref{weak stability stability}, $(\tilde\Omega,g)$ is stable with respect to $L(S)$. The assertion now  follows from  \cite[Theorem 3.1]{KamburovWang}. In fact, the modifications of the proof of \cite[Theorem 3.1]{KamburovWang} to the present setting, where $\Omega$ is a generalized domain rather than an ordinary domain, are purely formal.  
\end{proof}

\begin{lem} \label{no stable unbounded components}
Assume that $U=\mathbb{R}^3$,  that $S$ is congruent to $\mathbb{R}^2\times\{0\}$,  that $L(S)=L$ is a positive constant, and that $(\tilde \Omega,g)$ is stable with respect to $L(S)$. Then $\Omega$ has no unbounded components.
\end{lem}
\begin{proof}
	Suppose, for a contradiction, that $\Omega$ has an unbounded component.
	
	We may assume that $\Omega$ is connected. Otherwise, we replace $\Omega$ by one of its unbounded  component and proceed. 
	
	By scaling, we may assume that $L=1$.

	 The  proof of Lemma \ref{weak stability stability}, using that $(\tilde \Omega,g)$ is stable with respect to $L(S)$, shows that 
	$$
	\int_{\Omega}|\nabla^S\nabla^S v|^2\,\eta^2-|\nabla^S|\nabla^Sv||^2\,\eta^2\leq 36\,\int_{\Omega}|\nabla^S \eta|^2
	$$
		for all $\eta \in C^\infty(\Omega)\cap C^1(\tilde \Omega)$ with compact support.
	Note that $|\nabla^S\nabla^S v|^2-|\nabla^S|\nabla^Sv||^2\geq 0$ almost everywhere.
	Using the  logarithmic cut-off trick, we obtain that 
	$$
	\int_{\Omega}|\nabla^S\nabla^S v|^2-|\nabla^S|\nabla^Sv||^2\leq 0;
	$$ 
	see the proof of \cite[Theorem 1.2]{KamburovWang}. 
It follows that $|\nabla^S\nabla^S v|=|\nabla^S|\nabla^Sv||$ almost everywhere. On the one hand, by the equality case in the Cauchy-Schwarz inequality,  we see that $|\nabla^Sv|^2\,\nabla^S\nabla^S v$ and $\nabla^Sv\otimes \nabla^S v$ are proportional. Using \eqref{v assumptions 1}, we obtain that 
$$
|\nabla^Sv|^2\,\nabla^S\nabla^S v=-\nabla^Sv\otimes \nabla^S v.
$$
 Differentiating, we see that $|\nabla^S v|^{-1}\,\nabla^S v$ is locally constant in $\{y\in \Omega:\nabla^S v(y)\neq 0\}$. It follows that $v$ is locally a 1-dimensional quadratic function in $\{y\in \Omega:\nabla^S v(y)\neq 0\}$. On the other hand, if the interior of $\{y\in \Omega:\nabla^S v(y)= 0\}$ is nonempty, then there is $y\in \Omega$ with $\nabla^S\nabla^Sv(y)=0$. This contradicts our assumption that $L=1$. Thus, $v$ is globally a 1-dimensional quadratic function in $\Omega$.   In conjunction with  \eqref{v assumptions 2} and \eqref{v assumptions 3}, 
rotating and translating $S$ if necessary, we conclude that $\Omega=\{y\in\mathbb{R}^2:-1<e_2\cdot y<1\}\times\{0\}$ and 
	$$
	v=-\frac12\,(e_2\cdot y)^2+\frac12.
	$$
	In particular, $k(\Omega)=0$.
	
Since $(\tilde \Omega,g)$ is stable, 
\begin{align} \label{stable contradiction}  
\int_{\tilde \partial \Omega} \eta^2\leq \int_{\Omega} |\nabla^S \eta|^2
\end{align} 
for every $\eta \in C^\infty(\Omega)\cap C^1(\tilde \Omega)$ with compact support.
By the logarithmic cut-off trick, there are $\eta_1,\,\eta_2,\,\ldots \in C^\infty(\Omega)\cap C^1(\tilde \Omega)$ such that $\eta_k\to 1$ in $C_{loc}(\tilde \Omega)$ while
$$
\int_{\Omega} |\nabla^S \eta_k|^2\to 0.
$$
This contradicts \eqref{stable contradiction}.

This completes the proof of the lemma.
\end{proof}

\begin{lem} \label{serrin overdetermiend}
Assume that $U=\mathbb{R}^3$,  that $S$ is an affine plane, and that $L(S)=L$ is a positive constant.	If  $\Omega$ is connected and compact, then $ \Omega$ is congruent to $\{y\in \mathbb{R}^2:|y|<\sfrac{2}{L}\}\times \{0\}$ and 
	$$
	v(y)=\frac{1}{L}-\frac{L\,|y|^2}4.
	$$
\end{lem}
\begin{proof}

	The alternative proof of \cite[Theorem 1]{JamesSerrin} given in \cite{Weinberger} carries over verbatim to the present setting, where $\Omega$ is  a generalized domain rather than an ordinary domain. 
\end{proof}
\begin{prop} \label{positive curvature limit}
Assume that $U=\mathbb{R}^3$,  that $S$ is an affine plane, that $L(S)=L$ is a positive constant,  that $(\tilde \Omega,g)$ is weakly stable with respect to $L(S)$, and that $|k(\Omega)|_{L^\infty(\tilde \partial \Omega)}<\infty$. Then $\Omega$ is congruent to $\{y\in \mathbb{R}^2:|y|<\sfrac{2}{L}\}\times \{0\}$ and  
	$$
	v(y)=\frac{1}{L}-\frac{L\,|y|^2}{4}.
	$$
\end{prop}
\begin{proof}
	Suppose, for a contradiction, that $\Omega$ has an unbounded component. By \eqref{claim 1}, $\partial \Omega$ is unbounded. 
	By Lemma \ref{weak stability stability}, $(\tilde \Omega,g)$ is stable. 
	 This contradicts Lemma \ref{no stable unbounded components}.
	
	By Lemma \ref{serrin overdetermiend}, every component of $\Omega$ is congruent to a round disk of radius $\sfrac{2}{L}$. Assume that  $\Omega$ has at least two components,  $\Omega_1$ and $\Omega_2$. Using that $k(\Omega)\cdot \mu(\Omega)=\sfrac{L}2$, we see that the function $f=\chi|_{\Omega_1}-\chi|_{\Omega_2}$ satisfies 
	$$
	8\,\pi=\int_{\tilde \partial \Omega}L\,f^2>\int_{ \Omega}|\nabla^S f|^2+\int_{\tilde \partial \Omega}k(\Omega)\cdot \mu(\Omega)\,f^2=4\,\pi.
	$$ 
	Since
	$$
	\int_{\tilde \partial \Omega}f=0,
	$$
	this contradicts the weak stability of  $(\tilde\Omega,g)$ with respect to $L(S)$. Thus $\Omega$ is connected.
	
	This completes the proof of the lemma.
\end{proof}

\begin{lem} \label{preemtive}
	Assume that $U=\mathbb{R}^3\setminus\{0\}$, that $S$ is a plane through $0$, that $L(S)=0$, and that $0$ is an accumulation point of $\partial \Omega$.  Let $y\in \Omega$ be such that $\operatorname{dist}_{g}(y,\tilde \partial \Omega)\leq |y|$.  Then $v(y)\leq 3\operatorname{dist}_g(y,\tilde \partial \Omega)$. 
\end{lem} 
\begin{proof}

We distinguish two cases.

Case 1: Assume that $\operatorname{dist}_{g}(y,\tilde \partial \Omega)<|y|$.  

The assertion of  the lemma follows from the  argument leading to \eqref{claim 2}.

Case 2: Assume that $\operatorname{dist}_{g}(y,\tilde \partial \Omega)=|y|$. 

Let $z\in \tilde \partial \Omega$ be such that $|y-z|=\operatorname{dist}_g(y,\tilde \partial \Omega)$. Since $0\notin \tilde \partial \Omega$, there holds $z\neq 0$. Given $0<\varepsilon<1$, let $y_\varepsilon=(1-\varepsilon)\,y+\varepsilon\,z$. Note that $\operatorname{dist}_g(y_\varepsilon,\tilde\partial  \Omega)=(1-\varepsilon)\operatorname{dist}_g(y,\tilde \partial \Omega)$ and $$|y_\varepsilon|>|y|-\varepsilon\,|y-z|=(1-\varepsilon)\,|y|.$$
By Case 1, $v(y_\varepsilon)\leq 3\operatorname{dist}_g(y_\varepsilon,\tilde \partial \Omega)$. The assertion follows upon letting $\varepsilon\to 0$.
\end{proof}

\begin{lem} \label{extension}
		Assume that $U=\mathbb{R}^3\setminus\{0\}$, that $S$ is a plane through $0$, that $L(S)=0$, and that $0$ is an accumulation point of $\partial \Omega$. Let $y_j\in \Omega$ be points such that $y_j\to y$ for some $y\in  \partial \Omega$. Then $v(y_j)\to 0$. 
\end{lem}
\begin{proof}
		
	If $y\neq 0$, the assertion follows from \eqref{v assumptions 2}. 
	
	Assume that $y=0$.  Applying a rotation if necessary, we may arrange that $S=\mathbb{R}^2\times\{0\}$.
	
In view of Lemma \ref{preemtive}, we may further assume that $|y_j|<\operatorname{dist}_g(y_j,\tilde \partial \Omega)$. Fix $j\geq 1$. We claim that 
\begin{align}  \label{claim 4}
v(y_j)\leq 2\cdot 3^4\,|y_j|.
\end{align} 
 	\begin{figure}\centering
	\includegraphics[width=.6\linewidth]{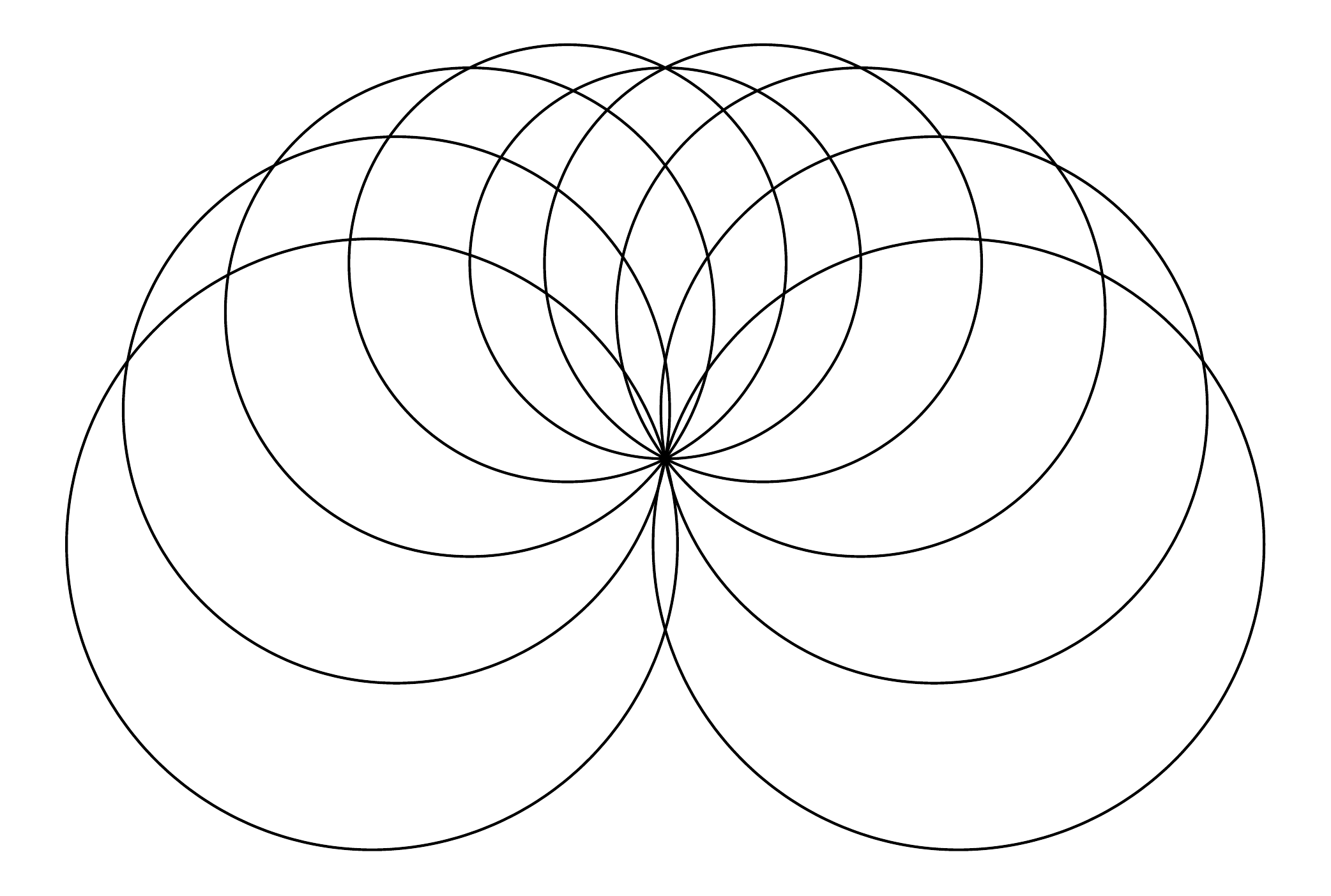}
	\caption{An illustration of the balls $B_{|y^i_j|}(y^i_j)$ and $B_{|\tilde y^i_j|}(\tilde y^i_j)$, $i=0,\,1,\,2,\,3,\,4$. For each $i=1,\,2,\,3,\,4$, there holds  $y^i_j\in \operatorname{cl}B_{\sfrac{|y^{i-1}_j|}2}(y^{i-1}_j)$ and $\tilde y^i_j\in \operatorname{cl}B_{\sfrac{|\tilde y^{i-1}_j|}2}(\tilde y^{i-1}_j)$. Moreover, the union of these balls contains a centered ball with a puncture at the origin.
	}
	\label{angst_figure}
\end{figure}
	Assume that $y_j=|y_j|\,(0,1,0)$. We put $\tilde y^0_j=y^0_j=y_j$ and  define the following points.
	\begin{align*}
		y^1_j&=|y_j|\,(-\sfrac12,1,0)\qquad\qquad 	\tilde	y^1_j=|y_j|\,(\sfrac12,1,0)\\
	 y^2_j&=|y_j|\,(-1,\sfrac34,0)\qquad \quad\,\,\,\,\,\, \tilde y^2_j=|y_j|\,(1,\sfrac34,0)\\
	 y^3_j&=|y_j|\,(-\sfrac{11}{8},\sfrac14,0)\qquad\quad \tilde y^3_j=|y_j|\,(\sfrac{11}{8},\sfrac14,0)\\
	 y^4_j&=|y_j|\,(-\sfrac32,-\sfrac7{16},0)\qquad \,\tilde y^4_j=|y_j|\,(\sfrac32,-\sfrac7{16},0)
	\end{align*}  
 Note that $$2\,|y^i_j-y^{i-1}_j|=|y^{i-1}_j|\qquad\text{and}\qquad 2\,|\tilde y^i_j-\tilde y^{i-1}_j|=|\tilde y^{i-1}_j|,$$
 $i=1,\,2,\,3,\,4,$ and that 
	$$
	(B_{\sfrac{|y_j|}2}(0)\setminus \{0\})\cap S\subset \left(B_{|y_j|}(y_j)\cup  \bigcup_{i=1}^4 \left(B_{|y_j^i|}(y_j^i)\cup B_{|\tilde y^i_j|}(\tilde y_j^i)\right)\right)\cap S;
	$$
	see Figure \ref{angst_figure}.
Since the origin is an accumulation point of $\partial \Omega$, there holds
$$
\partial \Omega\cap   \bigcup_{i=1}^4 \left(B_{|y_j^i|}(y_j^i)\cup B_{|\tilde y_j^i|}(\tilde y_j^i)\right)  \neq \emptyset.
$$
 Thus, there is a least $i_0\in\{1,\,2,\,3,\,4\}$ such that either $\operatorname{dist}(y^{i_0}_j,\tilde \partial \Omega)< |y_j^{i_0}|$ or  $\operatorname{dist}(\tilde y^{i_0}_j,\tilde \partial \Omega)< |\tilde y_j^{i_0}|$. We may assume that the former alternative occurs.  By the Harnack inequality for positive harmonic functions in dimension 2, we have 
	$$
	v(y_j)\leq 3^{i_0}\,v(y^{i_0}_j);
	$$
	see \cite[(1)]{SerrinHarnack}.
By Lemma \ref{preemtive}, $v({y^{i_0}_j})\leq 3\operatorname{dist}_g(y^{i_0}_j,\tilde \partial \Omega)$. In conjunction with the estimate
$$
\operatorname{dist}_g(y^{i_0}_j,\tilde \partial \Omega)<|y^{i_0}_j|<2\,|y_j|,
$$
we obtain \eqref{claim 4}. The assertion of the lemma follows from this.
\end{proof}
\begin{prop} \label{angst}
	Assume that $U=\mathbb{R}^3\setminus\{0\}$, that $S=\mathbb{R}^2\times \{0\}$, that $L(S)=0$, and that $\Omega$ is bounded. Then, for some $r>0$, $$\Omega=\{y\in \mathbb{R}^2:0<|y|<r\}\times\{0\},$$ and $$v=-r\,\log|y|+r\,\log r.$$
\end{prop}
\begin{proof}
	
	First, we claim that $0\in \operatorname{cl}\Omega$. For if not, by Lemma \ref{complete},   $(\tilde \Omega,g)$ is a compact Riemannian 2-manifold with $C^{1,1}$ boundary. By the  maximum principle, using \eqref{v assumptions 1} and \eqref{v assumptions 2}, we see that $v=0$. This contradicts \eqref{v assumptions 3}.
	
	Second, we claim that $0$ is not an accumulation point of $\partial \Omega$. For if not, then, by Lemma \ref{extension}, $v$ extends continuously to $\operatorname{cl}\Omega$ by  $v(0)=0$. By the  maximum principle, using \eqref{v assumptions 1} and \eqref{v assumptions 2}, $v=0$. This contradicts \eqref{v assumptions 3}.

   The same argument shows that $\Omega$ has only one connected component. In particular, $\tilde \partial\Omega$ is  closed.  Let $(\bar \Omega, g)$ be the metric completion of $(\tilde \Omega,g)$. Note that $(\bar \Omega,g)$ is a compact Riemannian 2-manifold with $C^{1,1}$ boundary. 
By B\^ocher's theorem for harmonic functions \cite{Bocher}, using \eqref{v assumptions 1}, there is  $w\in C^\infty( \bar \Omega\setminus \partial \bar \Omega)\cap C^{1,1}(\bar \Omega)$ and $r\in \mathbb{R}$ such that $w|_{\Omega}$ is harmonic and $v(y)=w(y)-r\,\log|y|$ for all $y\in \Omega$. By the divergence theorem, using \eqref{v assumptions 1} and \eqref{v assumptions 3}, we conclude that  
	$$
	-|\tilde \partial\Omega|=\int_{\tilde \partial\Omega}\mu( \Omega)\cdot \nabla^S v
	=-2\,\pi\,r.
	$$
	Thus, $2\,\pi\,r=|\tilde \partial \Omega|>0$. 
	
	Let $y_0,\,y_1\in \partial \Omega$ be such that $y_0=\inf \{|y|:y\in \partial \Omega\}$ and $y_1=\max \{|y|:y\in \partial \Omega\}$. Clearly, $\mu( \Omega)(y_0)=\sfrac{y_0}{|y_0|}$ and $\mu( \Omega)(y_1)=\sfrac{y_1}{|y_1|}$. Using that, by \eqref{v assumptions 2}, $w(y)=r\,\log|y|$ on $\tilde \partial \Omega$ and the maximum principle, we see that $w$ attains its minimum at $y_0$ and its maximum at $y_1$. By \eqref{v assumptions 3}, 
	$$
	0\geq \mu({ \Omega})(y_0)\cdot \nabla^Sw(y_0) =-1+\frac{r}{|y_0|}.	$$
	Thus, $|y_0|\geq \beta$. Likewise, we obtain $|y_1|\leq r$. By the strong maximum principle,  $w=r\,\log r$.
	
	This completes the proof of the proposition. 
\end{proof}

\section{Curvature estimates for weakly stable minimal capillary surfaces}
	Recall from Appendix \ref{appendix support surfaces} the definition of a support surface $S\subset \mathbb{R}^3$ with unit normal $\nu(S)$ and of the domain $D(S)\subset \mathbb{R}^3$ bounded by $S$.  Moreover, recall  our conventions for the second fundamental form $h(S)$, the shape operator $A(S)$, and the mean curvature $H(S)$ of $S$. Given $\psi \in C^\infty(S)$, recall the definition \eqref{Phi S} of $\Phi_S(\psi)$.

Recall from Appendix \ref{appendix minimal capillary surfaces}  the  definition of a minimal capillary surface $\Sigma \subset \mathbb{R}^3$ supported on a support surface $S\subset \mathbb{R}^3$  and that of its  wetting surface $S(\Sigma)\subset S$. In particular, recall that  $\Sigma\setminus \partial \Sigma \subset D(S)$.  Moreover, recall the concepts of  stability and weak stability for the free energy  of a minimal capillary surface. We use $\operatorname{dist}_\Sigma$ to denote the intrinsic distance function of $\Sigma$. 

Given $U\subset \mathbb{R}^3$ open, recall from Appendix \ref{appendix: generalized domain} the definition of a generalized domain $\Omega$ in $U\cap S$ and of its associated Riemannian 2-manifold $(\tilde \Omega,g)$. Given $L(S)\in C^\infty(S)$, recall the concepts of  stability and weak stability of a generalized domain  with respect to $L(S)$.  Moreover, recall the definitions of convergence of a sequence of domains to a generalized domain and of convergence of a sequence of functions defined on the domains of such a sequence to a function defined on the limiting generalized domain.

In this section,  we consider   support surfaces $S_k\subset \mathbb{R}^3$ and minimal capillary surfaces $\Sigma_k$ supported on $S_k$ with capillary angle $0<\theta_k<\pi$.

We assume that $\theta_k\nearrow  \pi$.

	Fix $0<r_0<r_2$ such that both  $S_{r_0}(0)$ and $S_{r_2}(0)$  intersect  $S_k$, $\Sigma_k$, and $\partial \Sigma_k$ transversely for all $k$. We assume that
	\begin{equation} 
\begin{aligned} \label{geometry bound S3} 
	\sup\{\sfrac1{i_{S_k}(y)}:y\in B_{r_2}(0)\cap S_k\}&=O(1),\\
	\sup\{|\nabla^{S_k}h(S_k)(y)|:y\in B_{r_2}(0)\cap S_k\}&=O(1),
\end{aligned} 
\end{equation} 
that
\begin{equation} \label{MC bound S3} 
	\begin{aligned}
		\sup\{|H(S_k)(y)|:y\in B_{r_2}(0)\cap S_k\}&=O(1)\sin\theta_k,
		\\  \sup\{|\nabla^{S_k}H(S_k)(y)|:y\in B_{r_2}(0)\cap S_k\}&=O(1)\sin\theta_k,
	\end{aligned} 
\end{equation} 
and that 
$$\text{$\Sigma_k$ is weakly stable for the free energy in $B_{r_2}(0)$}.$$

In this section, we prove that, for all sufficiently large $k$,  the normal graph of  a certain function  defined on $B_{r_0}(0)\cap S_k(\Sigma_k)$ is contained in $\Sigma_k$.

\begin{lem} \label{prelim curvature estimates}
Assume that  there are $r_0<r_1<r_2$ such that $S_{r_1}(0)$ intersects  $S_k$, $\Sigma_k$, and $\partial \Sigma_k$ transversely for all $k$ and $ u_k\in C^{0,1}(S_k)$ nonnegative with the following properties.
 \begin{equation}\label{hessian estimate Sigma assumptions 0}
	\begin{aligned}
		&\circ\qquad \text{$ \{\Phi_{S_k}(u_k)(y):y\in B_{r_1}(0)\cap S_k(\Sigma_k)\}$}\subset \Sigma_k\hspace{5cm}\\ 		
		&\circ\qquad \text{$u_k=0$ in $S_k\setminus S_k(\Sigma_k)$} \\
		&\circ\qquad \text{$ u_k\in C^\infty(B_{r_1}(0)\cap S_k(\Sigma_k))$, $u_k=O(1)\,\sin\theta_k$ in $B_{r_1}(0)\cap S_k(\Sigma_k)$, and}
 			\\&\qquad  \text{\,\,\,\,\,$\nabla^{S_k} u_k=O(1) \sin\theta_k$ in $ B_{r_1}(0)\cap  S_k(\Sigma_k)$}\\ 
		&\circ\qquad \text{$|\nabla^{S_k} u_k|=-\tan\theta_k$ on $B_{r_1}(0)\cap \partial \Sigma_k $}
	\end{aligned}
	\end{equation}
	There holds
	$$
 \text{$|\nabla^{S_k}\nabla^{S_k} u_k|=O(1)\sin\theta_k$ in $B_{r_0}(0)\cap  S_k(\Sigma_k)$}.
	$$
\end{lem}
\begin{proof}
	
	We follow the strategy of the proof of \cite[Proposition 4.1]{chodoshedelenli}.
	
Fix $r_0<\rho<r_1$. We claim that
	$$
	\sup\left\{\left(\rho-|y|\right)|\nabla^{S_k}\nabla^{S_k}u_k(y)|:y\in B_{\rho}(0)\cap S_k(\Sigma_k)\right\}=O(1)\sin\theta_k.
	$$

	For if not, then, passing to a subsequence, there are $y_k\in B_{\rho}(0)\cap S_k(\Sigma_k)$ such that
	\begin{align} \label{contradiction 0}  
	\left(\rho-|y_k|\right)	|\nabla^{S_k}\nabla^{S_k}u_k(y_k)|=	\sup\{\left(\rho-|y|\right)\,|\nabla^{S_k}\nabla^{S_k}u_k(y)|:y\in B_{\rho}(0)\cap S_k(\Sigma_k)\}\geq k\sin\theta_k.
	\end{align} 
	Moreover, $B_{\delta_k}(y_k)\subset B_\rho(0)$. Let $$\lambda_k=\frac{|\nabla^{S_k}\nabla^{S_k}u_k(y_k)|}{\sin\theta_k}\qquad\text{and}\qquad\delta_k=\frac{1}{2}\left(\rho-|y_k|\right).$$ Note that $$\lambda_k\to\infty\qquad\text{and}\qquad \delta_k\,\lambda_k\to\infty.$$ Let $y\in B_{\delta_k}(y_k)\cap S_k$. Note that 
	$$
	\rho-|y|\geq \rho-|y_k|-\delta_k= \frac{1}{2}\left(\rho-|y_k|\right).
	$$
	By \eqref{contradiction 0}, 
	\begin{align} \label{estimate} 
|\nabla^{S_k}\nabla^{S_k}u_k(y)|\leq2\,\frac{\rho-|y|}{\rho-|y_k|}\,|\nabla^{S_k}\nabla^{S_k}u_k(y)|\leq 2\,|\nabla^{S_k}\nabla^{S_k}u_k(y_k)|= 2\,\lambda_k\sin\theta_k. 
	\end{align} 

	We distinguish two cases.
	
	Case 1: Assume that $\lambda_k\operatorname{dist}(y_k,\,\partial \Sigma_k)=O(1)$. 
	
	Let $z_k\in \partial \Sigma_k$ be such that $\operatorname{dist}(y_k,\partial \Sigma_k)=|y_k-z_k|$. Passing to a further subsequence if necessary, there holds
	 $$y_k\in B_{\sfrac{\delta_k}2}(z_k)\subset B_{\delta_k}(y_k)\subset B_\rho(0)
	 $$
	 since $\delta_k\,\lambda_k\to\infty$.
	
 By \eqref{contradiction 0}, $4\,\rho_k\geq k$ where $$\rho_k=\frac{\delta_k\,\lambda_k}2.$$
	Let $\tilde S_k=\lambda_k\,(S_k-z_k)$, $\tilde \Sigma_k=\lambda_k\,(\Sigma_k-z_k)$, $\tilde y_k=\lambda_k\,(y_k-z_k)$, and $\tilde u_k\in C^{0,1}(\tilde S_k)$ be given by 
	$$
	\tilde u_k(y)=\lambda_k\,u_k\left(\lambda_k^{-1}\,y+z_k\right).
	$$
Note that $0\in \tilde S_k$.	In view of \eqref{geometry bound S3}, $\tilde S_k$ converges to a plane $\tilde S$  in $C_{loc}^{2,\alpha}$.  We may and will assume that $\tilde S=\mathbb{R}^2\times\{0\}$. By \eqref{MC bound S3} and since $\lambda_k\to \infty$, 
	\begin{align*} 
	\sup\{|H(\tilde S_k)(y)|:y\in B_{\rho_k}(0)\cap \tilde S_k\}&=o(1)\sin\theta_k\text{ and}\\
	\sup\{|\nabla^{\tilde S_k}H(\tilde S_k)(y)|:y\in B_{\rho_k}(0)\cap \tilde S_k\}&=o(1)\sin\theta_k.
	\end{align*}
Note that $0\in \partial \tilde \Sigma_k$ and $|\tilde y_k|=O(1)$. By the definition of $\lambda_k$ and \eqref{estimate},
		\begin{align} \label{curvature estimate rescaled 1} 
		|\nabla^{\tilde S_k}\nabla^{\tilde S_k}\tilde u_k(\tilde y_k)|=\sin\theta_k
	\end{align} 
	while
	\begin{align} \label{curvature estimate rescaled 2} 
		\sup\{|\nabla^{\tilde S_k}\nabla^{\tilde S_k}\tilde u_k(y)|:y\in B_{\rho_k}(0)\cap \tilde S_k(\tilde \Sigma_k)\}\leq 2\sin\theta_k.
	\end{align}
Moreover, $\tilde \Sigma_k$ is weakly stable for the free energy in $B_{\rho_k}(0)$. By \eqref{hessian and gradient},	the following hold.
	 \begin{equation*}
		\begin{aligned}
			&\circ\qquad \text{$ \{\Phi_{\tilde S_k}(\tilde u_k)(y):y\in B_{\rho_k}(0)\cap  {\tilde S}_k(\tilde \Sigma_k)\}\subset \tilde \Sigma_k$}\hspace{5cm}\\ 		
			&\circ\qquad \text{$\tilde u_k=0$ in $\tilde S_k\setminus \tilde S_k(\tilde \Sigma_k)$} \\
			&\circ\qquad \text{$ \tilde u_k\in C^\infty(B_{\rho_k}(0)\cap \tilde S_k(\tilde \Sigma_k))$}\\ 
			&\circ\qquad \text{$|\nabla^{\tilde S_k} \tilde u_k|=O(1) \sin\theta_k$ in $ B_{\rho_k}(0)\cap  {\tilde S}_k(\tilde\Sigma_k)$}\\ 
			&\circ\qquad \text{$|\nabla^{\tilde S_k} \tilde u_k|=-\tan\theta_k$ on $B_{\rho_k}(0)\cap \partial \tilde \Sigma_k $}
		\end{aligned}
	\end{equation*}
Let $L(\tilde S)\in C^\infty(\tilde S)$ be the constant function with value $0$. By Proposition \ref{convergence}  and Lemma \ref{stability lem}, there is a  generalized domain $\Omega$ in $\tilde S$ such that $(\tilde \Omega,g)$ is weakly stable with respect to $L(\tilde S)$ and $v\in C^\infty(\Omega)\cap C^{1,1}(\tilde \Omega)$ nonnegative that satisfies \eqref{v assumptions 1}, \eqref{v assumptions 2}, and \eqref{v assumptions 3} such that, for every $0<\alpha<1$,  $\tilde S_k(\tilde\Sigma_k)\setminus \partial \tilde \Sigma_k\to \Omega $ in $C_{loc}^{1,\alpha}$  and $(-\cot\theta_k)\,\tilde u_k\to v$  in $C^{1,\alpha}_{loc}(\tilde \Omega)$. By \eqref{curvature estimate rescaled 2} and Corollary \ref{bounded geometry}, $|k(\Omega)|_{L^\infty(\tilde \partial \Omega)}<\infty$.  According to Proposition \ref{half space}, $\Omega$ has either one or two components, each of which is congruent to a half-space, and, on each of these components, $v$ is an affine function. By Lemma  \ref{hodo graph}, 	$
\nabla^{\tilde S_k}\nabla^{\tilde S_k}\tilde u_k(\tilde y_k)=o(1)\sin\theta_k.
$ 
	This contradicts \eqref{curvature estimate rescaled 1}.

	Case 2: Assume that $\lambda_k\operatorname{dist}(y_k,\partial \Sigma_k)\to\infty$. 
	
Let $\rho_k>1$ be such that $2\,\rho_k\leq\delta_k\,\lambda_k$, $\rho_k=o(1)\,\lambda_k\operatorname{dist}(y_k,\partial \Sigma_k)$, and $\rho_k\to\infty$.
		Let $\tilde S_k=\lambda_k\,(S_k-y_k)$, $\tilde \Sigma_k=\lambda_k\,(\Sigma_k-y_k+u_k(y_k)\,\nu(S_k)(y_k))$, and $\tilde u_k\in C^{0,1}(\tilde S_k)$ be given by 
	$$
	\tilde u_k(y)=\lambda_k\,u_k\left(\lambda_k^{-1}\,y+y_k\right)-\lambda_k\,u_k(y_k).
	$$
	By \eqref{geometry bound S3}, rotating $\tilde S_k$ and $\tilde \Sigma_k$ if necessary, $\tilde S_k$ converges  to $\tilde S=\mathbb{R}^2\times\{0\}$  in $C_{loc}^{2,\alpha}$. By \eqref{MC bound S3},
\begin{align*} 
	\sup\{|H(\tilde S_k)(y)|:y\in B_{\rho_k}(0)\cap \tilde S_k\}&=o(1)\sin\theta_k\text{ and}\\
	\sup\{|\nabla^{\tilde S_k}H(\tilde S_k)(y)|:y\in B_{\rho_k}(0)\cap \tilde S_k\}&=o(1)\sin\theta_k.
\end{align*}
Note that $0\in \tilde \Sigma_k$ and  $B_{\rho_k}(0)\cap \partial \tilde \Sigma_k=\emptyset$. By  the definition of $\lambda_k$  and \eqref{estimate},
	\begin{align} \label{curvature estimate rescaled 3} 
		|\nabla^{\tilde S_k}\nabla^{\tilde S_k}\tilde u_k(0)|=\sin\theta_k
	\end{align} 
	while
	\begin{align*}
		\sup\{|\nabla^{\tilde S_k}\nabla^{\tilde S_k}\tilde u_k(y)|:y\in B_{\rho_k}(0)\cap \tilde S_k\}\leq 2\sin\theta_k.
	\end{align*}
	 Using \eqref{hessian estimate Sigma assumptions 0}, we see that the following hold.
	 \begin{equation*}
	 	\begin{aligned}
	 		&\circ\qquad \text{$ \{\Phi_{\tilde S_k}(\tilde u_k)(y):y\in B_{\rho_k}(0)\cap   {\tilde S}_k\}\subset \tilde \Sigma_k$}\hspace{5cm}\\ 			 	
	 		&\circ\qquad \text{$ \tilde u_k\in C^\infty(B_{\rho_k}(0)\cap \tilde S_k(\tilde \Sigma_k))$}\\ 
	 		&\circ\qquad \text{$|\nabla^{\tilde S_k} \tilde u_k|=O(1) \sin\theta_k$ in $ B_{\rho_k}(0)\cap  {\tilde S}_k(\tilde \Sigma_k)$}
	 	\end{aligned}
	 \end{equation*}

Arguing as in the proof of Proposition \ref{convergence}, we see that, passing to a subsequence, $(-\cot\theta_k)\,\tilde u_k$ converges in $C_{loc}^{2,\alpha}$ to a harmonic function $v\in C^\infty(\tilde S)$ with linear growth. By the Liouville theorem, $v$ is linear.  This contradicts \eqref{curvature estimate rescaled 3}.
	
	This completes the proof of the lemma.	
\end{proof}
	Given a support surface $S\subset \mathbb{R}^3$, a minimal capillary surface  $\Sigma\subset \mathbb{R}^3$ supported on $S$  with capillary angle $0<\theta<\pi$, and  $r>0$ such that $B_r(0)$ intersects $S$, $\Sigma$, and $\partial \Sigma$ transversely, we denote by  $I(S,\Sigma,r)>0$ the largest number such that 
$$
\frac{1}{2}\,(-\tan\theta)\operatorname{dist}_\Sigma(x,\partial \Sigma)\leq \operatorname{dist}(x,S)\leq 2\,(-\tan\theta)\operatorname{dist}_\Sigma(x,\partial \Sigma)
$$ 
for all $x\in B_r(0)\cap \Sigma $ with $\operatorname{dist}_\Sigma(x,\partial \Sigma)<I(S,\Sigma,r)$;
see also \cite[Lemma 4.13]{chodoshedelenli}.
Note that $I(S,\Sigma,r)$ is  nonincreasing  in $r$ and that $I(\lambda\,S,\lambda\,\Sigma,\lambda\,r)=\lambda\,I(S,\Sigma,r)$ for every $\lambda>0$. We extend the definition of $I(S,\Sigma,r)$ to all $r>0$ by requiring that $r\mapsto I(S,\Sigma,r)$  be upper semi-continuous. 
\begin{lem} \label{uniform graphicality}
Let  $r_0<r_1<r_2$ be such that $S_{r_1}(0)$ intersects $S_k$, $\Sigma_k$, and $\partial \Sigma_k$ transversely for all $k$. There holds $\liminf_{k\to\infty}I(S_k,\Sigma_k,r_1)>0.$
\end{lem}
\begin{proof}
	We follow the argument in \cite[Lemma 4.13]{chodoshedelenli}.

Since $\Sigma_k$ is properly embedded, there holds $I(S_k,\Sigma_k,r)>0$ for all $r>0$.
	
Let $r_1<\rho<r_2$. We claim that 
$$
\sup\left\{\frac{\rho-r}{I(S_k,\Sigma_k,r)}\,:r\in(0,\rho)\right\}=O(1).
$$

Suppose not. Passing to a subsequence, there are  $\rho_k\in(0,\rho)$ such that, for all $r\in(0,\rho)$, 
\begin{align} \label{rho k}
\frac{\rho-\rho_k}{I(S_k,\Sigma_k,\rho_k)}\geq \frac{\rho-r}{I(S_k,\Sigma_k,r)}.
\end{align}
Note that there is $x_k\in\Sigma_k$ with $|x_k|\leq \rho_k$  such that $\operatorname{dist}_{\Sigma_k}(x_k, \partial \Sigma_k)=I(S_k,\Sigma_k,\rho_k)$ and either
\begin{equation} \label{distance} 
\begin{aligned} 
2\,(-\tan\theta_k)\operatorname{dist}_{\Sigma_k}(x_k, \partial {\Sigma_k})&=\operatorname{dist}(x_k,S_k)\text{ or}\\
\qquad \frac12\,(-\tan\theta_k)	\operatorname{dist}_{\Sigma_k}(x_k, \partial {\Sigma_k})&=\operatorname{dist}(x_k,S_k).
\end{aligned} 
\end{equation} 
Let $$\lambda_k=\frac{1}{I(S_k,\Sigma_k,\rho_k)}\qquad\text{and}\qquad\delta_k=\frac12\left(\rho-\rho_k\right).$$ Note that $\lambda_k\to\infty$ and $\delta_k\,\lambda_k\to\infty$.  Let $\rho_k<r<\rho_k+\delta_k$. There holds
\begin{align*} 
\rho-r> \rho-\rho_k-\delta_k= \frac{1}{2}\left(\rho-\rho_k\right).
\end{align*} 
By \eqref{rho k},
\begin{align} \label{estimate 2} 
	2\,I(S_k,\Sigma_k,r)\geq\frac{\rho-\rho_k}{\rho-r}\,I(S_k,\Sigma_k,r)\geq I(S_k,\Sigma_k,\rho_k)= \frac1{\lambda_k}. 
\end{align} 
Let $z_k\in \partial \Sigma_k$ be such that $\sfrac1{\lambda_k}=\operatorname{dist}_{\Sigma_k}(x_k,\partial \Sigma_k)=\operatorname{dist}_{\Sigma_k}(x_k,z_k)$. Since $\sfrac1{\lambda_k}=o(1)\,\delta_k$, 
$$x_k\in B^{\Sigma_k}_{\sfrac{\delta_k}{2}}(z_k)\subset B^{\Sigma_k}_{\sfrac{3\,\delta_k}{4}}(x_k)\subset B_{\rho_k+\sfrac{3\,\delta_k}4}(0)\cap \Sigma_k.$$  

Let $$R_k=\frac12\,\delta_k\,\lambda_k.$$ Note that $R_k\to\infty$. We define $\tilde S_k=\sfrac{1}{\lambda_k}\,(S_k-z_k),$ $\tilde \Sigma_k=\sfrac{1}{\lambda_k}\,(\Sigma_k-z_k),$ and $\tilde x_k=\sfrac{1}{\lambda_k}\,(x_k-z_k)$. 	By \eqref{geometry bound S3}, $\tilde S_k$ converges  to a plane in $C_{loc}^{2,\alpha}$ and 
	\begin{align*} 
\sup\{|H(\tilde S_k)(y)|:y\in B_{\rho_k}(0)\cap \tilde S_k\}&=o(1)\sin\theta_k\text{ and}\\
\sup\{|\nabla^{\tilde S_k}H(\tilde S_k)(y)|:y\in B_{\rho_k}(0)\cap \tilde S_k\}&=o(1)\sin\theta_k.
\end{align*}
 By \eqref{estimate 2},  
$$
I(\tilde S_k,\tilde \Sigma_k,R_k)\geq\frac12.
$$
Moreover, there holds $0\in \partial \tilde \Sigma_k$ and
\begin{align} \label{distance 1} \operatorname{dist}_{\tilde \Sigma_k}(\tilde x_k,0)=1. 
	\end{align} By \eqref{distance}, either
	\begin{equation} \label{distance 2} 
		\begin{aligned} 
			2\,(-\tan\theta_k)\operatorname{dist}_{\tilde \Sigma_k}(\tilde x_k, \partial {\tilde \Sigma_k})&=\operatorname{dist}(\tilde x_k,\tilde S_k)\text{ or}\\
			\qquad \frac12\,(-\tan\theta_k)\operatorname{dist}_{\tilde \Sigma_k}(\tilde x_k, \partial {\tilde \Sigma_k})&=\operatorname{dist}(\tilde x_k,\tilde S_k).
		\end{aligned} 
	\end{equation} 

Applying Proposition \ref{graphicality 2} with increasing radii iteratively, we see that there are  $\tilde u_k\in C^{0,1}(\tilde S_k)$ nonnegative such that, for every  $R>1$ fixed such that $S_R(0)$ intersects $\tilde S_k$, $\tilde \Sigma_k$, and $\partial \tilde \Sigma_k$  transversely for all $k$, the following properties hold.
\begin{equation*} 
\begin{aligned}
		&\circ\qquad \{\Phi_{\tilde S_k}(\tilde u_k)(y):y\in  B_R(0)\cap {\tilde S}_k(\tilde \Sigma_k)\}\subset \tilde \Sigma_k\\ 
		&\circ\qquad \text{$\tilde u_k=0$ in $\tilde S_k\setminus \tilde S_k(\tilde \Sigma_k)$} \\ 
		&\circ\qquad \text{$ \tilde u_k\in C^\infty(B_R(0)\cap  \tilde S_k(\tilde\Sigma_k))$}\\ 
		&\circ\qquad \text{$|\nabla^{\tilde S_k} \tilde u_k|=O(1) \sin\theta_k$ in $ B_R(0)\cap\tilde S_k(\tilde\Sigma_k)$}\\ 
		&\circ\qquad \text{$|\nabla^{\tilde S_k} \tilde u_k|=-\tan\theta_k$ on $B_R(0)\cap \partial \tilde \Sigma_k $}
\end{aligned}
\end{equation*}
Moreover, applying Proposition \ref{graphicality 2} with $S=\tilde S_k-\tilde x$, $\Sigma=\tilde \Sigma_k-\tilde x$, and $r_0=2$ for different points $\tilde x\in B_R(0)\cap  \partial \tilde  \Sigma_k$, we see that, in fact, 
$$
\limsup_{R\to\infty} \limsup_{k\to\infty}\frac1{\sin\theta_k}\,\sup\{|\nabla^{\tilde S_k}\tilde u_k(\tilde y)|:\text{$\tilde y\in B_R(0)\cap \tilde S_k(\tilde \Sigma_k)$ and $\operatorname{dist}_{\tilde S_k}(\tilde y, \partial \tilde \Sigma_k)<2$} \}<\infty. 
$$
By Lemma \ref{prelim curvature estimates},
$
\text{$\nabla^{\tilde S_k}\nabla^{\tilde S_k} \tilde u_k=O(1)\sin\theta_k$}
$
in $B_R(0)\cap \tilde S_k(\tilde \Sigma_k)$ and 
$$
\limsup_{R\to\infty} \limsup_{k\to\infty}\frac1{\sin\theta_k}\,\sup\{|\nabla^{\tilde S_k}\nabla^{\tilde S_k}\tilde u_k(\tilde y)|:\text{$\tilde y\in B_R(0)\cap \tilde S_k(\tilde \Sigma_k)$ and $\operatorname{dist}_{\tilde S_k}(\tilde y, \partial \tilde \Sigma_k)<1$} \}<\infty. 
$$
In view of \eqref{distance 1} and \eqref{distance 2}, 
\begin{align} \label{distance contradiction 4} 
	2\,(-\tan\theta_k)\,\tilde u_k(\Pi_{\tilde S_k}(\tilde x_k))=1\qquad\text{or}\qquad \frac12\,(-\tan\theta_k)\,\tilde u_k(\Pi_{\tilde S_k}(\tilde x_k))=1.
\end{align} 
By contrast, passing to a subsequence as in the proof of Lemma \ref{prelim curvature estimates}, rotating each $\tilde S_k$ if necessary, the connected component of $\tilde S_k(\tilde \Sigma_k)$ that contains $\Pi_{\tilde S_k}(\tilde x_k)$ converges  in $C_{loc}^{1,\alpha}$ to the generalized domain $\Omega=\{\tilde y\in \mathbb{R}^2:\tilde y\cdot e_1>0\}\times \{0\}$ and $(-\cot\theta_k)\,\tilde u_k(\Pi_{\tilde S_k}(\tilde x_k))$ converges  in $C_{loc}^{1,\alpha}(\tilde \Omega)$ to $v\in C^\infty(\Omega)$ given by $v(\tilde y)=\tilde y\cdot e_1$. By \eqref{distance 1}, we conclude that $(-\cot\theta_k)\,\tilde u_k(\Pi_{\tilde S_k}(\tilde x_k))\to 1$. As this is incompatible with  \eqref{distance contradiction 4}, the assertion follows.
\end{proof}

	\begin{prop} \label{curvature estimates}
Assume that $0\in \partial \Sigma_k$ for every $k$. Then, for every $k$ sufficiently large, there is  $ u_k\in C^{0,1}(S_k)$ nonnegative with the following properties.
\begin{equation*}
	\begin{aligned}
		&\circ\qquad \text{$ \{\Phi_{S_k}(u_k)(y):y\in B_{r_0}(0)\cap S_k(\Sigma_k)\}$}\subset \Sigma_k\hspace{5cm}\\ 		
		&\circ\qquad \text{$u_k=0$ in $S_k\setminus S_k(\Sigma_k)$} \\
		&\circ\qquad \text{$ u_k\in C^\infty(B_{r_0} (0)\cap S_k(\Sigma_k))$}\\ 
		&\circ\qquad \text{$|\nabla^{S_k} u_k|=O(1) \sin\theta_k$ in $ B_{r_0}(0)\cap S_k(\Sigma_k)$}\\
		&\circ\qquad \text{$|\nabla^{S_k}\nabla^{S_k} u_k|=O(1)\sin\theta_k$ in $B_{r_0}(0)\cap  S_k(\Sigma_k)$}\\ 
		&\circ\qquad \text{$|\nabla^{S_k} u_k|=-\tan\theta_k$ on $B_{r_0}(0)\cap \partial \Sigma_k $}
	\end{aligned}
\end{equation*}
\end{prop}
\begin{proof}
This follows from Proposition \ref{graphicality 2}, Lemma \ref{prelim curvature estimates}, and Lemma \ref{uniform graphicality}.

\end{proof}
\begin{rema} \label{counter 3}
The assumption \eqref{geometry bound S3} cannot be dropped in Proposition \ref{curvature estimates}. To see this, let $S=\{(y,t)\in \mathbb{R}^2\times \mathbb{R}:|y|=\cosh t\}$ be the catenoid, $\sfrac{\pi}2<\theta_k<\pi$ angles such that $\theta_k\nearrow \pi$, and $S_k= \sin\theta_k\,S.$ Then $H(S_k)=0$. Moreover, the surface  $$\Sigma_k= \{y\in \mathbb{R}^2:|y|\leq1\}\times\{\sin(\theta_k)\operatorname{artanh}(-\cos\theta_k)\}$$  is a minimal capillary surface supported on $S_k$ with capillary angle $\theta_k$ that is stable for the free energy. There is no $u_k\in C^{0,1}(S_k)$ such that 
$$
\{\Phi_{S_k}(u_k)(y):y\in B_{1}(0)\cap S_k(\Sigma_k)\}\subset \Sigma_k
$$
for any $k$.
\end{rema}
\begin{rema} \label{counter 4}
	The assumption \eqref{MC bound S3} cannot be dropped in Proposition \ref{curvature estimates}. To see this, let $\sfrac{\pi}2<\theta_k,\,\tilde \theta_k<\pi$ be such that $\theta_k\nearrow \pi$, $\tilde \theta_k\nearrow \pi,$ and $\sin\theta_k=o(1)\,\sin\tilde \theta_k$. Let $S=\{x\in \mathbb{R}^3:|x|=1\}$ be the unit sphere and  $S_k=\sfrac{1}{ \sin\tilde\theta_k}\,S.$ Note that $|h(S_k)|=O(1)\,\sin\tilde \theta_k$. The surface $$\Sigma_k=\sfrac{1}{ \sin\tilde\theta_k}\{x\in \mathbb{R}^3:|x|\leq 1\text{ and }e_3\cdot x=\cos\theta_k\}$$ is a minimal capillary surface supported on $S_k$ with capillary angle $\theta_k$ that is weakly stable for the free energy. Let $u_k\in C^\infty( S_k(\Sigma_k))$ be such that $\Sigma_k=\{\Phi_{S_k}(u_k)(y):y\in  S_k(\Sigma_k)\}.$ Then
	$$
	\limsup_{k\to\infty}\frac{1}{\sin\theta_k}\sup\{|u_k(y)|:y\in S_k(\Sigma_k)\}=\infty. 
	$$
\end{rema}
\section{Blowdown limits of stable minimal capillary surfaces}
	Recall from Appendix \ref{appendix support surfaces} the definition of a support surface $S\subset \mathbb{R}^3$ with unit normal $\nu(S)$  and of the domain $D(S)\subset \mathbb{R}^3$ bounded by $S$. Moreover, recall  our conventions for the second fundamental form $h(S)$ and the shape operator $A(S)$ of $S$. Given $\psi \in C^\infty(S)$, recall the definition \eqref{Phi S} of $\Phi_S(\psi)$.

Recall from Appendix \ref{appendix minimal capillary surfaces}  the  definition of a minimal capillary surface $\Sigma \subset \mathbb{R}^3$ supported on a support surface $S\subset \mathbb{R}^3$  and that of its  wetting surface $S(\Sigma)\subset S$. In particular, recall that  $\Sigma\setminus \partial \Sigma \subset D(S)$.  Moreover, recall the concepts of  stability and weak stability for the free energy  of a minimal capillary surface.

Given $U\subset \mathbb{R}^3$ open, recall from Appendix \ref{appendix: generalized domain} the definition of a generalized domain $\Omega$ in $U\cap S$ and of its associated Riemannian 2-manifold $(\tilde \Omega,g)$. Given $L(S)\in C^\infty(S)$, recall the concepts of  stability and weak stability of a generalized domain  with respect to $L(S)$.  Moreover, recall the definitions of convergence of a sequence of domains to a generalized domain and of convergence of a sequence of functions defined on the domains of such a sequence to a function defined on the limiting generalized domain.

Recall from Appendix \ref{appendix: cemswft} the definition of a complete embedded minimal surface $S\subset \mathbb{R}^3$ with finite total curvature and the concept of such a surface being horizontal. Moreover, recall the definitions of an end of $S$ and of an intermediate layer of $S$ as well as the concepts of an  intermediate layer  having logarithmic height or bounded height.

	In this section,  we consider  a complete embedded minimal surface $S\subset \mathbb{R}^3$ with finite total curvature that is horizontal and  not an affine plane. Moreover, we consider compact minimal capillary surfaces $\Sigma_k$ supported on $S$ with capillary angle $\sfrac{\pi}2<\theta_k<\pi$.
	
	We assume that each $\Sigma_k$ is stable for the free energy and that $\theta_k\nearrow  \pi$.

\begin{lem} \label{disk}
	 $\Sigma_k$ is diffeomorphic to a union of disks.
\end{lem}
\begin{proof}
	This is \cite[Lemma 49]{eichmair2024penrose}.
\end{proof}
The outer radius $r(\Sigma_k)>0$ of $\Sigma_k$ is the quantity 
\begin{align} \label{area radius} 
r(\Sigma_k)=\sup\{|x|:x\in \Sigma_k\}.
\end{align}

	In this section, we show that  $\sin(\theta_k)\,\Sigma_k$ is close to a centered round disk provided that $k$ is sufficiently large.

\begin{lem} \label{divergence}
	There holds $r(\Sigma_k)\to \infty$.

\end{lem} 
\begin{proof}
	Suppose that, passing to a subsequence, $r(\Sigma_k)=O(1)$.
By Proposition \ref{curvature estimates},  $S(\Sigma_k)$ and $\Sigma_k$ are diffeomorphic for all $k$ sufficiently large. Moreover, there are 	$ u_k\in C^{0,1}(S)\cap C^\infty(S(\Sigma_k))$ nonnegative such that the following hold.
	\begin{equation*}
		\begin{aligned}
			&\circ\qquad \text{$ \{\Phi_{S}(u_k)(y):y\in S(\Sigma_k)\}$}= \Sigma_k\hspace{5cm}\\ 		
			&\circ\qquad \text{$u_k=0$ in $S\setminus S(\Sigma_k)$} \\
			&\circ\qquad \text{$|\nabla^{S} u_k|=O(1) \sin\theta_k$ in $ S(\Sigma_k)$}\\ 
						&\circ\qquad \text{$|\nabla^{S}\nabla^{S} u_k|=O(1)\sin\theta_k$ in $ S(\Sigma_k)$}\\
			&\circ\qquad \text{$|\nabla^{S} u_k|=-\tan\theta_k$ on $ \partial \Sigma_k $}
		\end{aligned}
	\end{equation*} 
	Let $L(S)\in C^\infty(S)$ be the constant function $0$ and $0<\alpha<1$. By Proposition \ref{convergence} and Lemma \ref{stability lem}, $S(\Sigma_k)\setminus \partial \Sigma_k$ converges in $C_{loc}^{1,\alpha}$  to a bounded generalized domain $\Omega$ in $S$ such that $(\tilde \Omega,g)$  is stable with respect to $L(S)$ and $v_k=(-\cot\theta_k)\,u_k$ converges  in $C_{loc}^{1,\alpha}(\tilde \Omega)$ to a function $v\in C^\infty(\Omega)\cap C^{1,1}(\tilde \Omega)$ that satisfies \eqref{PDE}, \eqref{v}, and \eqref{grad v}. This is incompatible with Proposition \ref{no bounded components}. 
\end{proof}

\begin{lem} \label{sandwich}
	Assume that $S$ has an intermediate layer $N\subset \mathbb{R}^3$ and that there are $x_k\in N\cap\Sigma_k$ such that $\nu(\Sigma_k)(x_k)\cdot e_3=0$. Then $x_k=O(1)$.
\end{lem} 
\begin{proof}
Suppose that, passing to a subsequence, $|x_k|\to \infty$.

Assume first that $N$ has logarithmic height.
	
	 Since $N$ has logarithmic height, there is   $b>0$ such that, for every $\upsilon\in \mathbb{R}^3$ with $|\upsilon|=1$ and $\upsilon\cdot e_3>\sfrac12$, there holds  $x_k+t_k\,\upsilon\in S$ for some $t_k\in (0,b\,\log |x_k|)$ depending on $\upsilon$. 	
	 Using Lemma \ref{lem:stable -> local graphicality} and that $S\cap \Sigma_k=\partial \Sigma_k$, we conclude that $\operatorname{dist}_{\Sigma_k}(x_k,\partial \Sigma_k)=O(1)\,\log |x_k|$.
	 
	  Let $z_k\in \partial \Sigma_k$ be such that $\operatorname{dist}_{\Sigma_k}(x_k,z_k)=\operatorname{dist}_{\Sigma_k}(x_k,\partial \Sigma_k)$. Then
	 $
	|x_k|=(1+o(1))\,|z_k|.
	 $ 
	 Let  $\tilde S_k=\sfrac1{\lambda_k}\,(S-z_k)$, $\tilde \Sigma_k=\sfrac{1}{\lambda_k}\,(\Sigma_k-z_k),$ $\tilde x_k=\sfrac1{\lambda_k}\,(x_k-z_k)$, and $\tilde N_k=\sfrac1{\lambda_k}\,( N-z_k)$ where $\lambda_k=\log|z_k|$. Then $ \operatorname{cl} \tilde N_k\cap \tilde S_k$ converges locally smoothly to the union of two disjoint planes. Moreover,  $\operatorname{dist}_{\tilde \Sigma_k}(\tilde x_k,0)=1+o(1)$. Using Proposition \ref{curvature estimates}, we conclude that $\nu(\tilde \Sigma_k)(\tilde x_k)=\nu(\tilde \Sigma_k)(0)+o(1)$. Note that $\nu(\tilde \Sigma_k)(\tilde x_k)=\nu(\Sigma_k)(x_k)$ and $\nu(\tilde \Sigma_k)(0)=\nu(\Sigma_k)(z_k)$. Using that $\theta_k\nearrow \pi$ and $|z_k|\to \infty$, we see that  either $\nu(\Sigma_k)(z_k)=e_3+o(1)$ or $\nu(\Sigma_k)(z_k)=-e_3+o(1)$. This is a contradiction.
	 
The argument in the case where  $N$ has bounded height is similar and only requires formal modifications. 	 
	 
	 This completes the proof of the lemma.
\end{proof}

Let $S^1,\,\ldots,\, S^m$ be the ends of $S$ labeled as in Appendix \ref{appendix: cemswft}. Since $S$ is not an affine plane, $m\geq 2$.

Let $D=\{y\in \mathbb{R}^2:|y|\leq 1\}$.

\begin{prop} \label{blowdown}
	Assume that $m$ is even. Then, passing to a subsequence, one of the following two alternatives holds.
	\begin{itemize}
		\item[$\circ$]$\Sigma_k$ is connected and either  $\partial \Sigma_k\subset S^1$ for all $k$  or  $\partial \Sigma_k\subset S^m$ for all $k$. Moreover, $\sfrac{1}{r(\Sigma_k)}\,\Sigma_k$ converges smoothly to $D\times \{0\}$.
		\item[$\circ$]$\Sigma_k$ has exactly two components $\Sigma^1_k$ and $\Sigma^m_k$ and there holds $\partial \Sigma^1_k\subset S^1$ and $\partial \Sigma^m_k\subset S^m$  for all $k$. Moreover, both $\sfrac{1}{r(\Sigma^1_k)}\,\Sigma^1_k$ and $\sfrac{1}{r(\Sigma^m_k)}\,\Sigma^m_k$ converge smoothly to $D\times \{0\}$.
		\end{itemize}

\end{prop} 
\begin{proof}

 By Lemma \ref{divergence}, $r(\Sigma_k)\to \infty$. Let  $x^{1}_{k}\in \Sigma_k$ be such that $|x^{1}_{k}|=r(\Sigma_k)$. By the maximum principle, $$x^{1}_k\in \partial \Sigma_k.$$
Let $\Sigma^1_k\subset \Sigma_k$ be the component of $\Sigma_k$ that contains $x^{1}_k$. Note that $r(\Sigma^1_k)=r(\Sigma_k)$.

Assume first that $x^{1}_k\in S^1$ for all $k$. Let $\tilde S_k=\sfrac1{r(\Sigma_k)}\,S$, $\tilde \Sigma_k=\sfrac{1}{r(\Sigma_k)}\,\Sigma_k$, and $\tilde S^1_k=\sfrac1{r(\Sigma_k)}\,S^1$. Note that $\tilde S^1_k$ converges  to $\tilde S=\mathbb{R}^2\times \{0\}$ locally smoothly in $\mathbb{R}^3\setminus \{0\}$. By Proposition \ref{curvature estimates}, there are
 concentric spherical shells $U_\ell\subset \mathbb{R}^3$, $\ell \geq 1$, of increasing outer radii and decreasing inner radii with 
 $$
 \mathbb{R}^3\setminus\{0\}=\bigcup_{\ell=1}^\infty U_\ell
 $$
and 	$\tilde u_k\in C^{0,1}(\tilde S_k)$ nonnegative  such that, for each fixed $\ell \geq 1$,  the following hold for all $k$  sufficiently large.
\begin{align*}
	&\circ\qquad  \text{$\{\Phi_{\tilde S_k}(\tilde u_k)(y):y\in U_{\ell}\cap \tilde S_k(\tilde \Sigma_k)\cap \tilde S^1_k\}\subset U_{\ell+1}\cap \tilde \Sigma_k$}\hspace{3cm}\\ 		
		&\circ\qquad \text{$\tilde u_k=0$ in $\tilde S_k\setminus \tilde S_k(\tilde\Sigma_k)$} \\
			&\circ\qquad \text{$\tilde  u_k\in C^\infty(U_{\ell}\cap  \tilde S_k(\tilde \Sigma_k)\cap \tilde S^1_k)$}\\ 
	&\circ\qquad \text{$|\nabla^{\tilde S_k} \tilde u_k|=O(1) \sin\theta_k$ in $ U_{\ell}\cap \tilde S_k(\tilde \Sigma_k)\cap \tilde S^1_k$}\\ 
	&\circ\qquad \text{$|\nabla^{\tilde S_k}\nabla^{\tilde S_k} \tilde u_k|=O(1)\sin\theta_k$ in $U_{\ell}\cap   \tilde S_k(\tilde \Sigma_k)\cap \tilde S^1_k$}\\ 
	&\circ\qquad \text{$|\nabla^{\tilde S_k}\tilde u_k|=-\tan\theta_k$ on $U_{\ell}\cap \partial \tilde \Sigma_k\cap \tilde S^1_k$}
\end{align*}
Let $L(\tilde S)\in C^\infty(\tilde S)$ be the constant function with value $0$. 
	By Proposition \ref{convergence}, $\tilde S(\tilde \Sigma_k)$ converges in $C_{loc}^{1,\alpha}$ in $\mathbb{R}^3\setminus\{0\})$ to a bounded generalized domain $\Omega$ in $\tilde S\cap \mathbb{R}^3\setminus \{0\}$ and $v_k=(-\cot\theta_k)\,\tilde u_k$ converges  in $C_{loc}^{1,\alpha }(\tilde \Omega)$ to $v\in C^\infty(\Omega)\cap C^{1,1}(\tilde \Omega)$ nonnegative satisfying \eqref{PDE}, \eqref{v}, and \eqref{grad v}. By Proposition \ref{angst},  $\Omega=D\times\{0\}$. In particular, $\partial \Omega \Subset \mathbb{R}^3\setminus\{0\}$. By Lemma \ref{hodo graph}, we see that $\sfrac1{r(\Sigma_k)}\,\Sigma^1_k$ converges in $C^{2,\alpha}_{loc}$ in $\mathbb{R}^3\setminus \{0\}$  to $D\times\{0\}$. 
	 Using that, by Lemma \ref{disk}, $\Sigma^1_k$ is a disk and Lemma \ref{lem: stable implies interior curvature estimates}, we see that both the second fundamental form and the area of $B_{\sfrac12}(0)\cap \sfrac{1}{r(\Sigma_k)}\,\Sigma^1_k$ are bounded independently of $k$. It follows that  $\sfrac{1}{r(\Sigma_k)}\,\Sigma^1_k$ converges  in $C^{1,\alpha}_{loc}(\mathbb{R}^3)$ to $D\times\{0\}$. By elliptic regularity and bootstrapping,  $\sfrac{1}{r(\Sigma_k)}\,\Sigma^1_k$ converges smoothly to $D\times\{0\}$.
	
	Similar considerations apply in the case where $x^{1}_k\in S^m$ for all $k$. 
	
	Suppose, for a contradiction, that, passing to a subsequence, $x^{1}_k\in S^i$ for some $1<i<m$. It follows that $S^i$ is contained in the boundary of an intermediate layer $N$. Let $1<j< m$ be such that $j\neq i$ and $\partial N\cap S^j\neq \emptyset$.  Given $\rho>1,$ let $\gamma_k:[0,T]\to \Sigma_k$ be a smooth curve with the following properties. 
\begin{align*} 
	&\circ\qquad  \text{$\gamma_k(0)\in S^i\cap \partial \Sigma_k$}\hspace{5cm}\\
	&\circ\qquad	\text{$|\gamma_k(0)|\geq \rho$}\\
	&\circ\qquad	\text{$\gamma_k(T)\in S^{j}\cap \partial \Sigma_k$}\\
	&\circ\qquad	\text{$|\gamma_k(T)|\geq \rho$}
\end{align*} 
	 Since $\theta_k\nearrow  \pi$,  $e_3\cdot\nu(\Sigma_k)(\gamma_k(0))$ and $e_3\cdot \nu(\Sigma_k)(\gamma_k(T))$ have opposite signs for sufficiently large $\rho$ and $k$. By continuity, there is $t\in(0,T)$ with $e_3\cdot \nu(\Sigma_k)(\gamma_k(t))=0$. Using Lemma \ref{sandwich}, we see that there is $\rho_0>0$  such that each $\gamma_k$ intersects $B_{\rho_0}(0)$ provided that $\rho$ and $k$ are sufficiently large. Thus, we can repeat the blowdown argument from the previous paragraph, essentially ignoring the presence of the second end $S^j$. It follows that, passing to a subsequence, $\sfrac{1}{r(\Sigma_k)}\,\partial \Sigma^1_k$ converges to $D\times\{0\}$ smoothly. By Lemma \ref{disk}, using  that $S^i$ is contained in the boundary of an intermediate layer, it follows that $S$ and $\Sigma_k$ intersect transversely away from $\partial \Sigma_k$. This is a contradiction. 
	 
	 It follows that, passing to a subsequence, either $\Sigma^1_k\subset S^1$ for all $k$ or $\Sigma^1_k\subset S^m$ for all $k$. We may assume that the former alternative occurs.
	 
	Assume that $\Sigma^m_k=\Sigma_k\setminus \Sigma^1_k$ is nonempty. Let $x^m_k\in \Sigma^{m}_k$ be such that $|x^m_k|=r(\Sigma^m_k)$ and $\Sigma'_k\subset \Sigma^m_k$ the component of $\Sigma^m_k$ that contains $x^m_k$. Note that $r(\Sigma'_k)=r(\Sigma^m_k)$. 
	Since $\sfrac{1}{r(\Sigma_k)}\,\Sigma^1_k$ converges smoothly to $D\times \{0\}$, $\Sigma^1_k$ is separating in $D(S)$ provided that $k$ is sufficiently large. It follows that $\Sigma^m_k$ is  admissible for all large $k$. We may therefore repeat the previous argument with $\Sigma^m_k$ in place of $\Sigma_k$. This shows that, passing to a subsequence, either $\partial  \Sigma'_k\subset S^1$ for all $k$ or $\partial  \Sigma'_k\subset S^m$ for all $k$ and that $\sfrac{1}{r(\Sigma^m_k)}\,\Sigma'_k$ converges smoothly to $D\times \{0\}$. If $\partial \Sigma'_k \subset S^1$ for all $k$, using that $\Sigma_k$ is admissible, we see that  $\nu(\Sigma_k)(x^{1}_k)$ and $\nu(\Sigma_k)(x^{m}_k)$ point in opposite directions, while $\nu(S)(x^{1}_k)$ and $\nu(S)(x^{m}_k)$ point in the same direction for all $k$ sufficiently large. This is incompatible with $\nu(\Sigma_k)\cdot \nu(S)=\cos\theta_k$.  It follows that $\partial  \Sigma^m_k\subset S^m$ for all $k$. Repeating this argument, we see that $ \Sigma^m_k=\Sigma'_k$. This completes the proof of the proposition. 
\end{proof}
\begin{prop} \label{blowdown odd}
	
	Assume that $m$ is odd. For all $k$ sufficiently large, $\Sigma_k$ is connected and $\partial \Sigma_k\subset S^1$. Moreover, $\sfrac{1}{r_k}\,\Sigma_k$ converges smoothly to $D\times \{0\}$.

\end{prop} 
\begin{proof}
This follows  as in the proof of Proposition \ref{blowdown}. The only difference is that $S^m$ is  contained in  the boundary of an intermediate layer. This implies that $\Sigma^1_k\subset S^1$ and  $\Sigma^m_k=\emptyset$ for all $k$ sufficiently large.	 
\end{proof} 

Recall from Appendix \ref{appendix: cemswft} the definition of $a^1,\,\ldots,\,a^m\in \mathbb{R}$ and $b^1,\,\ldots,\,b^m\in \mathbb{R}$.  
\begin{coro} \label{sin vs r}
In the setting of Proposition \ref{blowdown}, the following hold.

 If $\Sigma_k$ is connected, then 
\begin{align*} 
&\circ\qquad  \text{$\sin(\theta_k)\,r(\Sigma_k)=b^1+o(1)$ and $x\cdot e_3=b^1\,\log r(\Sigma_k)$ in $\Sigma_k$ if $\Sigma_k\subset S^1$ and } \\
&\circ \qquad \text{$\sin(\theta_k)\,r(\Sigma_k)=-b^m+o(1)$ and $x\cdot e_3=b^m\,\log r(\Sigma_k)$ in $\Sigma_k$ if $\Sigma_k\subset S^m$.}
 \end{align*} 

 If $\Sigma_k$ has two components $\Sigma^1_k\subset S^1$ and $\Sigma^m_k\subset S^m$, then
\begin{align*} 
	&\circ\qquad  \text{$\sin(\theta_k)\,r(\Sigma^1_k)=b^1+o(1)$ and $x\cdot e_3=b^1\,\log r(\Sigma^1_k)$ in $\Sigma^1_k$ and } \\
	&\circ \qquad \text{$\sin(\theta_k)\,r(\Sigma^m_k)=-b^m+o(1)$ and $x\cdot e_3=b^m\,\log r(\Sigma^m_k)$ in $\Sigma^m_k$.}
\end{align*} 

\end{coro}
\begin{proof}
	We may assume that $\Sigma_k$ is connected, that $\Sigma_k\subset S^1,$ and that  $a^1=0$.
	
	By Proposition \ref{blowdown},
	$$
\sup\left \{\left|\frac{|x-(e_3\cdot x)\,e_3|}{r(\Sigma_k)}-1\right|:x\in \partial \Sigma_k^1\right \}=o(1)
	$$
as $k\to\infty,$ so that
$$
e_3\cdot x=b^1\,\log r(\Sigma_k)+o(1)
$$
uniformly on $\partial \Sigma_k$. By the convex hull property, 
\begin{align} \label{height}  
	e_3\cdot x=b^1\,\log r(\Sigma_k)+o(1)
\end{align} 
uniformly in $\Sigma_k$. Let $x_k\in \Sigma_k$ be such that $|x_k|=\sfrac{r(\Sigma_k)}2$. Recall the characterization of $v$ in Proposition \ref{angst}. The  proof of Proposition \ref{blowdown}   shows that  there are $y_k\in S$ with  $|y_k|=(1+o(1))\,|x_k|$ such that
$$
e_3\cdot x_k=b^1\,\log|y_k|+(1+o(1))\sin(\theta_k)\,r(\Sigma_k)\,\log\frac{|y_k|}{r(\Sigma_k)}\,e_3\cdot \nu(S)(y_k).
$$
Using \eqref{height}, that $2\,|y_k|=r(\Sigma_k)+o(1)$, and that $e_3\cdot \nu(S)(y_k)=-1+o(1)$, we conclude that 
$$
b^1+o(1)=(1+o(1))\sin(\theta_k)\,r(\Sigma_k),
$$
as asserted.

\end{proof}
\begin{coro} \label{sin vs r odd}
	In the setting of Proposition \ref{blowdown odd}, there holds $\sin(\theta_k)\,r(\Sigma_k)=b^1+o(1)$ and $x\cdot e_3=b^1\,\log r(\Sigma_k)$ in $\Sigma_k$.
\end{coro}
\begin{proof}
	This follows exactly as in the proof of Corollary \ref{sin vs r}.
\end{proof}

\section{Blowup limits of weakly stable minimal capillary surfaces}
	Recall from Appendix \ref{appendix support surfaces} the definition of a support surface $S\subset \mathbb{R}^3$ with unit normal $\nu(S)$  and of the domain $D(S)\subset \mathbb{R}^3$ bounded by $S$.  Recall  our conventions for the second fundamental form $h(S)$, the shape operator $A(S)$, and the mean curvature $H(S)$ of $S$. Given $\psi \in C^\infty(S)$, recall the definition \eqref{Phi S} of $\Phi_S(\psi)$.

Recall from Appendix \ref{appendix minimal capillary surfaces}  the  definition of a minimal capillary surface $\Sigma \subset \mathbb{R}^3$ supported on a support surface $S\subset \mathbb{R}^3$  and that of its  wetting surface $S(\Sigma)\subset S$. In particular, recall that  $\Sigma\setminus \partial \Sigma \subset D(S)$.  Moreover, recall the concepts of  stability and weak stability for the free energy  of a minimal capillary surface.

Given $U\subset \mathbb{R}^3$ open, recall from Appendix \ref{appendix: generalized domain} the definition of a generalized domain $\Omega$ in $U\cap S$ and of its associated Riemannian 2-manifold $(\tilde \Omega,g)$. Given $L(S)\in C^\infty(S)$, recall the concepts of  stability and weak stability of a generalized domain  with respect to $L(S)$.  Moreover, recall the definitions of convergence of a sequence of domains to a generalized domain and of convergence of a sequence of functions defined on the domains of such a sequence to a function defined on the limiting generalized domain.

In this section, we consider a closed surface $S\subset \mathbb{R}^3$ with positive mean curvature and a sequence of compact minimal capillary surfaces $\Sigma_k$ supported on $S$ with capillary angles $\sfrac{\pi}2<\theta_k<\pi$.

We assume  that each $\Sigma_k$ is weakly stable for the free energy and that $\theta_k\nearrow \pi$.

In this section, we show that, passing to a subsequence, $\sfrac1{\sin\theta_k}\,\Sigma_k$ is close to a round disk provided that $k$ is sufficiently large. 

Let $x_k\in \partial\Sigma_k$. Passing to a subsequence, we may assume that $x_k\to z\in S$.

\begin{prop} \label{blowup} For each $k$ sufficiently large, $\Sigma_k$ is connected. Moreover, passing to a subsequence, $\sfrac{1}{\sin\theta_k}\,(\Sigma_k-x_k)$ converges smoothly to a round disk of radius $\sfrac{2}{H(S)(z)}$.
\end{prop}
\begin{proof}
Let $\tilde S_k=\sfrac1{\sin\theta_k}\,(S-x_k)$ and $\tilde \Sigma_k= \sfrac1{\sin\theta_k}\,(\Sigma_k-x_k)$. Note that $\tilde S_k$ converges locally smoothly to a  plane $\tilde S$ through $0$. We may and will assume that $\tilde S=\mathbb{R}^2\times \{0\}$. Note that $H(\tilde S_k)=O(1)\sin\theta_k$ and $\nabla^{\tilde S_k}H(\tilde S_k)=O(1)(\sin\theta_k)^2$. By Proposition \ref{curvature estimates}, 
	 there are
	concentric balls $U_\ell\subset \mathbb{R}^3$, $\ell \geq 1$, of increasing radii with 
	$$
	\mathbb{R}^3=\bigcup_{\ell=1}^\infty U_\ell
	$$
	and 	$\tilde u_k\in C^{0,1}(\tilde S_k)$  such that, for each fixed $\ell \geq 1$,  the following hold for all $k$  sufficiently large.
	\begin{align*}
		&\circ\qquad  \text{$\{\Phi_{\tilde S_k}(\tilde u_k)(y):y\in U_{\ell}\cap  \tilde S_k(\tilde \Sigma_k)\}\subset U_{\ell+1}\cap \tilde \Sigma_k$}\hspace{3cm}\\ 		
		&\circ\qquad \text{$\tilde u_k=0$ in $\tilde S_k\setminus \tilde S_k(\tilde \Sigma_k)$} \\
		&\circ\qquad \text{$\tilde  u_k\in C^\infty(U_{\ell}\cap \tilde S_k(\tilde \Sigma_k))$}\\ 
		&\circ\qquad \text{$|\nabla^{\tilde S_k}\tilde  u_k|=O(1) \sin\theta_k$ in $ U_{\ell}\cap \tilde S_k(\tilde \Sigma_k)$}\\ 
		&\circ\qquad \text{$|\nabla^{\tilde S_k}\nabla^{\tilde S_k} \tilde u_k|=O(1)\sin\theta_k$ in $U_{\ell}\cap   \tilde S_k(\tilde \Sigma_k)$}\\ 
		&\circ\qquad \text{$|\nabla^{\tilde S_k} \tilde u_k|=-\tan\theta_k$ on $U_{\ell}\cap \partial \tilde \Sigma_k$}
	\end{align*}
	Let $L(\tilde S)\in C^\infty(\tilde S)$ be the constant function with value $H(S)(z)$ and $0<\alpha<1$. 	By Proposition \ref{convergence} and Lemma \ref{weak stability lem}, passing to a subsequence, $\tilde S_k(\tilde \Sigma_k)$ converges in $C_{loc}^{1,\alpha}$  in $\mathbb{R}^3$ to a   generalized domain $\Omega$ in $\tilde S$ such that $(\tilde \Omega,g)$ is weakly stable with respect to $L(\tilde S)$. Moreover, $v_k=(-\cot\theta_k)\,\tilde u_k$ converges  in $C_{loc}^{1,\alpha}(\tilde \Omega)$ to $v\in C^\infty(\Omega)\cap C^{1,1}(\tilde \Omega)$ satisfying \eqref{PDE}, \eqref{v}, and \eqref{grad v}. By Proposition \ref{positive curvature limit},  $\Omega$ is congruent to a disk of radius $\sfrac{2}{H(S)(z)}$. By Lemma \ref{hodo graph} and bootstrapping, we see that $\tilde \Sigma_k$ converges  smoothly to $ \Omega$.
	
	Suppose, for a contradiction, that $\Sigma_k$ has a second component. By the argument in the previous paragraph, this component must be far away from $z$ on the scale  of $\sin\theta_k$. An  argument analogous to that given in the proof of Proposition \ref{positive curvature limit} shows that $\Sigma_k$ is not weakly stable. This contradicts our assumption.
\end{proof}
\begin{coro} \label{blowup coro}
	In the setting of the proof of Proposition \ref{blowup}, there are
	\begin{itemize}
		\item[$\circ$] $\tilde y_k\in \tilde S_k$ such that $\tilde S_k(\tilde \Sigma_k)$ is smoothly close to $$\left\{\tilde y_k+p+s\,\nu(\tilde S_k)(\tilde y_k):p\in T_{\tilde y_k}\tilde S_k, \, |p|\leq \frac{2}{H(S)(z)},\, s\in[-2,2]\right\}\cap \tilde S_k \text{ and}$$ 
		\item[$\circ$] $\tilde w_k\in C^\infty(\tilde S_k(\tilde \Sigma_k))$ such that 
		\begin{align*} 
			\frac{	\tilde u_{k}(y)}{\sin\theta_k}=\frac{1}{H(S)(z)}-\frac{H(S)(z)}{4}\,|y-\tilde y_k|^2+\tilde w_k(y)
		\end{align*} 
		and	  $|\tilde w_k|_{C^{\ell,\alpha}(\tilde S_k(\tilde \Sigma_k))}=o(1)$ for every $\ell\geq 2$ and $0<\alpha<1$.
	\end{itemize}

\end{coro}
\begin{proof}
	This follows from the proof of Proposition \ref{blowup}, using the characterization of $v$ in Proposition \ref{positive curvature limit}, Lemma \ref{hodo graph}, and bootstrapping.
\end{proof}
\section{Linear analysis on an asymptotically catenoidal support surface}
	Recall from Appendix \ref{appendix support surfaces} the definition of a support surface $S\subset \mathbb{R}^3$ with unit normal $\nu(S)$  and of the domain $D(S)\subset \mathbb{R}^3$ bounded by $S$.  Recall  our conventions for the second fundamental form $h(S)$, the shape operator $A(S)$, and the mean curvature $H(S)$ of $S$.

Recall from Appendix \ref{appendix minimal capillary surfaces}  the  definition of a minimal capillary surface $\Sigma \subset \mathbb{R}^3$ supported on a support surface $S\subset \mathbb{R}^3$  and that of its  wetting surface $S(\Sigma)\subset S$. In particular, recall that  $\Sigma\setminus \partial \Sigma \subset D(S)$.  Moreover, recall the concepts of  stability and weak stability for the free energy  of a minimal capillary surface.

Let $b\in \mathbb{R}$ with $b\neq 0$. 

In this section, we consider a family 
   $\psi_\lambda \in C^\infty(\mathbb{R}^2)$, $\lambda>1$,  such that 
\begin{align} \label{psi}
\psi_\lambda(y)=\frac{b}{\lambda}\,\log|y|+\phi_\lambda(y)
\end{align}
where $\phi_\lambda \in C^\infty(\mathbb{R}^2)$, $\lambda>1$, satisfy
\begin{equation} \label{phi}
\begin{aligned} 
\sum_{\ell=0}^5|(\nabla^{\mathbb{R}^2})^\ell\phi_\lambda|&=O(1)\,\frac{1}{\lambda^2}
\end{aligned} 
\end{equation}
in  $\{y\in \mathbb{R}^2:\sfrac12\leq |y|\leq \sfrac32\}$. Here and below, we use $O(1)$ to denote quantities that are independent of $\lambda>1$.

Let $D=\{y\in \mathbb{R}^2:|y|\leq 1\}$ and
$\tilde u_\lambda\in C^\infty(D)$ be such that 
\begin{equation} \label{harmonic pde}
\begin{dcases}
	\Delta^{\mathbb{R}^2}\tilde u_\lambda=0\qquad&\text{in $D$ and}\\
	\quad\,\,\,\,\,\tilde u_\lambda=\psi_\lambda &\text{on $\partial D$}.
\end{dcases}
\end{equation}

\begin{lem} \label{psi estimates}
	There holds
	$$
	\sum_{\ell=0}^4|(\nabla^{\mathbb{R}^2})^\ell\tilde u_\lambda|=O(1)\,\frac{1}{\lambda^2}.	$$
\end{lem}
\begin{proof}
	Note that $\psi_\lambda=\phi_\lambda$ on $\partial D$. The assertion follows from  \eqref{phi} and standard gradient estimates for harmonic functions; see, e.g., \cite[Theorem 2.10]{GilbargTrudinger}.
\end{proof}
Let $$S_\lambda=\{(y,\psi_\lambda(y)):y\in \mathbb{R}^2\}$$
and
$$
\tilde \Sigma_{\lambda}=\{(y,\tilde u_{\lambda}(y)):y\in D\}.
$$
Note that $S_\lambda$ is a support surface and that $\partial \tilde \Sigma_{\lambda}\subset S_\lambda$.

Let $\nu(S_\lambda)$ be the unit normal of $S_\lambda$ pointing down and $\nu(\tilde \Sigma_{\lambda})$  the unit normal of  $\tilde \Sigma_{\lambda}$ pointing up.

In this section, we show that $\tilde \Sigma_\lambda$ can be perturbed to a minimal capillary surface supported on $S_\lambda$ provided that $\lambda>1$ is sufficiently large. 

We will tacitly identify functions  on $\tilde \Sigma_\lambda$ with functions  on $D$ by precomposition with  
$$\Phi_{\mathbb{R}^2}(\tilde u_\lambda):D\to \tilde \Sigma_\lambda\qquad\text{given by}\qquad  \Phi_{\mathbb{R}^2}(\tilde u_\lambda)(y)=(y,\tilde u_\lambda(y)).$$ Likewise, we identify functions on $\partial \tilde \Sigma_\lambda$ with functions on $\partial D$. 

\begin{lem} \label{MC and angle estimate}
There holds
\begin{align} \label{MC LS} 
	\sum_{\ell=0}^2|(\nabla^{\mathbb{R}^2})^\ell H(\tilde \Sigma_{\lambda})|&=O(1)\,\frac{1}{\lambda^6}\text{ and}\\ \label{h LS}
	\sum_{\ell=0}^2|(\nabla^{\mathbb{R}^2})^\ell h(\tilde \Sigma_{\lambda})|&=O(1)\,\frac{1}{\lambda^2}
\end{align}
 in $D$.

There holds
\begin{align*} 
\nu(S_\lambda)\cdot\nu(\tilde \Sigma_{\lambda })&=-1+\frac{b^2}{2\,\lambda^2}+O(1)\,\frac{1}{\lambda^3}\text{ and}\\
\sum_{\ell=1}^3|(\nabla^{\partial D})^\ell(\nu(S_\lambda)\cdot\nu(\tilde \Sigma_{\lambda }))|&=O(1)\,\frac{1}{\lambda^3}
\end{align*} 
on $\partial D$.
\end{lem}
\begin{proof}

	Since $\tilde u_\lambda$ is harmonic, there holds 
	$$
	H(\tilde \Sigma_\lambda)=\frac{\nabla^{\mathbb{R}^2}\nabla^{\mathbb{R}^2}\tilde u_\lambda(\nabla^{\mathbb{R}^2} \tilde u_\lambda, \nabla^{\mathbb{R}^2} \tilde u_\lambda)}{(1+|\nabla^{\mathbb{R}^2}\tilde u_\lambda|^2)^{\sfrac32}}.
	$$
In conjunction with Lemma \ref{psi estimates}, we obtain \eqref{MC LS}. \eqref{h LS} follows from a similar estimate.

There holds
	$$
-\nu(S_\lambda)\cdot \nu(\tilde \Sigma_{\lambda})=\frac{1+\nabla^{\mathbb{R}^2}\psi_{\lambda}\cdot \nabla^{\mathbb{R}^2}\tilde u_{\lambda}}{(1+|\nabla^{\mathbb{R}^2}\psi_\lambda|^2)^{\sfrac12}\,(1+|\nabla^{\mathbb{R}^2}\tilde u_{\lambda}|^2)^{\sfrac12}}. 
	$$
By \eqref{psi}, \eqref{phi}, and Lemma \ref{psi estimates}
	$$
	\nabla^{\mathbb{R}^2}\psi_{\lambda}\cdot \nabla^{\mathbb{R}^2}\tilde u_{\lambda}=O(1)\,\frac{1}{\lambda^3}.
	$$
Likewise, 
	$$
\frac{1}{(1+|\nabla^{\mathbb{R}^2}\tilde u_{\lambda}|^2)^{\sfrac12}}=1+O(1)\,\frac{1}{\lambda^4}
	$$
	and
	\begin{align*} 
	\frac{1}{(1+|\nabla^{\mathbb{R}^2}\psi_\lambda|^2)^{\sfrac12}}&=	1-\frac{b^2}{2\,\lambda^2}\,\frac{1}{|y|^2}+O(1)\,\frac{1}{\lambda^3}.
		\end{align*} 
	The assertion follows from these estimates and the corresponding estimates for the higher derivatives.
\end{proof}

\begin{lem} \label{tangential compatibility}
Let $\alpha\in(0,1)$. 
	There are $\varepsilon>0$, $c>1$, and $\lambda_0>1$ with the following property. Let  $\lambda>\lambda_0$.  Given   $u^\perp \in C^{2,\alpha}(D)$ with $|u^\perp|_{C^{2,\alpha}(D)}<\varepsilon$, there is a unique  function $u^\top \in C^{2,\alpha}(D)$ with $\Delta^{\mathbb{R}^2}u^\top=0$ and  $|u^\top|_{C^{2,\alpha}(D)}\leq c\,|u^\perp|_{C^{2,\alpha}(D)}$ such that, on $ \partial D$, 
	\begin{equation}  \label{correct boundary}
		\begin{aligned} 
				\psi_\lambda\left((1+u^\top)\,y-\frac{u^{\perp}}{\lambda }\,\frac{\nabla^{\mathbb{R}^2}\tilde u_\lambda }{(1+|\nabla^{\mathbb{R}^2}\tilde u_\lambda|^2)^{\sfrac12}}\right)
			 =\tilde u_\lambda+u^\top\,y\cdot \nabla^{\mathbb{R}^2}\tilde u_\lambda+\frac{u^\perp}{\lambda\,(1+|\nabla^{\mathbb{R}^2}\tilde u_\lambda|^2)^{\sfrac12}}.
		\end{aligned} 
	\end{equation} 
\end{lem}
\begin{proof}
Let
	$$
	\mathcal{G}_\lambda:C^{2,\alpha}(\partial D)\times C^{2,\alpha}(\partial D)\to C^{2,\alpha}(\partial D)
	$$
be given	by 
	\begin{align*} 
	\mathcal{G}_\lambda(u^\top,u^\perp)&=\lambda\,\psi_\lambda\left((1+u^\top)\,y-\frac{u^{\perp}}{\lambda }\,\frac{\nabla^{\mathbb{R}^2}\tilde u_\lambda }{(1+|\nabla^{\mathbb{R}^2}\tilde u_\lambda|^2)^{\sfrac12}}\right)
	 \\&\qquad -\lambda\,\tilde u_\lambda-\lambda\,u^\top\,y\cdot \nabla^{\mathbb{R}^2}\tilde u_\lambda-\frac{u^\perp}{(1+|\nabla^{\mathbb{R}^2}\tilde u_\lambda|^2)^{\sfrac12}}.
	\end{align*} 
	By \eqref{harmonic pde}, $\mathcal{G}_\lambda(0,0)=0$. By \eqref{phi}, the Fr\'echet derivative
	$$
	\mathcal{D} \mathcal{G}_\lambda|_{(0,0)}{(\,\cdot\,,0)}: C^{2,\alpha}(\partial D)\to C^{2,\alpha}(\partial D)
	$$
	is given by
	$$
	\mathcal{D} {\mathcal{G}}_\lambda|_{(0,0)}{(v,0)}=\left(b-\lambda\,y\cdot \nabla^{\mathbb{R}^2}\tilde u_\lambda+O(1)\,\frac{1}{\lambda}\right)\,v
	$$
	
By Lemma \ref{psi estimates}, 
	
	$$
	|y\cdot \nabla^{\mathbb{R}^2}\tilde u_\lambda|_{C^{2,\alpha}(\partial D)}=O(1)\,\frac{1}{\lambda^2}.
	$$
Since $b\neq 0$, $\mathcal{D}_{} \mathcal{G}_\lambda|_{(0,0)}(\,\cdot\,,0)$ is invertible provided that $\lambda>1$ is sufficiently large. 	By the implicit function theorem, there are $\varepsilon>0$ and $c>1$ such that, given $u^\perp \in C^{2,\alpha}(D)$ with $|u^\perp|_{C^{2,\alpha}(D)}<\varepsilon,$ there is $u^\top \in C^{2,\alpha}(\partial D)$ with $|u^\top|_{C^{2,\alpha}(\partial D)}\leq c\,|u^\perp|_{C^{2,\alpha}(\partial D)}$  that satisfies \eqref{correct boundary} for all such $\lambda>1$; see \cite[Theorem 17.6]{GilbargTrudinger}. By elliptic regularity, the harmonic extension of $u^\top$ to $D$, which we again denote by $u^\top$, has  the asserted properties.
\end{proof}

Let $0<\alpha<1$. Let $\varepsilon>0$ and $\lambda_0>1$ be as in Lemma \ref{tangential compatibility}. 

 Let $X_\lambda \in C^\infty(D,\mathbb{R}^3)$ be given by $$X_\lambda(y)=(y,y\cdot \nabla^{\mathbb{R}^2}\tilde u_\lambda(y)).$$
Note that $X_\lambda\cdot \nu(\tilde \Sigma_\lambda)=0$.

 Given $\lambda>\lambda_0$ and $u^\perp\in C^{2,\alpha}(D)$ with $|u^\perp|_{C^{2,\alpha}(D)}<\varepsilon$, let $u^\top \in C^{2,\alpha}(D)$ be as in Lemma \ref{tangential compatibility}. Define
 
$$
\tilde \Sigma_{\lambda}(u^\perp)=\left\{(y,\tilde u_\lambda(y))+u^\top(y)\,X_\lambda(y)+\frac{u^\perp(y)}{\lambda}\,\nu(\tilde \Sigma_{\lambda})(y):y\in D\right\}.
$$
Note that $\tilde \Sigma_\lambda(0)=\tilde \Sigma_\lambda$. By Lemma \ref{tangential compatibility}, decreasing $\varepsilon>0$ if necessary,  $\tilde \Sigma_{\lambda}(u^\perp)\subset \mathbb{R}^3$ is a compact $C^{2,\alpha}$-surface with $\partial \tilde \Sigma_{\lambda}(u^\perp)\subset S_\lambda$. As before, we tacitly identify functions  on $\tilde \Sigma_\lambda(u^\perp)$ with  functions on $D$ by precomposition with the map
$$
D\to \tilde \Sigma_\lambda(u^\perp)\qquad\text{given by}\qquad y\mapsto (y,\tilde u_\lambda(y))+u^\top(y)\,X_\lambda(y)+\frac{u^\perp(y)}{\lambda}\,\nu(\tilde \Sigma_{\lambda})(y).
$$
 Likewise, we identify functions on $\partial \tilde \Sigma_\lambda(u^\perp)$ with functions on $\partial D$. 

Consider the map
$$
\mathcal{F}_\lambda:\left\{u^\perp\in C^{2,\alpha}(D):|u^\perp|_{C^{2,\alpha}(D)}<\varepsilon\right\}\to C^{0,\alpha}(D)\times  C^{1,\alpha}(\partial D)
$$
 given by 
$$
\mathcal{F}_\lambda(u^\perp)=\left(\lambda\,H(\tilde \Sigma_\lambda(u^\perp)),\lambda^2\left(\nu(S_\lambda)\cdot \nu(\tilde \Sigma_\lambda(u^\perp))+1-\frac{b^2}{2\,\lambda^2}\right)\right).
$$

By Lemma \ref{MC and angle estimate},  
\begin{align} \label{almost solution} 
|\mathcal{F}_{\lambda}(0)|_{C^{0,\alpha}(D)\times  C^{1,\alpha}(\partial D)}=O(1)\,\frac{1}{\lambda}.
\end{align} 

By Lemma \ref{MC and angle estimate}, increasing $\lambda_0>1$ if necessary, $\tilde \Sigma_\lambda$ and $S_\lambda$ intersect transversely along $\partial \tilde \Sigma_\lambda$. Let $\tilde \theta_{\lambda}\in C^\infty(\partial \tilde \Sigma_\lambda ,(\sfrac{\pi}2,\pi))$ be  such that $\cos\tilde \theta_{\lambda}=\nu(S_\lambda)\cdot \nu(\tilde \Sigma_{\lambda})$.

\begin{lem} \label{isomorphism} 
The Fr\'echet derivative
	$$
	\mathcal{D}\mathcal{F}_{\lambda}|_0:C^{2,\alpha}(D)\to C^{0,\alpha}(D)\times  C^{1,\alpha}(\partial D)
	$$
	is an isomorphism for all $\lambda>1$ sufficiently large.
\end{lem}
\begin{proof}
	Note that the initial normal speed of the variation $\{\tilde \Sigma_\lambda(s\,u^\perp):s\in(-1,1)\}$  is $\sfrac{u^\perp}{\lambda}$ and that the tangential part of the initial velocity is $u^\top\,X_\lambda$. 
	
	By the Jacobi equation,
	$$
	\frac{d}{ds}\bigg|_{s=0}\lambda\, H(\tilde \Sigma_\lambda(s\,u^\perp))=-\Delta^{\tilde \Sigma_\lambda}u^\perp-|h(\tilde \Sigma_\lambda)|^2\,u^\perp+\lambda\,u^\top\,X_\lambda\cdot\nabla^{\tilde \Sigma_\lambda} H(\tilde \Sigma_\lambda).
	$$
	By Lemma \ref{MC and angle estimate} and Lemma \ref{tangential compatibility},
	$$
	-\Delta^{\tilde \Sigma_\lambda}u^\perp-|h(\tilde \Sigma_\lambda)|^2\,u^\perp+\lambda\,u^\top\,X_\lambda\cdot \nabla^{\tilde \Sigma_\lambda} H(\tilde \Sigma_\lambda)=-\Delta^{\mathbb{R}^2}u^\perp+w_1
	$$
	where $w_1\in C^{0,\alpha}(D)$ satisfies 
	$$
	|w_1|_{C^{0,\alpha}(D)}=O(1)\,\frac{1}{\lambda^2}\,|u^\perp|_{C^{2,\alpha}(D)}.
	$$ 
	
	By Lemma \ref{capillary change},
	\begin{align*} 
	&\frac{d}{ds}\bigg|_{s=0}\lambda^2\,\nu(S_\lambda)\cdot\nu(\tilde\Sigma_{\lambda}(s\,u^\perp))
	\\&\qquad =-\lambda\sin(\tilde \theta_{\lambda})\,\mu(\tilde \Sigma_{\lambda})\cdot\nabla^{\tilde \Sigma_{\lambda}}u^\perp +\lambda\,h(S_\lambda)(\mu(S_\lambda(\tilde \Sigma_{\lambda})),\mu(S_\lambda(\tilde \Sigma_{\lambda})))\,u^\perp
	\\&\qquad\qquad-\lambda\cos(\tilde \theta_{\lambda})\,h(\tilde \Sigma_{\lambda})(\mu(\tilde \Sigma_{\lambda}),\mu(\tilde \Sigma_{\lambda}))\, u^\perp.
	\end{align*} 
	By \eqref{psi}, Lemma \ref{psi estimates}, and Lemma \ref{MC and angle estimate},
\begin{align*} 
\mu(\tilde \Sigma_{\lambda})\cdot \nabla^{\tilde \Sigma_{\lambda}} u^\perp&=y\cdot \nabla^{\mathbb{R}^2} u^\perp+w_2,\\
	\lambda\sin\tilde \theta_{\lambda}&=b+O(1)\,\frac{1}{\lambda},\\
	\lambda\,h(S_\lambda)(\mu(S_\lambda(\tilde \Sigma_{\lambda})),\mu(S_\lambda(\tilde \Sigma_{\lambda})))&=-b+O(1)\,\frac{1}{\lambda},\text{ and}\\
	\lambda\cos(\tilde\theta_{\lambda})\,h(\tilde \Sigma_{\lambda})(\mu(\tilde \Sigma_{\lambda}),\mu(\tilde \Sigma_{\lambda}))&=O(1)\,\frac{1}{\lambda}
	\end{align*} 
		where $w_2\in C^{1,\alpha}(\partial D)$ satisfies 
	$$
	|w_2|_{C^{1,\alpha}(\partial D)}=O(1)\,\frac{1}{\lambda^2}\,|u^\perp|_{C^{2,\alpha}(\partial D)}.
	$$ 

It follows that  $\mathcal{D}\mathcal{F}_\lambda|_{0}$ converges strongly to the operator
$$
L:C^{2,\alpha}(D)\to C^{0,\alpha}(D)\times  C^{1,\alpha}(\partial D)
$$
given by 
$$
L(v)=(-\Delta^{\mathbb{R}^2} v,-b\,y\cdot \nabla^{\mathbb{R}^2}v-v).
$$

This operator is an isomorphism by \cite[Theorem 6.31]{GilbargTrudinger}. 
 It follows that $\mathcal{D}\mathcal{F}_\lambda|_0$ is also an isomorphism for $\lambda>1$ sufficiently large.
 
\end{proof}
\begin{prop} \label{IFT existence}
	There are $\varepsilon>0$ and $\lambda_0>1$ with the following property. Let $\lambda>\lambda_0$ and $\sfrac{\pi}2<\theta_\lambda<\pi$  such that $\cos\theta_\lambda=-1+\sfrac{b^2}{2\,\lambda^2}$. There is  $u_\lambda^\perp \in C^{2,\alpha}(D)$ with $|u_\lambda^\perp|_{C^{2,\alpha}(D)}=O(1)\,\sfrac{1}{\lambda}$  such that 
	\begin{align} \label{Sigma lambda} 
\hat 	\Sigma_\lambda=\tilde \Sigma_\lambda(u_\lambda^\perp)
	\end{align} 
	satisfies $H(\hat \Sigma_\lambda)=0$ and $\nu(S_\lambda)\cdot \nu(\Sigma_\lambda)=\cos\theta_\lambda$  on $ \partial \hat \Sigma_\lambda$.

	Conversely, if $u^\perp\in C^{2,\alpha}(D)$ is such that $|u^\perp|_{C^{2,\alpha}(D)}<\varepsilon$, $H(\tilde \Sigma_\lambda(u^\perp))=0$, and $\nu(S_\lambda)\cdot \nu(\tilde \Sigma_\lambda(u^\perp))=\cos\theta_\lambda$ on $\partial \tilde \Sigma_\lambda(u^\perp)$,
	then $u^\perp= u_\lambda^\perp$.
\end{prop}
\begin{proof}
	This follows from the inverse function theorem, using \eqref{almost solution} and Lemma \ref{isomorphism}. 
\end{proof}

Let $\hat \Sigma_\lambda$ be as in \eqref{Sigma lambda}. 
\begin{lem} \label{admissible mp}
Assume that $H(S_\lambda)\geq 0$ on $\{x\in \mathbb{R}^3:\sfrac12\leq |x|\leq \sfrac32\}\cap S_\lambda$. Then $S_\lambda\cap \hat \Sigma_\lambda=\partial \hat \Sigma_\lambda$ for all $\lambda>1$ sufficiently large.
\end{lem}
\begin{proof}
	Suppose not. By \eqref{psi}, \eqref{phi}, Lemma \ref{psi estimates}, and Proposition \ref{IFT existence},  there are $s>0$ and $x\in \hat \Sigma_\lambda\setminus \partial \hat\Sigma_\lambda$ with $\sfrac12<|x+s\,e_3|<\sfrac32$ such that $\hat \Sigma_\lambda$ touches $S_\lambda-s\,e_3$ from one side 
	 at $x$. This contradicts the strong maximum principle.
\end{proof}
\begin{lem} \label{stable coro}
 $\hat \Sigma_\lambda$ is stable for the free energy for all $\lambda>1$ sufficiently large.
\end{lem}
\begin{proof}
	By a standard compactness argument and the trace  theorem,  
	there is $c>1$ such that
	$$
\int_{D}f^2\leq c\,\int_{D}	|\nabla^{\mathbb{R}^2}f|^2+c\,\int_{\partial D}f^2
	$$
	for all $f\in C^{1}(D)$. By Lemma \ref{psi estimates} and Proposition \ref{IFT existence}, $\hat \Sigma_\lambda$ converges to $D\times \{0\}$ in $C^{2,\alpha}$. In view of \eqref{stability}, this implies that $\hat \Sigma_\lambda$ is stable for the free energy for all $\lambda>1$ sufficiently large, as asserted.
\end{proof}
\section{Proof of Theorem \ref{THM A}, Theorem \ref{THM B}, and Theorem \ref{THM B 2}}
	Recall from Appendix \ref{appendix support surfaces} the definition of a support surface $S\subset \mathbb{R}^3$ with unit normal $\nu(S)$ and of the domain $D(S)\subset \mathbb{R}^3$ bounded by $S$.

Recall from Appendix \ref{appendix minimal capillary surfaces}  the  definition of a minimal capillary surface $\Sigma \subset \mathbb{R}^3$ supported on a support surface $S\subset \mathbb{R}^3$  and that of its  wetting surface $S(\Sigma)\subset S$. In particular, recall that  $\Sigma\setminus \partial \Sigma \subset D(S)$.  Moreover, recall the concepts of  stability and weak stability for the free energy  of a minimal capillary surface.

Recall from Appendix \ref{appendix: cemswft} the definition of a complete embedded minimal surface $S\subset \mathbb{R}^3$ with finite total curvature and the concept of such a surface being horizontal.

In this section, we consider a complete embedded minimal surface $S\subset \mathbb{R}^3$ with finite total curvature that is not congruent to an affine plane. We may assume that $S$ is horizontal.   Let $S^1,\,\ldots,\, S^m$ be the ends of $S$ labeled as in Appendix \ref{appendix: cemswft} and $b>0$ the logarithmic growth rate of $S^1$.

In this section, we prove Theorem \ref{THM A}, Theorem \ref{THM B}, and Theorem \ref{THM B 2}.

Below, $\lambda>1$ and $\sfrac{\pi}2<\theta_\lambda<\pi$ are related by 
$$
\cos\theta_\lambda=-1+\frac{b^2}{2\,\lambda^2}.
$$  

\begin{proof}[Proof of Theorem \ref{THM A}]
   Given $\lambda>\lambda_0$, let $S_\lambda$ be a complete graphical surface without boundary  that coincides with $\sfrac{1}{\lambda}\,S^1$ outside $B_{1/4}(0)$. By Proposition \ref{IFT existence} and Lemma \ref{admissible mp}, there is a minimal capillary surface $ \hat \Sigma_\lambda\subset \mathbb{R}^3$ supported on $S_\lambda$ with capillary angle  $\theta_\lambda$ provided that $\lambda_0>1$ is sufficiently large. Let $\theta_0=\theta_{\lambda_0}$. Given $\theta_0<\theta<\pi$, let 
	$
	\Sigma(\theta)=\lambda\,\hat \Sigma_\lambda
	$ 
	where $\lambda>\lambda_0$ is such that $\theta_\lambda=\theta$. 
	By Lemma \ref{psi estimates} and Proposition \ref{IFT existence}, $\sin(\theta)\,\Sigma(\theta)$ converges smoothly to $\{y\in \mathbb{R}^2:|y|\leq b\}\times \{0\}$ as $\theta \nearrow \pi$. By 
	 Lemma \ref{stable coro}, increasing $\sfrac{\pi}2<\theta_0<\pi$ if necessary, each  $\Sigma(\theta)$ is stable for the free energy. By \cite[Lemma 51]{eichmair2024penrose}, there is $\varepsilon(\theta)\in (0,\min\{\theta-\theta_0,\pi-\theta\})$ and a smooth foliation $\{\bar  \Sigma_{\bar \theta}  : \bar \theta\in(\theta-\varepsilon(\theta),\theta+\varepsilon(\theta))\}$ of  minimal capillary surfaces $\bar  \Sigma_{\bar  \theta}\subset \mathbb{R}^3$ supported on $S$ that are stable for the free energy such that $\bar  \Sigma_{\theta}=\Sigma(\theta)$. By scaling and the uniqueness statement of  Proposition \ref{IFT existence}, shrinking $\varepsilon(\theta)>0$ if necessary, $\bar  \Sigma_{\bar  \theta}= \Sigma(\bar  \theta)$ for all $\bar  \theta \in(\theta-\varepsilon(\theta),\theta+\varepsilon(\theta))$. It follows that $\{\Sigma(\theta):\theta_0<\theta<\pi\}$ is a smooth foliation. Since $\inf\{|x|:x\in\partial \Sigma(\theta)\}\to \infty$ as $\theta\nearrow \pi$, it follows that 
	 $$
	S^1\setminus \bigcup_{\theta_0<\theta<\pi}\partial \Sigma(\theta)
	 $$
	 is bounded. This completes the proof of Theorem \ref{THM A}.
\end{proof}
\begin{proof}[Proof of Theorem \ref{THM B}]
Let $\Sigma_k\subset \mathbb{R}^3$ be   minimal capillary surfaces supported on $S$ with capillary angle $\sfrac{\pi}2<\theta_k<\pi$. Assume that each $\Sigma_k$ is stable  for the free energy and that $\theta_k\nearrow \pi$. According to Proposition \ref{blowdown odd},  $\partial \Sigma_k\subset S^1$ for every $k$ sufficiently large and  $\sfrac{1}{r(\Sigma_k)}\,\Sigma_k$ converges smoothly to $\{y\in \mathbb{R}^2:|y|\leq 1\}\times\{0\}$, where we recall that $r(\Sigma_k)>0$ is  the area radius of $\Sigma_k$ defined in \eqref{area radius}. Let  $\lambda_k>0$ be such that $\theta_{\lambda_k}=\theta_k$.
	Note that $\lambda_k\to \infty$ and, by Corollary \ref{sin vs r odd}, $\lambda_k=(1+o(1))\,r(\Sigma_k)$. Let $S_{\lambda_k}$ be a complete graphical surface without boundary  that coincides with $\sfrac{1}{\lambda_k}\,S^1$ outside $B_{1/4}(0)$. By the uniqueness part of Proposition \ref{IFT existence}, using  Corollary \ref{sin vs r odd} again, $\sfrac{1}{\lambda_k}\,\Sigma_k=\hat \Sigma_{\lambda_k}$ where $\hat  \Sigma_{\lambda_k}$ is as in the proof of Theorem \ref{THM A}. 
	
	The proof of Theorem \ref{THM A} shows that $\Sigma_k=\Sigma(\theta_k)$ as asserted. This completes the proof of Theorem \ref{THM B}.
	\end{proof}
	\begin{proof}[Proof of Theorem \ref{THM B 2}]
		This is similar to the proof of Theorem \ref{THM B}. The necessary modifications are formal.
	\end{proof}

\section{Linear analysis on a closed support surface with positive mean curvature}

	Recall from Appendix \ref{appendix support surfaces} the definition of a support surface $S\subset \mathbb{R}^3$ with unit normal $\nu(S)$  and of the domain $D(S)\subset \mathbb{R}^3$ bounded by $S$.  Moreover, recall   our conventions for the second fundamental form $h(S)$, the shape operator $A(S)$, and the mean curvature $H(S)$ of $S$.

Recall from Appendix \ref{appendix minimal capillary surfaces}  the  definition of a minimal capillary surface $\Sigma \subset \mathbb{R}^3$ supported on a support surface $S\subset \mathbb{R}^3$  and that of its  wetting surface $S(\Sigma)\subset S$. In particular, recall that  $\Sigma\setminus \partial \Sigma \subset D(S)$.  Moreover, recall the concepts of  stability and weak stability for the free energy  of a minimal capillary surface.

In this section, we consider a closed support surface $S\subset \mathbb{R}^3$ with  $H(S)>0$.

 Given $z\in S$, let 
 \begin{align} \label{rotation} R_z:\mathbb{R}^3\to \mathbb{R}^3
 	\end{align}  be any rotation such that  $R_z(\nu(S)(z))=-e_3$ and $R_z(\upsilon)=\upsilon$ for all $\upsilon\in \mathbb{R}^3$ with  $\upsilon\cdot \nu(S)(z)=\upsilon\cdot e_3=0$.

Given $z\in S$ and $0<\rho<1$, let $$S_{z,\rho}=\sfrac{1}{\rho}\,R_z\,(S-z).$$ Note that $0\in S_{z,\rho}$ and  $\nu(S_{z,\rho})(0)=-e_3$ where  $\nu(S_{z,\rho})$ is the outward pointing unit normal of $S_{z,\rho}$.

Throughout this section, we use $O(1)$ to denote quantities that are bounded independently of  both $z$ and $\rho$.

\begin{lem} \label{local graph} There is $0<\rho_0<1$ with the following property. Let $z\in S$ and $0<\rho<\rho_0$. There is $\phi_{z,\rho}\in C^\infty(\mathbb{R}^2)$ 
	with 
	\begin{align*}
	&\circ\qquad  \phi_{z,\rho}(0)=0, \\
	&\circ\qquad  \nabla^{\mathbb{R}^2} \phi_{z,\rho}(0)=0,\\
	&\circ \qquad  \nabla^{\mathbb{R}^2}\nabla^{\mathbb{R}^2} \phi_{z,\rho}(0)=0, \text{ and}\\
	&\circ\qquad \sum_{\ell=0}^4|(\nabla^{\mathbb{R}^2})^\ell\phi_{z,\rho}|=O(1)\,\rho^2
	\end{align*} 
	such that  $\psi_{z,\rho} \in C^\infty(\mathbb{R}^2)$ given by 
	$$
	\psi_{z,\rho}(y)=\rho\,h(S)(z)(R^{-1}_z(y),R^{-1}_z(y))+\phi_{z,\rho}(y)
	$$
	satisfies $$
	B_{2}(0)\cap S_{z,\rho}\subset \{(y,\psi_{z,\rho}(y)):y\in \mathbb{R}^2\}.
	$$
	
\end{lem}

Let $$D=\{y\in \mathbb{R}^2:|y|\leq 1\}$$ and  $\tilde u_{z,\rho}\in C^\infty(D)$ be given by
\begin{align} \label{ u definition} 
	\tilde 	u_{z,\rho}(y)=\frac{\rho\,H(S)(z)}{4}\,\left(1-|y|^2\right).
\end{align}
Define
\begin{align} \label{tilde sigma definition} 
\tilde \Sigma_{z,\rho}=\{(y,\psi_{z,\rho}(y))-\tilde u_{z,\rho}(y)\,\nu(S_{z,\rho})(y,\psi_{z,\rho}(y)):y\in D\}
\end{align} 
where $\psi_{z,\rho}\in C^\infty(D)$ is as in Lemma \ref{local graph}. Since $\tilde u_{z,\rho}(y)=0$ for all $y\in \partial D$ and $H(S)(z)>0$, $\tilde \Sigma_{z,\rho}$ is an admissible surface supported on $S_{z,\rho}$ with wetting surface $$S_{z,\rho}(\tilde \Sigma_{z,\rho})=\{(y,\psi_{z,\rho}(y)):y\in D\}$$ provided that $\rho>0$ is sufficiently small. We tacitly identify functions on $\tilde \Sigma_{z,\rho}$ with functions on $D$ by precomposition with 
$$
\Phi_{z,\rho}:D\to \tilde \Sigma_{z,\rho}\qquad\text{given by}\qquad  \Phi_{z,\rho}(y)= (y,\psi_{z,\rho}(y))-\tilde u_{z,\rho}(y)\,\nu(S_{z,\rho})(y,\psi_{z,\rho}(y)).
$$
We also use $\Phi_{z,\rho}$ to identify functions on $\partial \tilde \Sigma_{z,\rho}$ with functions on $\partial D$.

In this section, we prove the following fact: For sufficiently small $\rho>0$, $\tilde \Sigma_{z,\rho}$ can be perturbed to a minimal surface that intersects $S_{z,\rho}$ at an angle that is constant up to first spherical harmonics. The proof is by way of a Lyapunov-Schmidt reduction.

\begin{lem} \label{u estimates}
	There holds 

\begin{align*}
\sum_{\ell=0}^2|(\nabla^{\mathbb{R}^2})^\ell \tilde u_{z,\rho}|=O(1)\,\rho,
\end{align*}
$
(\nabla^{\mathbb{R}^2})^3 \tilde u_{z,\rho}=0,$ and $ (\nabla^{\mathbb{R}^2})^4 \tilde u_{z,\rho}=0$ in $D$. Moreover, 	$$
\Delta^{\mathbb{R}^2} \tilde u_{z,\rho}=-\rho\,H(S)(z)
$$
in $D$ and 
$$
|\nabla^{\mathbb{R}^2} \tilde u_{z,\rho}|^2=\frac{\rho^2\,H(S)(z)^2}4
$$
on $\partial D$.
	
\end{lem}
\begin{proof}
	This follows from \eqref{ u definition}.
\end{proof}

Let  $\nu(\tilde \Sigma_{z,\rho})$ be the unit normal of $\tilde \Sigma_{z,\rho}$ that points out of the compact set bounded by $\tilde \Sigma_{z,\rho}$ and $S_{z,\rho}(\tilde \Sigma_{z,\rho})$.

\begin{lem} \label{rho MC}
	There holds

\begin{align*} 
	\sum_{\ell=0}^2|(\nabla^{\mathbb{R}^2})^\ell H(\tilde \Sigma_{z,\rho})|&=O(1)\,\rho^2\text{ and}\\
	\sum_{\ell=0}^2|(\nabla^{\mathbb{R}^2})^\ell h(\tilde \Sigma_{z,\rho})|&=O(1)\,\rho 
\end{align*}
in $D$.
	
	There holds  
	\begin{align} \label{compact angle} 
	\nu(S_{z,\rho})\cdot \nu(\tilde \Sigma_{z,\rho})=-1+\frac{\rho^2\,H(S)(z)^2}{8}+O(1)\,\rho^3
	\end{align}
	and 
	$$
\sum_{\ell=1}^3|(\nabla^{\partial D})^\ell	(\nu(S_{z,\rho})\cdot \nu(\tilde \Sigma_{z,\rho}))|=O(1)\,\rho^3
	$$ 
	on $\partial D$.
\end{lem}
\begin{proof} 
	This follows as in the proof of Lemma \ref{MC and angle estimate}, using Lemma \ref{local graph} and  Lemma \ref{u estimates}.

\end{proof}
In the following two lemmas, we identify functions on  $S_{z,\rho}(\tilde \Sigma_{z,\rho})$ with functions on $D$ by precomposition with the map 
$D\to S_{z,\rho}(\tilde \Sigma_{z,\rho})$ given by $y\mapsto (y,\psi_{z,\rho}(y)).$

\begin{lem} \label{closeness}
	There holds 
	\begin{align*}
	|S_{z,\rho}(\tilde \Sigma_{z,\rho})|=\pi+O(1)\,\rho^2
	\end{align*} 
	and
	$$
	|\tilde \Sigma_{z,\rho}|=|S_{z,\rho}(\tilde \Sigma_{z,\rho})|-\frac{\pi}{16}\,H(S)(z)^2\,\rho^2+O(1)\,\rho^4.
	$$
	
\end{lem}
\begin{proof}
	By Lemma \ref{local graph} and Taylor's theorem,
	$$
	|S_{z,\rho}(\tilde \Sigma_{z,\rho})|=\int_{D}(1+|\nabla^{\mathbb{R}^2} \psi_{z,\rho}|^2)^{\sfrac12}=\pi+O(1)\,\rho^2.
	$$

	By  Taylor's theorem, the formulae for  the first and second variation of area (see \cite[Chapter 1, \S1 and \S8]{ColdingMinicozzi}), Lemma \ref{local graph}, and Lemma \ref{u estimates},
	$$
	|\tilde \Sigma_{z,\rho}|=|S_{z,\rho}(\tilde \Sigma_{z,\rho})|-\int_{S_{z,\rho}(\tilde \Sigma_{z,\rho})}H(S_{z,\rho})\,\tilde u_{z,\rho}-\frac12\,\int_{S_{z,\rho}(\tilde \Sigma_{z,\rho})} \tilde u_{z,\rho}\,\Delta^{S_{z,\rho}}\tilde u_{z,\rho}+O(1)\,\rho^4.	$$
	
	By Lemma \ref{local graph} and Lemma \ref{u estimates},
	$$
	\int_{S_{z,\rho}(\tilde \Sigma_{z,\rho})} \tilde u_{z,\rho}\,\Delta^{S_{z,\rho}}\tilde u_{z,\rho}=-\rho\,H(S)(z)\,\int_{S_{z,\rho}(\tilde \Sigma_{z,\rho})} \tilde u_{z,\rho}+O(1)\,\rho^{4}
	$$
and
	$$
	\int_{S_{z,\rho}(\tilde \Sigma_{z,\rho})}H(S_{z,\rho})\,\tilde u_{z,\rho}=\rho\,H(S)(z)\,\int_{S_{z,\rho}(\tilde \Sigma_{z,\rho})}\tilde  u_{z,\rho}+O(1)\,\rho^4.
	$$
Likewise, using also \eqref{ u definition},
	$$
	\int_{S_{z,\rho}(\tilde \Sigma_{z,\rho})} \tilde u_{z,\rho}=\frac{\rho\,H(S)(z)}{4}\,\int_D(1-|y|^2)+O(1)\,\rho^3.
	$$
Note that
	$$
	\int_D(1-|y|^2)=\frac{\pi}{2}.
	$$
	The assertions follow from these estimates.
\end{proof}
In Lemma \ref{initial co-normal}, we will use that it is possible to choose the rotations $R_z$ defined in \eqref{rotation} locally smoothly in terms of $z$.
\begin{lem}
Let $z\in S$ and $0<\rho<\rho_0$.  Let $\gamma\in C^\infty((-1,1),S)$ be such that   $\gamma(0)=z$ and $\gamma'(0)=R_z^{-1}e_i$ for $i=1,\,2$. The initial co-normal speed of the variation \label{initial co-normal} $$\left\{R_{z}\left(R_{\gamma(s)}^{-1}S_{\gamma(s),\rho}(\tilde \Sigma_{\gamma(s),\rho})+\sfrac1\rho\,\gamma(s)\right)-\sfrac1\rho\,z:s\in(-1,1)\right\}$$
equals $\rho\,(e_i\cdot y)+O(1)$. 
\end{lem}
\begin{proof}
This follows from \eqref{tilde sigma definition}, using	Lemma \ref{local graph} and Lemma \ref{u estimates}.
\end{proof}
By \eqref{compact angle}, there is  $\tilde \theta_{z,\rho}\in C^\infty(\partial D,(\sfrac{\pi}2,\pi))$ with $$\cos\tilde \theta_{z,\rho}=\nu(S_{z,\rho})\cdot \nu(\tilde \Sigma_{z,\rho})$$ provided that $\rho>0$ is sufficiently small. 

Given $y\in D$, let $y^\perp=(-e_2\cdot y,e_1\cdot y)$. Note that 
$$(y,y\cdot \nabla^{\mathbb{R}^2}\psi_{z,\rho}),\, (y^\perp,y^\perp\cdot \nabla^{\mathbb{R}^2}\psi_{z,\rho})\in T_{(y,\psi_{z,\rho}(y))}S_{z,\rho}.$$  
Let $X_{z,\rho} \in C^\infty(D,\mathbb{R}^3)$ be given by 
\begin{align*} 
X_{z,\rho}&=(y,y\cdot\nabla^{\mathbb{R}^2} \psi_{z,\rho})-(y\cdot \nabla^{\mathbb{R}^2} \tilde u_{z,\rho})\, \nu(S_{z,\rho})-\tilde u_{z,\rho}\, A(S_{z,\rho})(y,y\cdot \nabla^{\mathbb{R}^2}\psi_{z,\rho})
\\&\qquad -\frac{(y\cdot \nabla^{\mathbb{R}^2}\psi_{z,\rho})\,(y^\perp \cdot \nabla^{\mathbb{R}^2}\psi_{z,\rho})}{1+(y^\perp \cdot \nabla^{\mathbb{R}^2}\psi_{z,\rho})^2} \bigg((y^\perp,y^\perp\cdot \nabla^{\mathbb{R}^2}\psi_{z,\rho})-(y^\perp \cdot \nabla^{\mathbb{R}^2} \tilde u_{z,\rho})\, \nu(S_{z,\rho})
\\&\qquad\qquad -\tilde u_{z,\rho}\, A(S_{z,\rho})(y^\perp,y^\perp\cdot \nabla^{\mathbb{R}^2}\psi_{z,\rho})\bigg)
\end{align*} 
Note that $X_{z,\rho}\cdot \nu(\tilde \Sigma_{z,\rho})=0$ in $D$. Moreover, since $\tilde u_{z,\rho}=0$ on $\partial D$,  
\begin{align} \label{tangetial at boundary} X_{z,\rho}=|X_{z,\rho}|\,\mu(\tilde \Sigma_{z,\rho})
\end{align} 
on $\partial D$. By Lemma \ref{local graph} and Lemma \ref{u estimates}
\begin{align*} 
	|X_{z,\rho}|=1+O(1)\,\rho
\end{align*} 
on $\partial D$.

\begin{lem} \label{tangential compatibility 2}
Let $0<\alpha<1$.	There are $\varepsilon>0$, $c>1$, and $\rho_0>0$ with the following property.  Let $z\in S$ and $0<\rho<\rho_0$. Given   $u^\perp \in C^{2,\alpha}(D)$ with $|u^\perp|_{C^{2,\alpha}(D)}<\varepsilon$, there is a unique  function $u^\top \in C^{2,\alpha}(D)$ with $\Delta^{\mathbb{R}^2}u^\top=0$ and $|u^\top|_{C^{2,\alpha}(D)}\leq c\,|u^\perp|_{C^{2,\alpha}(D)}$ such that 
	\begin{equation*}  
		\begin{aligned} 
			& \Phi_{z,\rho}(y)	+u^\top(y)\, X_{z,\rho}(y)-\tan(\tilde \theta_{z,\rho}(y))\,u^\perp(y)\,\nu(\tilde \Sigma_{z,\rho})(y)\in S_{z,\rho}
		\end{aligned} 
	\end{equation*} 
	for all $y\in \partial D$. 
Moreover, there holds
$$
u^\top(y)=(1+O(1)\,\rho)\,u^\perp(y)
$$
for all $y\in \partial D$.
\end{lem}
\begin{proof}
This follows as in the proof of Lemma \ref{tangential compatibility}.
\end{proof}
Let $0<\alpha<1$. Let $\varepsilon>0$ and $\rho_0>0$ be as in Lemma \ref{tangential compatibility 2}. 

Given $z\in S$, $0<\rho<\rho_0$, and $u^\perp\in C^{2,\alpha}(D)$ with $|u^\perp|_{C^{2,\alpha}(D)}<\varepsilon$, let $u^\top \in C^{2,\alpha}(D)$ be as in Lemma \ref{tangential compatibility 2}. Define
$$
\tilde \Sigma_{z,\rho}(u^\perp)=\left\{ \Phi_{z,\rho}(y)+u^\top(y)\, X_{z,\rho}(y)-\tan(\tilde \theta_{z,\rho}(y))\,u^\perp(y)\,\nu(\tilde \Sigma_{z,\rho})(y):y\in D\right\}.
$$
Note that $\tilde \Sigma_{z,\rho}(0)=\tilde \Sigma_{z,\rho}$. By Lemma \ref{tangential compatibility 2}, decreasing $\varepsilon>0$ if necessary,  $\tilde \Sigma_{z,\rho}(u^\perp)$ is a compact $C^{2,\alpha}$-surface with $\partial \tilde \Sigma_{z,\rho}(u^\perp)\subset S_{z,\rho}$. Moreover, using \eqref{ u definition} and that $H(S)(z)>0$, we see that $\tilde \Sigma_{z,\rho}(u^\perp)$ is an admissible surface supported on $S_{z,\rho}$ provided that $\varepsilon>0$ is sufficiently small. We tacitly identify functions  on $\tilde \Sigma_{z,\rho}(u^\perp)$ with functions  on $D$ by precomposition with the map 
$$
D\to \tilde \Sigma_{z,\rho}(u^\perp)\qquad\text{given by}\qquad y\mapsto   \Phi_{z,\rho}(y)+u^\top(y)\, X_{z,\rho}(y)-\tan(\tilde \theta_{z,\rho}(y))\,u^\perp(y)\,\nu(\tilde \Sigma_{z,\rho})(y).
$$
Likewise, we identify functions on $\partial \tilde \Sigma_{z,\rho}(u^\perp)$ with functions on $\partial D$. 

 We denote by $\Lambda^0(D)\subset C^\infty(D)$ and $\Lambda^0(\partial D)\subset C^\infty(\partial D)$ the 1-dimensional subspaces of constant functions and by $\Lambda^1(D)\subset C^\infty(D)$ and  $\Lambda^1(\partial D)\subset C^\infty(\partial D)$ the 2-dimensional subspaces spanned by the coordinate functions. 
  We use $\perp$ to denote the $L^2$-orthogonal complement in $C^\infty(D)$ and $C^\infty(\partial D)$ of these subspaces.
\begin{lem} \label{tangential correction}
Decreasing $\varepsilon>0$ and $\rho_0>0$ if necessary, there is $c>1$ with the following property.  Let $z_0\in S$ and  $0<\rho<\rho_0$. Given $u_{z_0}^\perp\in C^{2,\alpha}(D)$ with $|u_{z_0}^\perp|_{C^{2,\alpha}(D)}<\varepsilon$, there are $z\in S$ and  $u_z^\perp \in C^{2,\alpha}(D)\cap \Lambda^{1}(\partial D)^\perp$ such that 
\begin{align*}
	&\circ\qquad\text{$\operatorname{dist}_S(z,z_0)\leq c\,\rho\,|u_{z_0}^\perp|_{C^{2,\alpha}(D)}$}\\
	&\circ\qquad\text{$|u_z^\perp|_{C^{2,\alpha}(D)}\leq c\,|u_{z_0}^\perp|_{C^{2,\alpha}(D)}$, and}\\
	&\circ\qquad\text{$R_{z_0}\left(R_z^{-1}\tilde \Sigma_{z,\rho}( u_z^\perp)+\sfrac1\rho\,z\right)-\sfrac1\rho\,z_0=\tilde \Sigma_{z_0,\rho}(u_{z_0}^\perp).$}
\end{align*}
\end{lem}
\begin{proof}
	
An argument analogous to that of the proof of Lemma \ref{tangential compatibility} shows that, decreasing $\rho_0>0$ if necessary, there is $\delta>0$ with the following property. Given  $u_{z_0}^\perp\in C^{2,\alpha}(D)$ with $|u_{z_0}^\perp|_{C^{2,\alpha}(D)}<\sfrac{\varepsilon}{2}$ and $z\in S$ with $\operatorname{dist}_{S}(z,z_0)<\delta\,\rho$, there is a unique $u_z^\perp \in C^{2,\alpha}(D)$ with  $| u_z^\perp|_{C^{2,\alpha}(D)}<\varepsilon$ such that 
	\begin{align*}   
	R_{z_0}\left(R_z^{-1}\tilde \Sigma_{z,\rho}( u_z^\perp)+\sfrac1\rho\,z\right)-\sfrac1\rho\,z_0=\tilde \Sigma_{z_0,\rho}(u_{z_0}^\perp).
	\end{align*}  Moreover, the dependence of both $u^\perp_z$  and $u^\top_z$ on $z$ is smooth. We use $\mathcal{D}u^\perp$ to denote the Fr\'echet derivative of $u_z^\perp$ with respect to $z$ and  $\mathcal{D}u^\top$ to denote the Fr\'echet derivative of $u_z^\top$ with respect to $z$. 
We may and will assume that $T_{z_0}S=\mathbb{R}^2\times \{0\}$.

Assume that $u_{z_0}^\perp=0$. By \eqref{tangetial at boundary} and Lemma \ref{angles}, 
\begin{equation*} 
\begin{aligned} 
	&\Phi_{z,\rho}+u_z^\top\, X_{z,\rho}-\tan(\tilde \theta_{z,\rho})\,u_z^\perp\,\nu(\tilde \Sigma_{z,\rho})
	\\
&\qquad=\Phi_{z,\rho}-\left(\cos(\tilde \theta_{z,\rho})\,|X_{z,\rho}|\,u^\top_z+\sin(\tilde \theta_{z,\rho})\,\tan(\tilde \theta_{z,\rho})\,u^\perp_z\right)\mu(S_{z,\rho}(\tilde \Sigma_{z,\rho}))
\\&\qquad\qquad  +\sin(\tilde \theta_{z,\rho})\,\left(|X_{z,\rho}|\,u_z^\top-u_z^\perp\right)\,\nu(S_{z,\rho})
\end{aligned} 
\end{equation*}
on $\partial D$.
 In view of Lemma \ref{initial co-normal}, using that $u^\top_{z_0}=u^\perp_{z_0}=0$, Lemma \ref{local graph}, Lemma \ref{u estimates}, and Lemma \ref{rho MC},  we find that 
\begin{equation*} 
\begin{aligned} 
(\cos\tilde \theta_{z_0,\rho})^2\,|X_{z_0,\rho}|\,\mathcal{D} u^\top|_{z_0} e_i
+(\sin\tilde \theta_{z_0,\rho})^2\,\mathcal{D} u^\perp|_{z_0} e_i&=
\sfrac{1}\rho\,e_i\cdot y+O(1)
\text{ and}\\
\sin(\tilde \theta_{z_0,\rho})\,\left(|X_{z_0,\rho}|\,\mathcal{D} u^\top|_{z_0} e_i-\mathcal{D} u^\perp|_{z_0} e_i\right)&=O(1)\,\rho
\end{aligned}
\end{equation*}
on $\partial D$ where $i=1,\,2$. Thus,
\begin{align}
	\label{variation} (\mathcal{D} u^\perp|_{z_0} e_i)(y)=\sfrac{1}{\rho}\,e_i\cdot y+O(1).
\end{align}

 Let 
$$
\mathcal{G}_{z_0,\rho}:\{z\in S:\operatorname{dist}_{S}(z,z_0)<\delta\,\rho\}\times \{u_{z_0}^\perp \in C^{2,\alpha}(D):|u_{z_0}^\perp|_{C^{2,\alpha}(D)}<\sfrac{\varepsilon}{2}\}\to \mathbb{R}^2
$$ 
be given by 
$$
\mathcal{G}_{z_0,\rho}(z,u_{z_0}^\perp)=\rho\,\left(\int_{\partial D}(e_1\cdot y)\,u^\perp_{z},\int_{\partial D}(e_2\cdot y)\,u^\perp_{z}\right).
$$
Note that $\mathcal{G}_{z_0,\rho}(z_0,0)=(0,0)$. By \eqref{variation}, 
 the Fr\'echet derivative
$$
\mathcal{D}\mathcal{G}_{z_0,\rho}|_{(z_0,0)}(\,\cdot\,,0):\mathbb{R}^2\cong\mathbb{R}^2\times \{0\}\to \mathbb{R}^2
$$ 
is given by 
$$
\mathcal{D}\mathcal{G}_{z_0,\rho}|_{(z_0,0)}(e_i,0)=\left(\int_{\partial D}(e_1\cdot y)\,(e_i\cdot y),\int_{\partial D}(e_2\cdot y)\,(e_i\cdot y)\right)+O(1)\,\rho,
$$
i.e., 
$$\mathcal{D}\mathcal{G}_{z_0,\rho}|_{(z_0,0)}(e_1,0)=(\pi,0)+O(1)\,\rho\qquad\text{and}\qquad \mathcal{D}\mathcal{G}_{z_0,\rho}|_{(z_0,0)}(e_2,0)=(0,\pi)+O(1)\,\rho.$$  
It follows that $\mathcal{D}\mathcal{G}_{z_0,\rho}|_{(z_0,0)}(\,\cdot\,,0)$ is invertible provided that $\rho>0$ is sufficiently small. The assertion of the lemma  follows from this, scaling, and the implicit function theorem; see \cite[Theorem 17.6]{GilbargTrudinger}.
\end{proof}

\begin{lem} \label{angle} Decreasing $\varepsilon>0$ and $0<\rho_0<1$ if necessary, the following holds. Let $z\in S,$ $0<\rho<\rho_0$. Given $u^\perp\in C^{2,\alpha}(D)$ with $|u^\perp|_{C^{2,\alpha}(D)}<\varepsilon$, there holds
	$$
\nu(S_{z,\rho})\cdot \nu(\tilde \Sigma_{z,\rho}(u^\perp))=-1+\frac{H(S)(z)^2\,\rho^2}{4}+O(1)\,\varepsilon\,\rho^2
	$$
	on $\partial D$.
\end{lem}
\begin{proof}
Note that the initial normal speed of the variation $$\{\tilde \Sigma_{z,\rho}(s\,u^\perp):s\in(-1,1)\}$$  is $(-\tan\tilde \theta_{z,\rho})\,u^\perp$. 

	By Lemma \ref{capillary change},
\begin{align*} 
	&\frac{d}{ds}\bigg|_{s=0}\nu(S_{z,\rho})\cdot\nu(\tilde\Sigma_{z,\rho}(s\,u^\perp))
	\\&\qquad =\tan(\tilde \theta_{z,\rho})\sin(\tilde \theta_{z,\rho})\,\mu(\tilde \Sigma_{z,\rho})\cdot \nabla^{\tilde \Sigma_{z,\rho}}u^\perp 
	\\&\qquad\qquad +\sin(\tilde \theta_{z,\rho})\,h(\tilde \Sigma_{z,\rho})(\mu(\tilde \Sigma_{z,\rho}),\mu(\tilde \Sigma_{z,\rho}))\,u^\perp\\&\qquad\qquad-\tan(\tilde \theta_{z,\rho})\,h(S_{z,\rho})(\mu(S_{z,\rho}(\tilde \Sigma_{z,\rho})),\mu(S_{z,\rho}(\tilde \Sigma_{z,\rho})))\,u^\perp.
\end{align*} 
By Lemma \ref{local graph}, Lemma \ref{u estimates}, and Lemma \ref{rho MC},
\begin{align*} 
	\tan(\tilde \theta_{z,\rho})\sin(\tilde \theta_{z,\rho})\,\mu(\tilde \Sigma_{z,\rho})\cdot \nabla^{\tilde \Sigma_{z,\rho}}u^\perp =O(1)\,\varepsilon\,\rho^2.
\end{align*}
By Lemma \ref{local graph}, Lemma \ref{u estimates},  Lemma \ref{rho MC}, and Corollary \ref{graph 2},
\begin{align*} 
	&\left(\sin(\tilde \theta_{z,\rho})\,h(\tilde \Sigma_{z,\rho})(\mu(\tilde \Sigma_{z,\rho}),\mu(\tilde \Sigma_{z,\rho}))-\tan(\tilde \theta_{z,\rho})\,h(S_{z,\rho})(\mu(S_{z,\rho}(\tilde \Sigma_{z,\rho})),\mu(S_{z,\rho}(\tilde \Sigma_{z,\rho})))\right)\,u^\perp\\&\qquad=O(1)\,\varepsilon\,\rho\, |\nabla^{\mathbb{R}^2}\nabla^{\mathbb{R}^2}\tilde u_{z,\rho}|
	\\&\qquad =O(1)\,\varepsilon\,\rho^2.
\end{align*}
The assertion follows from Taylor's theorem, using these estimates and  Lemma \ref{rho MC}. 
\end{proof}
By Lemma \ref{angle}, given $z\in S$, $0<\rho<\rho_0$, and  $u^\perp\in C^{2,\alpha}(D)$ with $|u^\perp|_{C^{2,\alpha}(D)}<\varepsilon$, there is $\tilde \theta_{z,\rho}(u^\perp) \in C^{2,\alpha}(\partial D,(\sfrac{\pi}2,\pi))$ such that $$\cos\tilde\theta_{z,\rho}(u^\perp)=\nu(S_{z,\rho})\cdot \nu(\tilde \Sigma_{z,\rho}(u^\perp)).$$ Moreover, decreasing $\varepsilon>0$ if necessary, by Lemma \ref{local graph}, Lemma \ref{u estimates}, and \eqref{ u definition}, using that $H(S)(z)>0$, $\tilde \Sigma_{z,\rho}(u^\perp)$ is an admissible surface on $S_{z,\rho}$. Its wetting surface is given by $S_{z,\rho}(\tilde \Sigma_{z,\rho}(u^\perp))$.

Let $\Lambda^{0,1}(\partial D) =\Lambda^0(\partial D)\oplus \Lambda^1(\partial D)$. 
Consider the map
\begin{align*} 
\mathcal{F}_{z,\rho}:&\left\{u^\perp\in C^{2,\alpha}(D) \cap \Lambda^{1}(\partial D)^\perp:|u^\perp|_{C^{2,\alpha}(D)}<\varepsilon\right\} \to C^{0,\alpha}(D)\times \left( C^{1,\alpha}(\partial D)\cap  \Lambda^{0,1}(\partial D)^\perp\right)\times \mathbb{R}
\end{align*} 
 given by 
\begin{align*} 
&\mathcal{F}_{z,\rho}(u^\perp)
=\left(\cot(\tilde \theta_{z,\rho})\,H(\tilde \Sigma_{z,\rho}(u^\perp)),\tan(\tilde \theta_{z,\rho})\operatorname{proj}_{\Lambda^{0,1}(\partial D)^\perp}\cot(\tilde \theta_{z,\rho}(u^\perp)),|S_{z,\rho}(\tilde \Sigma_{z,\rho}(u^\perp))|-\pi)\right).
\end{align*}

By Lemma \ref{rho MC} and Lemma \ref{closeness}, 
\begin{align} \label{almost solution 2} 
	|\mathcal{F}_{z,\rho}(0)|_{C^{0,\alpha}(D)\times C^{1,\alpha}(\partial D)\times \mathbb{R}}=O(1)\,\rho.
\end{align} 
Indeed, note that, e.g., by Lemma \ref{rho MC},
\begin{align} \label{cot}  
\cot\tilde \theta_{z,\rho}=-\frac{2}{\rho\,H(S)(z)}+O(1).
\end{align}

\begin{lem} \label{isomorphism 2} 
 The Fr\'echet derivative
	$$
	\mathcal{D}\mathcal{F}_{z,\rho}|_{0}:C^{2,\alpha}(D)\cap \Lambda^{1}(\partial D)^\perp\to C^{0,\alpha}(D)\times \left(  C^{1,\alpha}(\partial D)\cap\Lambda^{0,1}(\partial D)^\perp\right)\times \mathbb{R}
	$$
	is an isomorphism for all $z\in S$ and  sufficiently small $\rho>0$.
\end{lem}
\begin{proof}

	First,	note that the initial normal speed of the variation $$\{\tilde \Sigma_{z,\rho}(s\,u^\perp):s\in(-1,1)\}$$  is $(-\tan\tilde \theta_{z,\rho})\,u^\perp$ and the initial tangential velocity is $u^\top\,X_{z,\rho}$. 
	
	By the Jacobi equation,
	$$
	\frac{d}{ds}\bigg|_{s=0}\cot(\tilde \theta_{z,\rho})\, H(\tilde \Sigma_{z,\rho}(s\,u^\perp))=\Delta^{\tilde \Sigma_{z,\rho}}u^\perp+|h(\tilde \Sigma_{z,\rho})|^2\,u^\perp+\cot(\tilde \theta_{z,\rho})\,u^\top\,X_{z,\rho}\cdot\nabla^{\tilde \Sigma_{z,\rho}} H(\tilde \Sigma_{z,\rho}).
	$$
	By Lemma \ref{local graph}, Lemma \ref{u estimates}, Lemma \ref{rho MC},   Lemma \ref{tangential compatibility 2}, and \eqref{cot},
	$$
	\Delta^{\tilde \Sigma_{z,\rho}}u^\perp+|h(\tilde \Sigma_{z,\rho})|^2\,u^\perp+\cot(\tilde \theta_{z,\rho})\, X_{z,\rho}\cdot \nabla^{\tilde \Sigma_{z,\rho}} H(\tilde \Sigma_{z,\rho})\,u^\top=\Delta^{\mathbb{R}^2}u^\perp+w_1
	$$
	where $w_1\in C^{0,\alpha}(D)$ satisfies 
	$$
	|w_1|_{C^{0,\alpha}(D)}=O(1)\,\rho\,|u^\perp|_{C^{2,\alpha}(D)}.
	$$ 
	
	By Lemma \ref{capillary change},
	\begin{align*} 
		&(\cos\tilde \theta_{z,\rho})^2\,\frac{d}{ds}\bigg|_{s=0}\tan(\tilde \theta_{z,\rho})\cot(\tilde \theta_{z,\rho}(s\,u^\perp))
		\\&\qquad =\mu(\tilde \Sigma_{z,\rho})\cdot\nabla^{\tilde \Sigma_{z,\rho}}u^\perp
		\\&\qquad \qquad -\frac{1}{\sin\tilde \theta_{z,\rho}}\,h(S_{z,\rho}(\tilde \Sigma_{z,\rho}))(\mu(S_{z,\rho}(\tilde \Sigma_{z,\rho})),\mu(S_{z,\rho}(\tilde \Sigma_{z,\rho})))\,u^\perp
		\\&\qquad\qquad+\cos(\tilde \theta_{z,\rho})\,h(\tilde \Sigma_{z,\rho})(\mu(\tilde \Sigma_{z,\rho}),\mu(\tilde \Sigma_{z,\rho}))\, u^\perp.
	\end{align*} 
	By Lemma \ref{local graph} and Lemma \ref{u estimates},
	\begin{align*} 
	\mu(\tilde \Sigma_{z,\rho})\cdot	\nabla^{\tilde \Sigma_{z,\rho}} u^\perp &=y\cdot \nabla^{\mathbb{R}^2} u^\perp+w_2
		\end{align*}
			where $w_2\in C^{1,\alpha}(\partial D)$ satisfies 
		$$
		|w_2|_{C^{1,\alpha}(\partial D)}=O(1)\,\rho\,|u^\perp|_{C^{2,\alpha}(D)}.
		$$ 
		Moreover, by Lemma \ref{local graph}, Lemma \ref{u estimates}, Lemma \ref{rho MC}, and Corollary  \ref{graph 2},
		\begin{align*} 
&h(S_{z,\rho})(\mu(S_{z,\rho}(\tilde \Sigma_{z,\rho})),\mu(S_{z,\rho}(\tilde \Sigma_{z,\rho})))
-\cos(\tilde \theta_{z,\rho})\,h(\tilde \Sigma_{z,\rho})(\mu(\tilde \Sigma_{z,\rho}),\mu(\tilde \Sigma_{z,\rho}))
\\&\qquad =- 
\nabla^{\mathbb{R}^2}\nabla^{\mathbb{R}^2}\tilde u_{z,\rho}(y,y)+O(1)\,\rho^3.
	\end{align*} 
	By  \eqref{ u definition}, 
	$$
	- \nabla^{\mathbb{R}^2}\nabla^{\mathbb{R}^2}\tilde u_{z,\rho}(y)(y,y)=\frac{\rho\,H(S)(z)}{2}
	$$
	on $\partial D$. Moreover,	by Lemma \ref{rho MC},
	$$
\sin\tilde \theta_{z,\rho}=\frac{\rho\,H(p)}{2}+O(1)\,\rho^2.
	$$
	
By \eqref{enclosed area first},
$$\frac{d}{ds}\bigg|_{s=0} |S(\tilde \Sigma_{{z,\rho}}(s\,u^\perp))|=-\int_{\partial \tilde \Sigma_{z,\rho}}\frac{1}{\cos\tilde \theta_{z,\rho}}\, u^\perp. $$
	By Lemma \ref{rho MC},
	$$
	-\int_{\partial \tilde \Sigma_{z,\rho}}\frac{1}{\cos\tilde \theta_{z,\rho}}\, u^\perp=\int_{\partial D}u^\perp+O(1)\,\rho^2\,\int_{\partial D}|u^\perp|.
	$$
	
These estimates imply that $\mathcal{D}\mathcal{F}_{z,\rho}|_{0}$ 	converges  strongly to
	$$
	L:C^{2,\alpha}(D)\cap \Lambda^1(\partial D)^\perp \to C^{0,\alpha}(D)\times \left(  C^{1,\alpha}(\partial D)\cap \Lambda^{0,1}(\partial D)^\perp\right) \times \mathbb{R}
	$$
 	given by 
	$$
	L(v)=\left(\Delta^{\mathbb{R}^2} v,\operatorname{proj}_{\Lambda^{0,1}(\partial D)^\perp}(y\cdot \nabla^{\mathbb{R}^2}v-v),\operatorname{proj}_{\Lambda^0(\partial D)}v\right)
	$$
	as $\rho\searrow 0$.
We claim that $L$ is an isomorphism. Indeed, by the comment below \cite[Theorem 6.31]{GilbargTrudinger}, $L$ is a Fredholm operator of index $0$. Let $v\in C^{2,\alpha}(D)\cap \Lambda^1(\partial  D)^\perp$ be such that $L(v)=0$. In particular, $v|_{\partial D}\perp \Lambda^0(\partial D)$.  By \cite[Example 1.3.1]{Spectralgeometry}, 
 $v\in \Lambda^1(\partial D)$. Thus, $v=0$.   By the Fredholm alternative, $L$ is an isomorphism. It follows that  $\mathcal{D}\mathcal{F}_{z,\rho}|_{0}$ is also an isomorphism provided that $\rho>0$ is sufficiently small. 
	
\end{proof}

\begin{prop} \label{IFT existence 2}
	There are $\varepsilon>0$ and $0<\rho_0<1$ with the following property. let $z\in S$ and $0<\rho<\rho_0$. There is  $u_{z,\rho}^\perp \in C^{2,\alpha}(D) \cap\Lambda^1(D)^\perp $ with $|u_{z,\rho}^\perp|_{C^{2,\alpha}(D)}=O(1)\,\rho$ such that 
	\begin{align} \label{IFT EXT 2 Label} 
	\hat \Sigma_{z,\rho}=\tilde \Sigma_{z,\rho}(u_{z,\rho}^\perp)
	\end{align} 
	satisfies 
	\begin{align*}
		&\circ\qquad\text{$H(\hat \Sigma_{z,\rho})=0$,}\\
		&\circ\qquad\text{ $\cot\theta_{z,\rho} \in \Lambda^{0,1}(\partial D)$ where $\theta_{z,\rho}=\tilde \theta_{z,\rho}(u^\perp_{z,\rho})$, and}\\
		&\circ\qquad\text{ $|S_{z,\rho}(\hat \Sigma_{z,\rho})|=\pi$.
		}
	\end{align*}
		
	Conversely, if $u^\perp\in C^{2,\alpha}(D)$ with $|u^\perp|_{C^{2,\alpha}(D)}<\varepsilon$ is such that
		\begin{align*}
		&\circ\qquad\text{$H(\tilde\Sigma_{z,\rho}(u^\perp))=0$,}\\
		&\circ\qquad\text{ $\cot\tilde\theta_{z,\rho}(u^\perp) \in \Lambda^{0,1}(\partial D)$, and}\\
		&\circ\qquad\text{ $| S_{z,\rho}(\tilde \Sigma_{z,\rho}(u^\perp))|=\pi$,
		}
	\end{align*}
	then $u^\perp= u_{z,\rho}^\perp$.
\end{prop}
\begin{proof}
	This follows from the inverse function theorem, using \eqref{almost solution 2} and Lemma \ref{isomorphism 2}. 
\end{proof}
We define  $G_{\rho}:S\to \mathbb{R}$ by 
\begin{align} \label{Grho} 
G_{\rho}(z)=\frac{16}{\pi}\,\frac{|\hat \Sigma_{z,\rho}|-\pi}{\rho^2}.
\end{align} 
Note that all estimates leading to Proposition \ref{IFT existence 2} depend smoothly on $z\in S$. The inverse function theorem as applied in the proof of Proposition \ref{IFT existence 2} therefore implies that $\hat \Sigma_{z,\rho}$ depends smoothly on $z\in S$. This shows that $G_\rho \in C^\infty(S)$; see also  \cite[Proposition 6]{BrendleEichmair}. Likewise, we see that $G_\rho$ also depends smoothly on $\rho$.

In the following lemma, $\nabla ^S$ indicates differentiation with respect to $z$ and the dash indicates differentiation with respect to $\rho$.

\begin{lem} \label{expansion}
There is $F_\rho \in C^\infty(S)$ that depends smoothly on $\rho$ such that   
$$
G_{\rho}(z)=-H(S)(z)^2+F_\rho(z)
$$
where 
$$
\sum_{\ell=0}^2|(\nabla^S)^\ell F_\rho|=O(1)\,\rho^2\qquad \text{and}\qquad \sum_{\ell=0}^2|(\nabla^S)^\ell F'_\rho|=O(1)\,\rho.
$$

\end{lem}
\begin{proof}
The initial normal speed of the variation $\{\tilde \Sigma_{z,\rho}(s\,u_{z,\rho}^\perp):s\in(-1,1)\}$  is $(-\tan\tilde \theta_{z,\rho})\,u_{z,\rho}^\perp$.
	
By Taylor's theorem, using Lemma \ref{local graph}, Lemma \ref{u estimates}, Lemma \ref{rho MC}, Proposition \ref{IFT existence 2}, \eqref{area first}, and \eqref{enclosed area first}, 
\begin{align*}
	|\hat \Sigma_{z,\rho}|-\pi&=|\tilde \Sigma_{z,\rho}|-|S_{z,\rho}(\tilde \Sigma_{z,\rho})|\\
	&\qquad-\int_{\tilde \Sigma_{z,\rho}}\tan(\tilde\theta_{z,\rho})\,H(\tilde \Sigma_{z,\rho})\,u^\perp_{z,\rho}+\int_{\partial \tilde \Sigma_{z,\rho}}u_{z,\rho}^\perp+\int_{\partial \tilde \Sigma_{z,\rho}}\frac{1}{\cos\tilde \theta_{z,\rho}}\,u^\perp_{z,\rho}+O(1)\,\rho^4.
	\end{align*}
By Lemma \ref{local graph}, Lemma \ref{u estimates}, and Lemma \ref{rho MC}
$$
\int_{\tilde \Sigma_{z,\rho}}\tan(\tilde \theta_{z,\rho})\,H(\tilde \Sigma_{z,\rho})\,u^\perp_{z,\rho}=O(1)\,\rho^4.
$$
By Lemma \ref{local graph}, Lemma \ref{u estimates}, and Lemma \ref{rho MC}, 
$$
\int_{\partial \tilde \Sigma_{z,\rho}}u_{z,\rho}^\perp+\int_{\partial \tilde \Sigma_{z,\rho}}\frac{1}{\cos\tilde \theta_{z,\rho}}\,u^\perp_{z,\rho}=O(1)\,\rho^2\,\int_{\partial \tilde \Sigma_{z,\rho}}u^\perp_{z,\rho}+O(1)\,\rho^4.
$$
Moreover, by Taylor's theorem, using Lemma \ref{local graph}, Lemma \ref{u estimates}, Lemma \ref{rho MC}, Proposition \ref{IFT existence 2},  and \eqref{enclosed area first},
$$
\pi-|S_{z,\rho}(\tilde \Sigma_{z,\rho})|=\int_{\partial \tilde \Sigma_{z,\rho}}u_{z,\rho}^\perp+O(1)\,\rho^3.
$$
In conjunction with Lemma \ref{closeness}, we conclude that
$$
\int_{\partial \tilde \Sigma_{z,\rho}}u_{z,\rho}^\perp=O(1)\,\rho^2.
$$
It follows that 
$$|\hat \Sigma_{z,\rho}|-\pi=|\tilde \Sigma_{z,\rho}|-|S_{z,\rho}(\tilde \Sigma_{z,\rho})|+O(1)\,\rho^4.$$
The assertion now follows from Lemma \ref{closeness} and the corresponding estimates for the derivatives of $F_\rho$.
\end{proof}
\begin{lem} \label{LS reduction}
The following holds provided  $\rho>0$ is sufficiently small. Let $z_0\in S$. Then
	$\hat \Sigma_{z_0,\rho}$ is a minimal capillary surface supported on $S_{z_0,\rho}$ if and only if $\nabla^S G_{\rho}(z_0)=0$. 
\end{lem}
\begin{proof}
This follows as in the proof of \cite[Lemma 19]{cmc}. Indeed,  by Proposition \ref{IFT existence 2}, $\hat \Sigma_{z_0,\rho}$ is a minimal capillary surface supported on $S_{z_0,\rho}$ if and only if
\begin{align} \label{condition}  
\int_{\partial D}(e_i\cdot y)\cot\theta_{z_0,\rho}=0
\end{align} 
for $i=1,\,2$. By Lemma \ref{initial co-normal}, Proposition \ref{IFT existence 2}, and \eqref{area first},  \eqref{condition} holds if and only if $\nabla^SG_\rho(z_0)=0$.

\end{proof}

\begin{prop} \label{mean curvature consequences}
Let $z_k\in S$ and $\rho_k>0$ be such that $\rho_k\searrow 0$ and $z_k\to z$ for some $z\in S$. Assume that each  $\hat \Sigma_{z_k,\rho_k}$ is a minimal capillary surface supported on $S_{z_k,\rho_k}$. Then $H(S)(z)=0$. If, in addition, each $\hat \Sigma_{z_k,\rho_k}$ is weakly stable for the free energy, then $\nabla^{S}\nabla^SH(S)(z)\leq 0$. 
\end{prop}
\begin{proof}
Assume  that each $\hat \Sigma_{z_k,\rho_k}$ is a minimal capillary surface supported on $S_{z_k,\rho_k}$.		By Lemma \ref{expansion} and Lemma \ref{LS reduction},  $\nabla^S H(S)(z_k)=O(1)\,\rho_k^2$. Thus, $\nabla^S H(S)(z)=0$. 
	
Assume in addition that each  $\hat \Sigma_{z_k,\rho_k}$ is weakly stable for the free energy. By Corollary \ref{weakly stable and second derivative test}, using that $|S_{z,\rho_k}(\hat \Sigma_{z,\rho_k})|=\pi$ for all $z\in S$, there holds 
	$
	\nabla^S\nabla^S G_{\rho_k}(z_k)\geq 0.
	$
In conjunction with Lemma \ref{expansion}, we conclude that $\nabla^S\nabla^S H(S)(z)\leq 0$, as asserted.
\end{proof}
 \label{LS2}
\section{Proof of Theorem \ref{THM D} and Theorem \ref{THM C}}
	Recall from Appendix \ref{appendix support surfaces} the definition of a support surface $S\subset \mathbb{R}^3$ with unit normal $\nu(S)$  and of the domain $D(S)\subset \mathbb{R}^3$ bounded by $S$.  Moreover, recall   our conventions for the second fundamental form $h(S)$, the shape operator $A(S)$, and the mean curvature $H(S)$ of $S$.

Recall from Appendix \ref{appendix minimal capillary surfaces}  the  definition of a minimal capillary surface $\Sigma \subset \mathbb{R}^3$ supported on a support surface $S\subset \mathbb{R}^3$  and that of its  wetting surface $S(\Sigma)\subset S$. In particular, recall that  $\Sigma\setminus \partial \Sigma \subset D(S)$.  Moreover, recall the concepts of  stability and weak stability for the free energy  of a minimal capillary surface.

In this section, we consider a closed surface $S\subset \mathbb{R}^3$ with outward unit normal $\nu(S)$ and mean curvature $H(S)>0$.

In this section, we prove Theorem \ref{THM D} and Theorem \ref{THM C}.

\begin{proof}[Proof of Theorem \ref{THM D}]
	We continue to use the notation introduced in Section \ref{LS2}.
	
	Let $z\in S$ be such that $\nabla^S H(S)(z)=0$ and $\nabla^S\nabla^SH(S)(z)<0$. Given $\rho_0>0$ sufficiently small and $0<\rho<\rho_0$, let ${G}_\rho: S\to \mathbb{R}$ be as in \eqref{Grho}. Moreover, let ${G}_0:S\to \mathbb{R}$ be given by ${G}_0(z)=-H(S)(z)^2$. 
	 By Lemma \ref{expansion}, $G_\rho$ converges to $G_0$ in $C^2(S)$ as $\rho\nearrow 0$. Since $\nabla^S H(S)(z)=0$ and $\nabla^S\nabla^SH(S)(z)<0$, there holds $\nabla^S G_0(z)=0$ and $\nabla^S\nabla^S G_0(z)>0$. By the inverse function theorem, there is $\delta>0$ such that, shrinking $\rho_0>0$ if necessary, there is a unique $z(\rho) \in S$  with $\operatorname{dist}_S(z(\rho),z)<\delta$ and $\nabla^S G_\rho(z(\rho))=0$ for every $0<\rho<\rho_0$. Moreover, there holds $\operatorname{dist}_S(z(\rho),z)=O(1)\,\rho^2$ and
	 \begin{align} \label{convexity}
	 	\liminf_{\rho\searrow 0}\inf\{\nabla^S\nabla^SG_\rho(z(\rho))(\upsilon,\upsilon):\upsilon\in T_{z(\rho)}(S)\text{ and }|\upsilon|=1\}>0.
	 \end{align}
	 By Lemma \ref{LS reduction}, $$\Sigma(\theta(\rho))=\rho\,R_{z(\rho)}\hat  \Sigma_{z(\rho),\rho}+z(\rho)$$
	 with $\hat  \Sigma_{z(\rho),\rho}$  as in \eqref{IFT EXT 2 Label} is a minimal capillary surface supported on $S$ with capillary angle $\sfrac{\pi}2<\theta(\rho)<\pi$.  Moreover, by  Lemma \ref{local graph}, Lemma \ref{u estimates}, and Proposition \ref{IFT existence 2}, $\sfrac{1}{\rho}\,(\Sigma(\theta(\rho))-z(\rho))$ converges smoothly to $\{y\in \mathbb{R}^2:|y|\leq 1\}\times\{0\}$. By Taylor's theorem, using Lemma \ref{rho MC}, Proposition \ref{IFT existence 2}, and Lemma \ref{capillary change}, there holds 
	$$
	\cos\theta(\rho)=-1+\frac{H^2(S)(z(\rho))}{8}\,\rho^2+O(1)\,\rho^4.
	$$
	This shows that, shrinking $\rho_0>0$ if necessary, $\theta(\rho)$ is a decreasing function of $\rho$ and that $$\sin\theta(\rho)=\frac{H(S)(z)}{2}\,\rho+O(1)\,\rho^2.$$
	In particular, 
	\begin{align} \label{smoothly close}
	\text{$
	\frac{1}{\sin\theta(\rho)}\,\Sigma(\theta(\rho))
	$
	is smoothly close to a round disk of radius $\frac{2}{H(S)(z)}$ as $\rho\searrow  0$.}
	\end{align} 
	
	 Differentiating the equation $\nabla^S G(z(\rho))=0$ with respect to $\rho$, using Lemma \ref{expansion} and \eqref{convexity}, we obtain that 
	$$
	z'(\rho)=O(1)\,\rho.
	$$
	 Arguing as in the proof of \cite[Proposition 49]{largearea} and using \eqref{smoothly close},  
we see that the co-normal speed of the variation	  $\{S(\Sigma(\theta(\rho))) :\rho\in(0,\rho_0)\}$ is  $-1+O(1)\,\rho.$ 	  
	   Let $f(\rho)\in C^\infty(\Sigma(\theta(\rho)))$ be the normal speed of the variation $\{\Sigma(\theta(\rho)):\rho\in(0,\rho_0) \}$. By \eqref{speed comparison}, $f(\rho)=-(1+o(1))\sin\theta(\rho)$ on $\partial \Sigma(\theta(\rho))$. Moreover, by the Jacobi equation, using that each $\Sigma(\theta(\rho))$ is minimal, 
	 $$
	 -\Delta^{\Sigma(\theta(\rho))}-|h(\Sigma(\theta(\rho)))|^2\,f(\rho)=0.
	 $$
	 In view of \eqref{smoothly close}, the maximum principle and a standard compactness argument show that $f(\rho)<0$ provided that $\rho>0$ is  sufficiently small. This shows that for $\rho_0>0$ sufficiently small, the surfaces $\{\Sigma(\theta(\rho)):\rho\in(0,\rho_0) \}$ form a smooth foliation.
	 
	 Finally, note that $\Sigma(\theta(\rho))$ is weakly stable for the free energy if and only if $\hat \Sigma_{z(\rho),\rho}$ is weakly stable for the free energy. To see that the latter holds, we  argue as in \cite[p.~1052--1053]{cmc}. We  recall the three main steps of this argument. As in Section \ref{LS2}, we identify functions on $\hat \Sigma_{z(\rho),\rho}$ with functions on $D$.
	 
Let $f\in C^\infty(D).$	  On the one hand, as in the proof of Lemma \ref{stable coro}, we see that  
	 \begin{equation*} 
	 \begin{aligned} 
	 &	\int_{\hat \Sigma_{z(\rho),\rho}} |\nabla^{\hat \Sigma_{z(\rho),\rho}} f|^2-\int_{\hat \Sigma_{z(\rho),\rho}} |h(\hat \Sigma_{z(\rho),\rho})|^2\,f^2\\&\qquad +\int_{\partial \hat \Sigma_{z(\rho),\rho}} k(\hat \Sigma_{z(\rho),\rho})\cdot\mu(\hat \Sigma_{z(\rho),\rho})\,f^2-\frac{1}{\sin\theta(\rho)}\,\int_{ \partial \hat \Sigma_{z(\rho),\rho}} H(S_{z(\rho),\rho})\,f^2
	 \\&\to  \int_{D}|\nabla^{\mathbb{R}^2} f|^2-\int_{\partial D}f^2
	 \end{aligned}
	 \end{equation*}
	 as $\rho\searrow 0$. By \cite[Example 1.3.1]{Spectralgeometry},
	 $$
	 \int_{D}|\nabla^{\mathbb{R}^2} f|^2-\int_{\partial D}f^2\geq \int_{\partial D} f^2
	 $$
	 provided that $f\perp \Lambda^0(D)\oplus \Lambda^1(D)$.

	  On the other hand, using Lemma \ref{initial co-normal}, Lemma \ref{second derivative test}, and  \eqref{convexity}, we see that   there is $\varepsilon>0$ such that  
	 \begin{align*}
		&	\int_{\hat \Sigma_{z(\rho),\rho}} |\nabla^{\hat \Sigma_{z(\rho),\rho}} f|^2-\int_{\hat \Sigma_{z(\rho),\rho}} |h(\hat \Sigma_{z(\rho),\rho})|^2\,f^2\\&\qquad +\int_{ \partial \hat \Sigma_{z(\rho),\rho}} k(\hat \Sigma_{z(\rho),\rho})\cdot\mu(\hat \Sigma_{z(\rho),\rho})\,f^2-\frac{1}{\sin\theta(\rho)}\,\int_{ \partial\hat \Sigma_{z(\rho),\rho}} H(S_{z(\rho),\rho})\,f^2
		\\&\geq \varepsilon\,\rho^2 \int_{\partial D}\,f^2 
	\end{align*} 
	provided that $f\in \Lambda^1(D)$. 
	
	Finally, note that  $f=f_0+f_1+f_2$ for unique $f_0\in \Lambda^0(D)$, $f_1\in \Lambda^1(D)$, and $f_2\perp \Lambda^0(D)\oplus \Lambda^1(D)$ and that 
	$$
	f^2_0=O(1)\,\rho\,\int_{\partial D}f^2
	$$
	provided that
		$$
	\int_{\partial \hat \Sigma_{z(\rho),\rho}}f=0.
	$$

 This completes the proof of Theorem \ref{THM D}.
\end{proof}
\begin{proof}[Proof of Theorem \ref{THM C}]
		We continue to use the notation introduced in Section \ref{LS2}.

	Let
$\Sigma_k\subset \mathbb{R}^3$ be  minimal capillary surfaces supported on $S$ with capillary angle $\sfrac{\pi}2<\theta_k<\pi$. Assume that each $\Sigma_k$ is weakly stable for the free energy and that $\theta_k\nearrow \pi$. By Proposition \ref{blowup}, passing to a subsequence, there is $z\in S$ such that 
$$
\frac{H(S)(z)}{2\sin\theta_k}\,(\Sigma_k-z) 
$$
is smoothly close to a centered round disk of radius $1$. Thus, 
$$
\rho_k=(1+o(1))\,\frac{2\sin\theta_k}{H(S)(z)}
$$ where $\rho_k>0$ is such that $|S(\Sigma_k)|=\pi\,\rho_k^2$. 
 By Corollary \ref{blowup coro}, Lemma \ref{tangential correction},  and Proposition \ref{IFT existence 2}, there are $z_k\in S$ such that
$$
\Sigma_k=\rho_k\,R_{z_k}\hat \Sigma_{z_k,\rho_k}+z_k.
$$
By Lemma \ref{LS reduction}, $\nabla^S {G}_{\rho_k}(z_k)=0$. By Proposition \ref{mean curvature consequences}, $\nabla^SH(S)(z)=0$ and $\nabla^S\nabla^S H(S)(z)\leq 0$. 

Assume in addition that $H(S)$ has no degenerate maxima. It follows that $\nabla^S\nabla^S H(S)(z)< 0$. Since $\operatorname{dist}_S(z_k,z)=o(1)$,   we conclude that $z_k=z(\rho_k)$ where $z(\rho_k)$ is as in the proof of Theorem \ref{THM D}. It follows that $\Sigma_k=\Sigma(\theta(\rho_k))$ and that $\theta(\rho_k)=\theta_k$ where $\theta(\rho_k)$ and $\Sigma(\theta(\rho_k))$ are as in the proof of Theorem \ref{THM D}.

This completes the proof of Theorem \ref{THM C}.
\end{proof}

\begin{appendices} 
	
	\section{Support surfaces}
	\label{appendix support surfaces}
	
	Let $S\subset \mathbb{R}^3$ be a properly embedded surface  such that $\mathbb{R}^3\setminus S$ has exactly two connected components. In particular, $\partial S=\emptyset$. Let  $\nu(S)$ be a unit normal of $S$ and $D(S)$ be the component of $\mathbb{R}^3\setminus S$ that $\nu(S)$ points out of.  We call $D(S)$ the domain bounded by $S$.
	
	 We denote the second fundamental form and mean curvature of $S$ with respect to $\nu(S)$ by $h(S)$ and $H(S)$, i.e.,
	 $$
	 h(S)(X,Y)=X\cdot D_Y\nu(S)
	 $$
for all $X,\,Y\in \mathfrak{X}(S)$. The shape operator $A(S):\mathfrak{X}(S)\to \mathfrak{X}(S)$ of $S$ is the tensor defined by 
	$$
	A(S) X=D_X\nu(S).
	$$
 Given $y\in S$, there is $r>0$ such that
 $$
  |y-t\,\nu(S)(y)-z|>|t|
 $$
 for all $t\in(-r,r)$ and $z\in S$ with $z\neq y$. The reach $i_S(y)\in(0,\infty]$ of $S$ at $y$ is the supremum of all $r>0$ with this property.   Note that  
	\begin{equation} 
	\begin{aligned} \label{reach upper bound} 
	|h(S)(y)|\leq \frac{\sqrt{2}}{i_S(y)}\qquad \text{and}\qquad
	|A(S)(y)\upsilon|\leq \frac{|\upsilon|}{i_S(y)}\text{ for every $\upsilon \in T_y \Sigma$}
	\end{aligned}
	\end{equation}   
	for all $y\in S$. Note that  $i_S:S\to(0,\infty]$ is lower semi-continuous. It follows that 
	$$
	U(S)=\{y+t\,\nu(S)(y):y\in S\text{ and } t\in (-i_S(y),i_S(y))\}
	$$
	is an open neighborhood of $S$ in $\mathbb{R}^3$. 
	
	We call $S$ a support surface if the reach
	$$
	i_S=\inf\{i_S(y):y\in S\}
	$$ 
of $S$ is positive.

	Given a support surface $S\subset \mathbb{R}^3$, let  $\Pi_S:\mathbb{R}^3\to S$ be the nearest point projection and $\operatorname{dist}(\,\cdot\,,S):\mathbb{R}^3\to \mathbb{R}$ the signed distance function to $S$ that is positive in $D(S)$. Note that $\Pi_S$ and $\operatorname{dist}(\,\cdot\,,S)$  are smooth in $U(S)$.
	
	 Given $x\in U(S)$ and $\upsilon\in \mathbb{R}^3$, let $$\upsilon^{\top(S)(x)}=\upsilon-(\nu(S)(\Pi_S(x))\cdot \upsilon)\,\nu(S)(\Pi_S(x)).$$ Note that $\upsilon^{\top(S)(x)}\in T_{\Pi_S(x)}S$. We define the self-adjoint tangent space automorphism $P_S(x):T_{\Pi_S(x)}S\to T_{\Pi_S(x)}S$ by $$P_S(x)\upsilon=\upsilon-\operatorname{dist}(x,S)\,A(S)(\Pi_S(x))\upsilon.$$
	
	\begin{lem}
		Let $S\subset \mathbb{R}^3$ be a support surface and $\upsilon,\,\xi\in \mathbb{R}^3$. There holds 
		\begin{align}
			D\operatorname{dist}(\,\cdot\,,S)(x)(\upsilon)&=-\upsilon\cdot \nu(S)(\Pi_S(x))\text{ and} 	\label{first distance} \\
		\notag	D^2\operatorname{dist}(\,\cdot\,,S)(x)(\upsilon,\xi)&=-h(S)(\Pi_S(x))(P_S(x)^{-1}\upsilon^{\top(S)(x)},\xi^{\top(S)(x)})
		\end{align}
 for all $x\in U(S)$.		Moreover,
			\begin{align} \label{first projection}
			D\Pi_S(x)(\upsilon)&=P_S(x)^{-1}\upsilon^{\top(S)(x)}.
		\end{align}
	\end{lem} 
		\begin{coro} 
		Let $S\subset \mathbb{R}^3$ be a support surface. There holds 
		\begin{align} 
			\label{hessian and gradient}
			D^2\operatorname{dist}(\,\cdot\,,S)(D\operatorname{dist}(\,\cdot\,,S),\upsilon )&=0
		\end{align}
		for all $x\in U(S)$ and $\upsilon \in \mathbb{R}^3$.  
There holds  
		\begin{align} \label{distance hessian estimate}
			|D^2\operatorname{dist}(\,\cdot\,,S)(x)|\leq 2\,|h(S)(\Pi_S(x))|
		\end{align}
		for all $x\in \mathbb{R}^3$ with $2\,|\operatorname{dist}(x,S)|\leq i_S(\Pi_S(x))$.
	\end{coro} 

	\begin{lem} \label{escape}
		Let $\Lambda>1$ and $r_2>r_1>0$. There is $\bar L>0$ with the following property. Assume that $S\subset \mathbb{R}^3$ is a support surface such that both $S_{r_2}(0)$ and $S_{r_1}(0)$ intersect $S$  transversely and that 
		\begin{align} \label{geometry bound escape} 
		1\leq \Lambda\,i_S(y)
		\end{align} 
		for all $y\in B_{r_2}(0)\cap S$.
 Assume that $y_1,\,y_2$ lie in the same connected component of $B_{r_1}(0)\cap S$. There is  a smooth curve $\beta:[0, T]\to B_{r_1}(0) \cap S$ of length less than $\bar L$ with $\beta(0)=y_1$ and $\beta(T)=y_2$.  
	
\end{lem}
\begin{proof}
	By \eqref{geometry bound escape}, there is $0<\varepsilon<\min\{r_2-r_1,r_1\}$ that depends on $\Lambda$, for every $y\in B_{r_1}(0)\cap S$, $B_{\varepsilon}(y)\cap S$ has only one component with intrinsic diameter at most $4\,\varepsilon$.  Note that  $B_{r_1}(0)\cap S$ can be covered by the integer ceiling of $$\frac{\frac{4\,\pi}{3}\,r_1^3}{\frac{4\,\pi}{3}\,\left(\frac{\varepsilon}{3}\right)^3}=27\,\frac{r_1^3}{\varepsilon^3}$$ such balls. The assertion of the lemma follows with $\bar L= 108\,\sfrac{r_1^3}{\varepsilon^2}$.
\end{proof}
Given $u \in C^\infty(S)$ such that  $|u(y)|<i_S(y)$ for every $y\in S$, let $$S(u)=\{y-u(y)\,\nu(S)(y):y\in S\}.$$ Note that $S(u)$ is a properly embedded hypersurface and that $\mathbb{R}^3\setminus S(u)$ has exactly two components. We orient $S(u)$ by the unit normal that points toward the component that contains all $x\in \mathbb{R}^3$ with $\operatorname{dist}(x,S)>i_S(\Pi_S(x))$. We define the diffeomorphism $\Phi_S(u):S\to S(u)$ by 
\begin{align} \label{Phi S} \Phi_S(u)(y)=y-u(y)\,\nu(S)(y)
	\end{align}  
and tacitly identify geometric quantities on $S(u)$ with geometric quantities on  $S$ by precomposition with $\Phi_S(u)$. 
 	\begin{lem} \label{graph}
 		Let $S\subset \mathbb{R}^3$ be a support surface and $u \in C^\infty(S)$ be such that  $|u(y)|<i_S(y)$ for every $y\in S$. There holds 
 		$$
 		\nu(S(u))=\frac{-\nu(S)-P_S(\Phi_S(u))^{-1}\nabla^S u}{(1+|P_S(\Phi_S(u))^{-1}\nabla^S u|^2)^{\sfrac12}}.
 		$$	
 		Moreover, given  $X,\,Y\in \mathfrak{X}(S)$,
 		\begin{align*}
 	&(1+|P_S(\Phi_S(u))^{-1}\nabla^S u|^2)^{\sfrac12}\,h(S(u))\big(P_S(\Phi_S(u))\,X-\nabla^S u\cdot X\,\nu(S),P_S(\Phi_S(u))\,Y-\nabla^S u\cdot Y\,\nu(S)\big)
 	\\&\qquad =-h(S)(X,P_S(\Phi_S(u))Y)
 	 -\nabla^S\nabla^S u(X,Y) -u\,\langle(\nabla^S_XA(S))Y,P_S(\Phi_S(u))^{-1}\nabla^S u\rangle\\&\qquad \qquad -\nabla^S_Xu\,\langle A(S)Y,P_S(\Phi_S(u))^{-1}\nabla^S u\rangle.	
 		\end{align*}
 	
 	\end{lem}
 		\begin{coro} \label{graph 2}
 		Let $S\subset \mathbb{R}^3$ be a support surface and $u \in C^\infty(S)$ be such that  $|u(y)|<i_S(y)$ for every $y\in S$. Let $X,\,Y\in \mathfrak{X}(S)$.  For all $y\in S$ with $u(y)=0$,  there holds 
 		$$
 		\nu(S(u))=\frac{-\nu(S)-\nabla^S u}{(1+|\nabla^S u|^2)^{\sfrac12}}.
 		$$	
 		and	\begin{align*}
 			&(1+|\nabla^S u|^2)^{\sfrac12}\,h(S(u))\big(X-\nabla^S u\cdot X\,\nu(S),Y-\nabla^S u\cdot Y\,\nu(S)\big)
 			\\&\qquad =-h(S)(X,Y)
 			-\nabla^S\nabla^S u(X,Y)  -\nabla^S_Xu\,h(S)(Y,\nabla^S u).	
 		\end{align*}
 		
 	\end{coro}
 	\section{Admissible surfaces and minimal capillary surfaces}
 	
 	\label{appendix minimal capillary surfaces}
 Recall from Appendix \ref{appendix support surfaces} the definition of a support surface $S\subset \mathbb{R}^3$ with unit normal $\nu(S)$ and of the domain $D(S)$ bounded by a support surface $S\subset \mathbb{R}^3$. 	Moreover, recall   our conventions for the second fundamental form $h(S)$, the shape operator $A(S)$, and the mean curvature $H(S)$ of $S$. 
 
 Let $S\subset \mathbb{R}^3$ be a support surface.

Let $\Sigma \subset \operatorname{cl}D(S)$ be a properly embedded two-sided surface that intersects $S$ transversely along $\partial \Sigma$ and is separating in $D(S)$. In particular, $\Sigma\setminus \partial \Sigma$ is the relative boundary in $D(S)$ of an open set $\Omega\subset D(S)$.  Let $\nu(\Sigma)$ be a unit normal of $\Sigma$ and $\mu(\Sigma)$ the outward co-normal of $\partial \Sigma\subset \Sigma$. We say that $\Sigma$ is an admissible surface supported on $S$ if $\nu(\Sigma)$ either points out of $\Omega$ or into $\Omega$ and if either $\nu(S)\cdot \nu(\Sigma)>0$ on $\partial \Sigma$ or $\nu(S)\cdot \nu(\Sigma)<0$ on $\partial \Sigma$.

	\begin{lem} \label{wetting surface}
		Let $\Sigma \subset \operatorname{cl} D(S)$ be an admissible surface supported on $S$. Then there is a closed top-dimensional submanifold $S(\Sigma)$ of $S$ with boundary equal to $\partial \Sigma$ such that $\mu(S(\Sigma))\cdot \mu(\Sigma)>0$ on $\partial \Sigma$ where $\mu(S(\Sigma))$ is the outward pointing co-normal of $\partial \Sigma\subset S(\Sigma)$.
	\end{lem}
	\begin{proof}
We may and will assume that $\nu(\Sigma)$ points out of $\Omega$.

Assume first that $\nu(S)\cdot \nu(\Sigma)<0$ on $\partial \Sigma$. Let $S(\Sigma)=(\operatorname{cl}\Omega) \cap S$. Since $\Sigma$ and $S$ intersect transversely along $\partial \Sigma$ and $\Sigma$ is properly embedded, $S(\Sigma)$ is a top-dimensional properly embedded submanifold of $S$ with boundary equal to $\partial \Sigma$.
 Since $\Sigma$ and $S$ intersect transversely along $\partial \Sigma$ and $\nu(S)$ points out of $D(S)$, there holds $\nu(S)\cdot \mu(\Sigma)>0$ on $\partial \Sigma$.	 Moreover, since $\nu(\Sigma)$ points out of $\Omega$ and $S$ and $\Sigma$ intersect transversely along $\partial \Sigma$, there holds $\mu(S(\Sigma))\cdot \nu(\Sigma)>0$ on $\partial \Sigma$. Note that 
$$
\mu(S(\Sigma))=(\nu(\Sigma)\cdot \mu(S(\Sigma)))\,\nu(\Sigma)+(\mu(\Sigma)\cdot\mu(S(\Sigma)))\,\mu(\Sigma).
$$
Using that $\nu(S)\cdot \mu(S(\Sigma))=0$ and $\nu(\Sigma)\cdot \nu(S)<0$, we obtain
$$
\mu(S(\Sigma))\cdot \mu(\Sigma)=-\frac{\nu(\Sigma)\cdot \mu(S(\Sigma))\,\nu(S)\cdot\nu(\Sigma)}{\nu(S)\cdot \mu(\Sigma)}>0,
$$
as asserted.

Assume now that $\nu(S)\cdot \nu(\Sigma)>0$ on $\partial \Sigma$. We then take $S(\Sigma)=\operatorname{cl}(D(S)\setminus \operatorname{cl}\Omega)\cap S$. With this choice, the argument proceeds as above.

This completes the proof of the lemma.
	\end{proof}
	Given an admissible surface $\Sigma \subset \operatorname{cl}D(S)$, we call $S(\Sigma)$ as in Lemma \ref{wetting surface} the wetting surface of $\Sigma$. Note that the definition of $S(\Sigma)$ only depends on $\Sigma$ and not on the choice of $\Omega$ or $\nu(\Sigma)$.

	Let $\Sigma\subset \operatorname{cl}D(S)$ be an admissible surface supported on $S$. Since $\Sigma$ and $S$ intersect transversely, there is $\theta\in C^\infty(\partial\Sigma,(0,\pi))$ such that 
	$$
	\nu(S)\cdot\nu(\Sigma)=\cos\theta.
	$$
	 	The tangent $e( \Sigma)$  of $\partial \Sigma$ is the vector field  given by $e( \Sigma)=\nu(\Sigma)\times \mu(\Sigma)$ on $\partial \Sigma$.
	Note that $e( \Sigma)$ is tangent to both $S$ and $\Sigma$.
	
			\begin{lem}  \label{angles}
		The following relations hold on $\partial \Sigma$. 
		\begin{equation*}
			\begin{aligned} 
				\nu(S)&=\cos(\theta)\,  \nu(\Sigma)+\sin(\theta) \,\mu(\Sigma)\qquad\qquad\qquad\qquad\qquad\qquad \\
				\mu(S(\Sigma))&=\sin(\theta) \,\nu(\Sigma)-\cos(\theta) \,\mu(\Sigma)\\
				\nu(\Sigma)&=\cos(\theta)\,  \nu(S)+\sin(\theta) \,\mu(S(\Sigma))\\
				\mu(\Sigma)&=\sin(\theta) \,\nu(S)-\cos(\theta) \,\mu(S(\Sigma))
			\end{aligned}
		\end{equation*}
		\end{lem}
		\begin{proof}
			This is \cite[Lemma 41]{eichmair2024penrose}.
		\end{proof}
		We denote the second fundamental form and mean curvature of $\Sigma$ with respect to $\nu(\Sigma)$ by  $h(\Sigma)$ and $H(\Sigma)$. 	We denote the  curvature vector of $\partial \Sigma$ by	$k( \Sigma)$. Our convention is such that if $\alpha:I\to\partial \Sigma$ is a smooth curve with unit speed, then $k(\Sigma)\circ \alpha=\alpha''$.
		\begin{lem}
			There holds 
			\begin{equation} \label{curvature} 
				\begin{aligned} 
					(\sin\theta)^2\,	k( \Sigma)&=h(\Sigma)(e( \Sigma),e( \Sigma))\,\left(\nu(\Sigma)-\cos(\theta)\,\nu(S)\right)
					\\&\qquad +h(S)(e( \Sigma),e( \Sigma))\,\left(\nu(S)-\cos(\theta)\,\nu(\Sigma)\right).
				\end{aligned} 
			\end{equation}
			and
			\begin{align*} 
				k(\Sigma)\cdot\mu(\Sigma)=\frac{1}{\sin\theta}\,h(S)(e( \Sigma),e( \Sigma))-\cot(\theta)\,h(\Sigma)(e( \Sigma),e( \Sigma)).
			\end{align*}
		\end{lem}
	\begin{proof}
		This is \cite[Lemma 41]{eichmair2024penrose}.
	\end{proof}	
		Let $\varepsilon>0$ and $F\in C^\infty(\Sigma\times(-\varepsilon,\varepsilon), \operatorname{cl} D(S))$ be such that $F(x,0)=x$ for all $x\in \Sigma$ and $F(x,s)\in S$ for every $x\in \partial \Sigma$ and $s\in(-\varepsilon,\varepsilon)$.  	Decreasing $\varepsilon>0$ if necessary, we obtain a smooth family $\{\Sigma(s):s\in(-\varepsilon,\varepsilon)\}$ of admissible surfaces $\Sigma(s)=\{F(x,s):x\in \Sigma\}$, which we call an admissible variation of $\Sigma$.
		 We tacitly identify functions and maps defined on $\Sigma(s)$ with functions and maps defined on $\Sigma$ by precomposition with $F(\,\cdot\,,s)$. We call $f\in C^\infty(\Sigma)$ given by 
		 $$
		 f(x)=\nu(\Sigma)(x)\cdot(\partial_sF)(x,0)$$
		 the initial normal speed of the variation $F$ and $X\in C^\infty(\Sigma,\mathbb{R}^3)$ given by 
		 $$
		 X(x)=(\partial_sF)(x,0)-f(x)\,\nu(\Sigma)(x)
		 $$
		 the initial tangential velocity of $F$. 	Reparametrizing $F$, we may assume that  $X|_{\partial \Sigma}$ is parallel to $\mu(\Sigma)$ so that 
	\begin{align} \label{speed comparison} 
		X=-\cot(\theta)\,f\,\mu( \Sigma).
	\end{align}
	Given $U\subset \mathbb{R}^3$ open and bounded, we say that an admissible variation $\{\Sigma(s):s\in(-\varepsilon,\varepsilon)\}$ is supported in $U$ if $(\Sigma(s)\setminus \Sigma) \cup (\Sigma\setminus \Sigma(s))\Subset U$ for all $s\in(-\varepsilon,\varepsilon)$.
		\begin{lem} \label{first variation}
			Let $U\subset \mathbb{R}^3$ be open and bounded and $\{\Sigma(s):s\in(-\varepsilon,\varepsilon)\}$ an admissible variation of $\Sigma$ that is supported in $U$.			 There holds
		\begin{align} \label{area first}  
			\frac{d}{ds}\bigg|_{s=0}|U\cap\Sigma(s)|=\int_{\Sigma}H(\Sigma)\,f-\int_{\partial\Sigma} \cot(\theta)\,f
		\end{align} 
		and
		\begin{align} \label{enclosed area first}
			\frac{d}{ds}\bigg|_{s=0} |U\cap S(\Sigma(s))|=\int_{\partial \Sigma} \frac{1}{\sin\theta}\,f.
		\end{align} 
		\end{lem}
		\begin{proof}
			This is \cite[Lemma 44]{eichmair2024penrose}.
		\end{proof}
		\begin{lem} \label{capillary change}
Let $\{\Sigma(s):s\in(-\varepsilon,\varepsilon)\}$ be an admissible variation of $\Sigma$.			There holds
			\begin{align*}
				&\frac{d}{ds}\bigg|_{s=0}\nu(S)\cdot  \nu(\Sigma(s))   \\&\quad =-\sin(\theta)\,\mu(\Sigma)\cdot \nabla^\Sigma f+h(S)(\mu(S(\Sigma)),\mu(S(\Sigma)))\,f-\cos(\theta)\,h(\Sigma)(\mu(\Sigma),\mu(\Sigma))\,f
			\end{align*}
			 on $\partial \Sigma$.
		\end{lem}
		\begin{proof}
			This is \cite[Lemma 43]{eichmair2024penrose}.
		\end{proof}

	An admissible surface $\Sigma\subset \mathbb{R}^3$ supported on $S$ is called a minimal capillary surface supported on $S$ with capillary angle $0<\theta<\pi$ if $\nu(\Sigma)\cdot \nu(S)=\cos\theta$ on $\partial \Sigma$, i.e., $\Sigma$ and $S$ intersect at a constant angle along $\partial \Sigma$. 
	
Given $0<\theta<\pi$ and  $U\subset \mathbb{R}^3$ open and bounded, the free energy $J_\theta(\Sigma,U)$ with angle $\theta$ of an admissible surface $\Sigma\subset \mathbb{R}^3$ supported on $S$ in $U$ is 
$$
J_\theta(\Sigma,U)=|U\cap \Sigma|+\cos(\theta)\,|U\cap S(\Sigma)|.
$$
By Lemma \ref{first variation}, a necessary and sufficient condition for $\Sigma$ to be a minimal capillary surface supported on $S$ with capillary angle $0<\theta<\pi$ is that
$$
\frac{d}{ds}\bigg|_{s=0}J_\theta(\Sigma(s),U)=0
$$
for every  $U\subset \mathbb{R}^3$ open and bounded and every admissible variation $\{\Sigma(s):s\in(-\varepsilon,\varepsilon)\}$ that is supported in $U$.

Given  $U\subset \mathbb{R}^3$ open and bounded, we say that a minimal capillary surface $\Sigma\subset \mathbb{R}^3$ supported on $S$ with capillary angle $0<\theta<\pi$ is stable for the free energy in $U$ if
	\begin{align} \label{stability} 
\int_\Sigma |\nabla^\Sigma f|^2\geq 		\int_\Sigma |h(\Sigma)|^2\,f^2-\int_{\partial \Sigma} k(\Sigma)\cdot\mu(\Sigma)\,f^2+\frac{1}{\sin\theta}\,\int_{\partial \Sigma} H(S)\,f^2
	\end{align} 
	for every $f\in C^\infty(\Sigma)$ with compact support in $U\cap \Sigma$. 
	Here, 	 $\nabla^\Sigma f$ is the gradient of $f$. We say that $\Sigma$ is weakly stable for the free energy in $U$ if \eqref{stability} holds for all $f\in C^\infty(\Sigma)$ with compact support in $U\cap \Sigma$ and
$$
\int_{\partial \Sigma} f=0.
$$	
$\Sigma$ is called stable respectively weakly stable for the free energy if $\Sigma$ is  stable respectively weakly stable for the free energy in $U$ for every $U\subset \mathbb{R}^3$ open and bounded.
\begin{lem} \label{second derivative test}
Let $U\subset \mathbb{R}^3$ be open and bounded and $\{\Sigma(s):s\in(-\varepsilon,\varepsilon)\}$ an admissible variation of $\Sigma$ that is supported in $U$.	 There holds 	$$
	\frac{d^2}{ds^2}\bigg|_{s=0}J_\theta(\Sigma(s),U)=\int_\Sigma |\nabla^\Sigma f|^2- 		\int_\Sigma |h(\Sigma)|^2\,f^2+\int_{\partial \Sigma} k(\Sigma)\cdot\mu(\Sigma)\,f^2-\frac{1}{\sin\theta}\,\int_{\partial \Sigma} H(S)\,f^2.
	$$
\end{lem}
\begin{proof}
	This is \cite[(46)]{eichmair2024penrose}.
\end{proof}

\begin{coro} 
	Let $\Sigma\subset \mathbb{R}^3$ be a minimal capillary surface supported on $S$ with capillary angle $0<\theta<\pi$. Let $U\subset \mathbb{R}^3$ be open and bounded. Then $\Sigma$ is stable for the free energy in $U$ if and only if 
	$$
	\frac{d^2}{ds^2}\bigg|_{s=0}J_\theta(\Sigma(s),U)\geq 0
	$$ 
	for every admissible variation $\{\Sigma(s):s\in(-\varepsilon,\varepsilon)\}$ of $\Sigma$ that is supported in $U$. 
\end{coro}

\begin{coro} \label{weakly stable and second derivative test}
	Let $\Sigma\subset \mathbb{R}^3$ be a minimal capillary surface supported on $S$ with capillary angle $0<\theta<\pi$. Let $U\subset \mathbb{R}^3$ be open and bounded. Then $\Sigma$ is weakly stable for the free energy in $U$ if and only if 
	$$
\frac{d^2}{ds^2}\bigg|_{s=0}J_\theta(\Sigma(s),U)\geq 0
	$$ 
	for every admissible variation $\{\Sigma(s):s\in(-\varepsilon,\varepsilon)\}$ of $\Sigma$ that is supported in $U$ and such that 
	$$
	\frac{d}{ds}\bigg|_{s=0}|S(\Sigma(s))|=0.
	$$
\end{coro}

\section{A partial differential equation for the distance function}
	Recall from Appendix \ref{appendix support surfaces} the definition of a support surface $S\subset \mathbb{R}^3$ with unit normal $\nu(S)$  and of the domain $D(S)\subset \mathbb{R}^3$ bounded by $S$.  Recall that $\Pi_S:\mathbb{R}^3\to S$ is the nearest point projection and that $\operatorname{dist}(\,\cdot\,,S)$ is the signed distance function to $S$ that is positive in $D(S)$. Moreover, recall the definition of the reach $i_S$ of $S$ and our conventions for the second fundamental form $h(S)$, the shape operator $A(S)$, and the mean curvature $H(S)$ of $S$.

	\begin{lem} \label{distance super harmonic}

 Let $S\subset \mathbb{R}^3$ be a support surface and  
$U\subset \mathbb{R}^3$  open. Assume that there is $\Lambda>1$ such that 
$$
1\leq \Lambda\,i_S(y)
$$
for all $y\in U\cap S$. Let $\Sigma\subset D(S)$ be a two-sided properly embedded minimal surface with  $U\cap \partial \Sigma=\emptyset$.  There are $w\in L^\infty(\Sigma)$ and $X\in L^\infty(\Sigma,\mathbb{R}^3)$ such that, on $$\{x\in \Sigma:\Pi_S(x)\in U\cap S\text{ and }2\,\Lambda\operatorname{dist}(x,S)<1\},$$ $|w|\leq 4\,\Lambda^2$, $|X|\leq 4\,\Lambda$,  and
		$$
		-\Delta^\Sigma \operatorname{dist}(\,\cdot\,,S)+X\cdot \nabla^\Sigma \operatorname{dist}(\,\cdot\,,S)+w\operatorname{dist}(\,\cdot\,,S)=H(S)\circ \Pi_S.
		$$ 
	\end{lem}
	\begin{proof}
		Given $t\in(0,\sfrac{1}{2\,\Lambda})$, let $S_t=\{x\in D(S):\operatorname{dist}(x,S)=t\}$. Note that $U\cap S_t$ is a smooth two-sided hypersurface and that  $(D\operatorname{dist}(\,\cdot\,,S))|_{S_t}$ is a unit normal of $S_t$. We orient $S_t$ by choosing $\nu(S_t)=-(D\operatorname{dist}(\,\cdot\,,S))|_{S_t}$. Using \eqref{hessian and gradient}, we have   
		\begin{equation*}  
		\begin{aligned}   
		\Delta \operatorname{dist}(\,\cdot\,,S)&=\Delta^{S_t}\operatorname{dist}(\,\cdot\,,S)+D^2\operatorname{dist}(\,\cdot\,,S)(\nu(S_t),\nu(S_t))+H(S_t)\,D\operatorname{dist}(\,\cdot\,,S)\cdot\,\nu(S_t)\\&=-H(S_t).
		\end{aligned} 
		\end{equation*}	
		By Lemma \ref{graph}, the principal curvatures of $S_t$ at $x\in S_t$ are 
		$$
		\frac{\kappa_1}{1-t\,\kappa_1}\qquad\text{and}\qquad \frac{\kappa_2}{1-t\,\kappa_2},
		$$
		where $\kappa_1$ and $\kappa_2$ are the principal curvatures of $S$ at $\Pi_S(x)$. Thus,
		$$
		\frac{H(S_t)(x)-H(S)(\Pi_S(x))}{t}=\frac{\kappa_1^2}{1-t\,\kappa_1}+\frac{\kappa_2^2}{1-t\,\kappa_2}.
		$$
		Using \eqref{reach upper bound} and that $2\,t<i_S(\Pi_S(x))$ for every $x\in U\cap S_t$, we conclude that
		$$
		0\leq 	\frac{H(S_t)(x)-H(S)(\Pi_S(x))}{t}\leq 4\,\Lambda^2.
		$$
		We choose $w\in L^\infty(\Sigma)$ such that, on $U\cap \Sigma$, 
		$$
		w=\frac{H(S)(\Pi_S(x))-H(S_{\operatorname{dist}(x,S)})(x)}{\operatorname{dist}(x,S)}.
		$$
		Note that, on $U\cap \Sigma$, 
		\begin{align} \label{subharmonic}
\Delta \operatorname{dist}(\,\cdot\,,S)=w\operatorname{dist}(\,\cdot\,,S)-H(S)\circ \Pi_S.
		\end{align}
		
Let $\nu(\Sigma)$ be a choice of unit normal for $\Sigma$.	Note that
		$$
		D\operatorname{dist}(\,\cdot\,,S)=\nabla^\Sigma \operatorname{dist}(\,\cdot\,,S)+(\nu(\Sigma)\cdot D\operatorname{dist}(\,\cdot\,,S))\,\nu(\Sigma).
		$$ 
		and
	
		$$
		|\nu(\Sigma)\cdot D\operatorname{dist}(\,\cdot\,,S)|=(1-|\nabla^{\Sigma}\operatorname{dist}(\,\cdot\,,S)|^2)^{\sfrac12}.
		$$
		In conjunction with \eqref{hessian and gradient}, we see that 
		\begin{align*} 
		&D^2 \operatorname{dist}(\,\cdot\,,S)(\nu(\Sigma),\nu(\Sigma))\\&\qquad =-D^2 \operatorname{dist}(\,\cdot\,,S)(\nabla^\Sigma \operatorname{dist}(\,\cdot\,,S),\nu(\Sigma))\\&\qquad \qquad +\left(1-(1-|\nabla^{\Sigma}\operatorname{dist}(\,\cdot\,,S)|^2)^{\sfrac12}\right)\,D^2\operatorname{dist}(\,\cdot\,,S)(\nu(\Sigma),\nu(\Sigma))
		\end{align*} 
		if $D\operatorname{dist}(\,\cdot\,,S)\cdot \nu(\Sigma)\geq 0$ and 
			\begin{align*} 
			&D^2 \operatorname{dist}(\,\cdot\,,S)(\nu(\Sigma),\nu(\Sigma))\\&\qquad =D^2 \operatorname{dist}(\,\cdot\,,S)(\nabla^\Sigma \operatorname{dist}(\,\cdot\,,S),\nu(\Sigma))\\&\qquad \qquad +\left(1-(1-|\nabla^{\Sigma}\operatorname{dist}(\,\cdot\,,S)|^2)^{\sfrac12}\right)\,D^2\operatorname{dist}(\,\cdot\,,S)(\nu(\Sigma),\nu(\Sigma))
		\end{align*} 
		if $D\operatorname{dist}(\,\cdot\,,S)\cdot \nu(\Sigma)<0$.
		
		Let $Y\in L^\infty(U\cap \Sigma,\mathbb{R}^3)$   be given by	
		$$
		Y=\begin{dcases}-\nu(\Sigma)\lrcorner D^2 \operatorname{dist}(\,\cdot\,,S)\qquad&\text{if $D\operatorname{dist}(\,\cdot\,,S)\cdot \nu(\Sigma)\geq 0$ and}\\
			\nu(\Sigma)\lrcorner D^2 \operatorname{dist}(\,\cdot\,,S) &\text{if $D\operatorname{dist}(\,\cdot\,,S)\cdot \nu(\Sigma)< 0$}.
		\end{dcases}
		$$
	Note that  $|Y|\leq |D^2 \operatorname{dist}(\,\cdot\,,S)|$. Let $Z\in L^\infty(U\cap \Sigma,\mathbb{R}^3)$ be given by	
		$$
		Z=\frac{1-(1-|\nabla^{\Sigma}\operatorname{dist}(\,\cdot\,,S)|^2)^{\sfrac12}}{|\nabla^\Sigma\operatorname{dist}(\,\cdot\,,S)|^2}\,D^2\operatorname{dist}(\,\cdot\,,S)(\nu(\Sigma),\nu(\Sigma))\,\nabla^\Sigma \operatorname{dist}(\,\cdot\,,S)
		$$
		if $\nabla^\Sigma \operatorname{dist}(\,\cdot\,,S)\neq 0$ and otherwise $Z=0$. 		 Note that  $|Z|\leq |D^2 \operatorname{dist}(\,\cdot\,,S)|$. 
		
		Let  $X\in L^\infty(\Sigma,\mathbb{R}^3)$ be such that  $X=-Y-Z$ on $U\cap \Sigma$. Note that, on $U\cap \Sigma$,
		\begin{align} \label{hessian conversion}
			D^2 \operatorname{dist}(\,\cdot\,,S)(\nu(\Sigma),\nu(\Sigma))=-X\cdot \nabla^\Sigma \operatorname{dist}(\,\cdot\,,S).
		\end{align}
		
		Since $\Sigma$ is minimal, 
			$$
		\Delta^\Sigma \operatorname{dist}(\,\cdot\,,S)=\Delta \operatorname{dist}(\,\cdot\,,S)-D^2 \operatorname{dist}(\,\cdot\,,S)(\nu(\Sigma),\nu(\Sigma)).
		$$
		The assertion follows in conjunction with \eqref{distance hessian estimate}, \eqref{subharmonic}, and \eqref{hessian conversion}.
	\end{proof}

		\section{Stable minimal surfaces}
	Recall from Appendix \ref{appendix support surfaces} the definition of a support surface $S\subset \mathbb{R}^3$  with unit normal $\nu(S)$ and of the domain $D(S)\subset \mathbb{R}^3$ bounded by $S$.  Recall that $\Pi_S:\mathbb{R}^3\to S$ is the nearest point projection and that $\operatorname{dist}(\,\cdot\,,S)$ is the signed distance function to $S$ that is positive in $D(S)$. Moreover, recall the definition of the reach $i_S$ of $S$ and our conventions for the second fundamental form $h(S)$, the shape operator $A(S)$, and the mean curvature $H(S)$ of $S$.

Let $\Sigma\subset \mathbb{R}^3$ be a  two-sided properly embedded minimal surface. An open subset $V\subset \Sigma\setminus \partial \Sigma$ is called stable if  
$$
\int_\Sigma |h(\Sigma)|^2\,f^2\leq \int_\Sigma |\nabla^\Sigma f|^2
$$  
for all $f\in C^\infty(\Sigma)$ with compact support in $V$. We use $\operatorname{dist}_\Sigma$ to denote the intrinsic distance function of $\Sigma$. Given $x\in \Sigma$ and $r>0$, we agree that $B_r^\Sigma(x)=\{z\in \Sigma:\operatorname{dist}_\Sigma(z,x)<r\}$.
		\begin{lem} \label{lem: stable implies interior curvature estimates}
			There is $c>1$ with the following property. Let $\Sigma \subset \mathbb{R}^3$ be a two-sided properly embedded minimal surface with $0\in \Sigma$ and $r>0$ such that $B^\Sigma_{2\,r}(0)\cap \partial \Sigma=\emptyset$ and  $B^\Sigma_{2\,r}(0)$ is stable. Then 
			$$
			|h(\Sigma)(x)|\leq \frac{c}{r}
			$$ 
			for every $x\in B^\Sigma_{r}(0)$. 
		\end{lem}
		\begin{proof}
		This is, e.g., a special case of \cite[Theorem 2.10]{ColdingMinicozzi}.
		\end{proof}
	
		\begin{lem} \label{lem:stable -> local graphicality}
			Given $\varepsilon>0$, there is $\delta\in(0,1)$ with the following property. Let  $\Sigma\subset \mathbb{R}^3$ be a two-sided properly embedded minimal surface with $0\in  \Sigma$ and $r>0$ such that  $B^\Sigma_{2\,r}(0)\cap \partial \Sigma=\emptyset$ and $B^\Sigma_{2\,r}(0)$ is stable.    For every $x\in B^\Sigma_{r}(0)$, there is  $u\in C^\infty(T_x\Sigma)$ satisfying 
			$$
			r^{-1}\,|u|+|\nabla^{T_x\Sigma}u|+r\,|\nabla^{T_x\Sigma}\nabla^{T_x\Sigma}u|\leq \varepsilon 
			$$
			such that 
			\begin{align} \label{coordinates} B^\Sigma_{\delta\,r}(x)\subset \{x+y+u(y)\,\nu(\Sigma)(x):y\in T_x\Sigma\}.
				\end{align} 
		\end{lem}
		\begin{proof}
		This follows from Lemma \ref{lem: stable implies interior curvature estimates} and, e.g., \cite[Lemma 2.4]{ColdingMinicozzi}.
		\end{proof}

\begin{lem} \label{harnack inequality}
	Given $A>1$, there is $c>1$ with the following property. Let  $\Sigma\subset \mathbb{R}^3$ be a two-sided properly embedded minimal surface with $0\in  \Sigma$ and $r>0$ such that  $B^\Sigma_{2\,r}(0)\cap \partial \Sigma=\emptyset$ and  $B^\Sigma_{2\,r}(0)$ is stable.  Assume that $f\in C^\infty(B^\Sigma_{2\,r}(0))$ is positive and satisfies
	$$-\Delta^\Sigma f+X\cdot \nabla^\Sigma f+w\,f=q$$
	for some $w\in L^\infty(\Sigma)$ with $r^2\,|w|\leq A^2,$ $X\in L^\infty(\Sigma,\mathbb{R}^3)$  with $r\,|X|\leq A,$ and $q\in L^\infty(\Sigma)$.
	 Then 
	\begin{align}
		\sup\{f(x):x\in B^\Sigma_{r}(0)\}\leq c\,\inf \{f(x):x\in B^\Sigma_{r}(0)\}+c\,r^2\,\sup\{|q(x)|:x\in B^\Sigma_{2\,r}(0)\}.
	\end{align} 
\end{lem}
\begin{proof}
	By scaling, we may assume that $r=1$.
	
	 Applying Lemma \ref{lem:stable -> local graphicality} with $\varepsilon=1/2$, we see that, in $B^\Sigma_{\delta}(0)$, parametrized as in \eqref{coordinates},  $f$ is a solution of a uniformly elliptic partial differential equation in divergence form with bounded coefficients, bounded lower order terms, and bounded source term.  
	 The assertion  follows from the Harnack inequality for such operators, see \cite[Theorem 8.17 and Theorem 8.18]{GilbargTrudinger}, and a standard covering argument. 
	 
\end{proof}

\begin{lem} \label{slab lemma}
Given $B>1$, there is $c>1$ such that, for all sufficiently small $0<\varepsilon<\sfrac{1}{2\,B}$, the following holds.  Let $r>0$, $S\subset \mathbb{R}^3$ be a support surface,   
 and $\Sigma\subset D(S)$ be a two-sided properly embedded minimal surface with $0\in \Sigma$ such that
 \begin{align}
 	&\circ\qquad r<B\,i_S(y)\text{ for all $y\in B_{2\,r}(0)\cap S$,} \notag \\
 	&\circ \qquad \text{$B^\Sigma_{2\,r}(0)\cap \partial \Sigma=\emptyset,$   $B^\Sigma_{2\,r}(0)$ is stable, and}\notag \\
 	&\circ\qquad  \operatorname{dist}(x,S)\leq r\,\varepsilon \text{  for all $x\in B^\Sigma_{2\,r}(0)$.}\notag 
 \end{align}
   There is   $u\in C^\infty(S)$ positive such that 
\begin{align*} 
r^{-1}\,|u|+|\nabla^S u|+r\,|\nabla^S\nabla^Su|\leq c\,\varepsilon
\end{align*} 
and $B^\Sigma_{r}(0)\subset \{y-u(y)\,\nu(S)(y):y\in S\}$.\end{lem} 

\begin{proof}
		By scaling, we may assume that $r=1$. 
		
		Suppose, for a contradiction, that there are numbers $\varepsilon_k>0$ with $\varepsilon_k=o(1)$, support surfaces $S_k\subset \mathbb{R}^3$,  
		and  two-sided properly embedded minimal surfaces $\Sigma_k\subset D(S_k)$ such that
		\begin{align} \label{reach and curvature bound}
		\liminf_{k\to\infty} \inf\{i_{S_k}(y):y\in B_2(0)\}>0,
		\end{align}   $B^{\Sigma_k}_2(0)\cap \partial \Sigma_k=\emptyset$,  $B^{\Sigma_k}_2(0)$ is stable, and $\operatorname{dist}(x,S_k)\leq \varepsilon_k$ for all $x\in B^{\Sigma_k}_2(0)$, but there are no positive functions $u_k\in C^\infty(S_k)$ with 
		$$
		|u_k|+|\nabla^{S_k} u_k|+|\nabla^{S_k}\nabla^{S_k}u_k|\leq k\,\varepsilon_k
		$$
		such that $B^{\Sigma_k}_{1}(0)\subset \{y-u_k(y)\,\nu(S_k)(y):y\in S_k\}$.

	By \eqref{reach upper bound} and \eqref{reach and curvature bound}, passing to a subsequence, there is a properly embedded $C^{1,1}$ surface $S\subset \mathbb{R}^3$ 
	 such that $S_k$ converges in $C^{1,\alpha}_{loc}$ in $B_2(0)$ to    $ S$ for every $0<\alpha<1$. Moreover, using Lemma \ref{lem: stable implies interior curvature estimates} and Lemma \ref{lem:stable -> local graphicality}, we see that, passing to a further subsequence, $\operatorname{cl} B^{\Sigma_k}_{\sfrac74}(0)$ converges smoothly to a compact minimal surface $\Sigma\subset \mathbb{R}^3$ with $\operatorname{dist}(z,S)=0$ for all $z\in \Sigma$, i.e., $\Sigma\subset S$. In particular, there are positive $u_k \in C^\infty(S_k)$ with
		\begin{align} \label{a priori}
	|u_k|+|\nabla ^{S_k}u_k|=o(1)\qquad\text{and}\qquad 	|\nabla^{S_k}\nabla^{S_k}u_k|=O(1)
		\end{align}  such that $B^{\Sigma_k}_{\sfrac32}(0)\subset \{y- u_k(y)\,\nu(S_k)(y):y\in S_k\}$. Using that $$0<u_k(y)= \operatorname{dist}(y-u_k(y)\,\nu(S_k)(y),S)<\varepsilon_k$$ whenever $y-u_k(y)\,\nu(S_k)(y)\in B^{\Sigma_k}_{\sfrac54}(0)$ and \eqref{a priori}, we see that  $\nabla^{S_k} u_k(y)= O(1)\,\varepsilon_k$ whenever $y-u_k(y)\,\nu(S_k)(y)\in B^{\Sigma_k}_{\sfrac54}(0)$. By the interior Schauder estimates for elliptic equations \cite[Theorem 6.2]{GilbargTrudinger}, we conclude that $\nabla^{S_k}\,\nabla^{S_k} u_k(y)=O(1)\,\varepsilon_k$ whenever $y-u_k(y)\,\nu(S_k)(y)\in B^{\Sigma_k}_{1}(0)$, contradicting our supposition.

		This completes the proof of the lemma.
\end{proof}
\section{Generalized domains, weighted potentials, and stability}
\label{appendix: generalized domain}
	Recall from Appendix \ref{appendix support surfaces} the definition of a support surface $S\subset \mathbb{R}^3$ with unit normal $\nu(S)$. Moreover, recall  our convention for the second fundamental form $h(S)$ of $S$.

 Let $S\subset \mathbb{R}^3$ be a support surface. Let $U_1,\,U_2,\,\ldots\subset \mathbb{R}^3$ be open sets with smooth boundary such that $U_\ell\Subset U_{\ell+1}$ and  $\partial U_\ell$ intersects $S$ transversely for all  $\ell\geq 1$. Let
$$
U=\bigcup_{\ell=1}^\infty U_\ell.
$$

Given $\Omega\subset U\cap S$ relatively open, we denote by $\partial \Omega$ the topological boundary of $\Omega$ in $U\cap S$. A relatively open subset $\Omega \subset U\cap S$ is called a generalized domain in $U\cap S$ if for every $p\in U\cap \partial \Omega$ there are 
\begin{equation} \label{setup} 
\begin{aligned}
	&\circ \qquad \text{$\varepsilon>0$,}\\
	&\circ \qquad \text{a chart $\Phi:(-\varepsilon,\varepsilon)\times (-\varepsilon,\varepsilon)\to S$ whose images contains $p$, and}\\
	&\circ \qquad\text{$h^1,\,h^2\in C^{1,1}((-\varepsilon,\varepsilon))$ with $-\varepsilon<h^1(a)\leq h^2(a)<\varepsilon$ for all $a\in(-\varepsilon,\varepsilon)$}
\end{aligned}
\end{equation} 
  such that either $$\Omega\cap\Phi((-\varepsilon,\varepsilon)\times (-\varepsilon,\varepsilon))=\Phi(\{(a,b)\in(-\varepsilon,\varepsilon)\times (-\varepsilon,\varepsilon): b<h^1(a)\})$$ or  $$\Omega\cap\Phi((-\varepsilon,\varepsilon)\times (-\varepsilon,\varepsilon))=\Phi(\{(a,b)\in(-\varepsilon,\varepsilon)\times (-\varepsilon,\varepsilon): b<h^1(a)\text{ or } h^2(a)<b\}).$$ 
  We call $p$ an intersection point of $\partial \Omega$ if the first alternative does not occur for any choice of chart $\Phi$ as above. We call $\Omega$ a domain in $U\cap S$ if $\partial \Omega$ has no intersection points.
    	\begin{figure}\centering
  	\includegraphics[width=.9\linewidth]{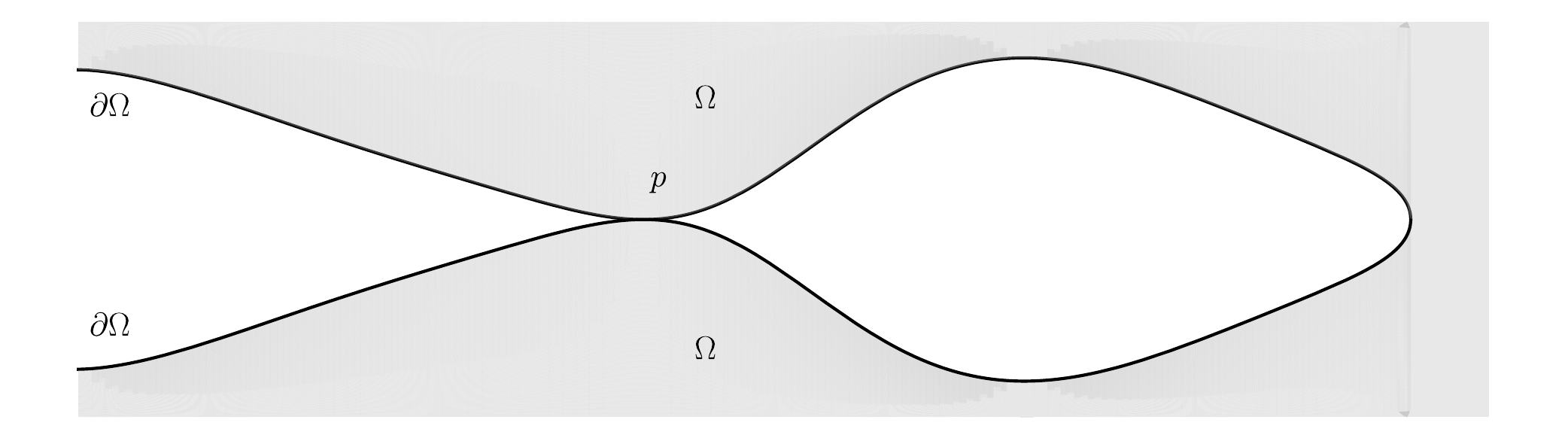}
  	\caption{The shaded region illustrates a generalized domain $\Omega$ in $\mathbb{R}^2$ with an intersection point $p$. 
  	}
  	\label{exterior_surface}
  \end{figure}

For each fixed $\ell\geq 1$, consider the metric completion of $U_\ell\cap \Omega$ equipped with the intrinsic length metric. Let $\tilde \Omega$ be the union of all these completions. Note that $(\tilde \Omega,g)$ has the structure of a smooth Riemannian 2-manifold with $C_{loc}^{1,1}$-boundary $\tilde \partial \Omega=\tilde \Omega \setminus \Omega$ where $g$ is the Riemannian metric induced by the restriction of the Euclidean metric to $S$. We call $(\tilde \Omega,g)$ the Riemannian 2-manifold associated to the generalized domain $\Omega$. 

 We record the following lemma.
\begin{lem}\label{complete}
	$(\tilde \Omega,g)$ is complete if and only if $\operatorname{cl}\Omega\subset U$.
\end{lem}
  $\tilde \partial \Omega$ consists of  one copy  of all elements of $U\cap \partial \Omega$ that are not intersection points and of two copies of all elements of $\partial \Omega$ that are intersection points. We will denote elements of $\tilde \partial \Omega$ and the corresponding  elements of $\partial \Omega$ by the same symbol if  there is no ambiguity.

Let $\mu(\Omega)$ be the outward co-normal of $\tilde \partial \Omega$. Since $\tilde \partial \Omega$ is $C_{loc}^{1,1}$, the geodesic curvature $$k( \Omega)\cdot\mu(\Omega)$$ of $\tilde \partial \Omega$ with respect to $\mu(\Omega)$ exists at almost every point.  

Let $L(S)\in C^\infty(S)$. We say that $(\tilde \Omega,g)$ is stable with respect to $L(S)$ if  
\begin{align} \label{generalized stability} 
	\int_{ \Omega} |h(S)|^2\,f^2+\int_{\tilde \partial \Omega} L(S)\,f^2\leq \int_{ \Omega} |\nabla^S f|^2+\int_{\tilde \partial \Omega} k(\Omega)\cdot \mu(\Omega)\,f^2
\end{align} 
holds for all $f\in C^1(\tilde \Omega)$ with compact support. 
We say that $(\tilde \Omega,g)$ is weakly stable if $\eqref{generalized stability}$ holds for all compactly supported $f\in C^1(\tilde \Omega)$ with  
$$
\int_{\tilde \partial \Omega} f=0.
$$
If $(\tilde \Omega,g)$ is stable with respect to $L(S)$ and $f\in C^1(\tilde \Omega)$ with compact support is such that 
\begin{align*} 
	\int_{\Omega} |h(S)|^2\,f^2+\int_{\tilde \partial \Omega} L(S)\,f^2= \int_\Omega |\nabla^S f|^2+\int_{\tilde \partial \Omega} k(\Omega)\cdot \mu(\Omega)\,f^2,
\end{align*} 
then, by elliptic regularity, $f\in C^\infty(\Omega)\cap C_{loc}^{1,\alpha}(\tilde \Omega)$ for every $0<\alpha<1$. Moreover, $f$ satisfies the Euler-Lagrange equation
\begin{equation} \label{Euler-Lagrange}
	\begin{dcases}	-\Delta^{S}f=|h(S)|^2f\qquad&\text{in $\Omega$ and}\\
		-\mu(\Omega)\cdot \nabla^S f=k(\Omega)\cdot \mu(\Omega)\,f-L(S)\,f&\text{almost everywhere on $\tilde \partial \Omega$}.	
	\end{dcases} 
\end{equation}

Given $0<\alpha<1$, we say that a sequence $\Omega_1,\,\Omega_2,\,\ldots\subset S$ of domains in  $U \cap S$ converges in $C_{loc}^{1,\alpha}$ to a generalized domain $\Omega$ in $U\cap S$ if   the following holds for every $p\in U\cap \partial \Omega$. Let $\varepsilon>0$, $\Phi$, $h^1$, and $h^2$ be as in \eqref{setup}. Shrinking $\varepsilon>0$ if necessary, for all $k$ sufficiently large, there are $h^1_k,\,h^2_k\in C^{1,1}((-\varepsilon,\varepsilon))$ with $-\varepsilon<h^1_k(a)<h^2_k(a)<\varepsilon$ for all $a\in(-\varepsilon,\varepsilon)$ such that $h^1_k\to h^1$ and $h^2_k\to h^2$ in $C^{1,\alpha}_{loc}((-\varepsilon,\varepsilon))$  and either
\begin{equation*}
	\begin{aligned} 
\Omega_k\cap\Phi((-\varepsilon,\varepsilon)\times(-\varepsilon,\varepsilon))&=\Phi(\{(a,b)\in(-\varepsilon,\varepsilon)\times(-\varepsilon,\varepsilon): b<h^1_k(a)\})\text{ or}\\
\Omega_k\cap\Phi((-\varepsilon,\varepsilon)\times(-\varepsilon,\varepsilon))&=\Phi(\{(a,b)\in(-\varepsilon,\varepsilon)\times(-\varepsilon,\varepsilon): b<h^1_k(a)\text{ or } h^2_k(a)<b\}).
\end{aligned} 
\end{equation*}
 Given a sequence $\Omega_1,\,\Omega_2,\,\ldots \subset S$ of domains  in $U \cap S$ that converges in $C^{1,\alpha}_{loc}$ to a generalized domain $\Omega$ in $U\cap S$, we say that a sequence $f_1,\,f_2,\,\ldots$ of functions $f_k\in C_{loc}^{1,\alpha}(\operatorname{cl}\Omega_k)$ converges in $C^{1,\alpha}_{loc}$ to $f\in C^{1,\alpha}(\tilde \Omega)$ if the following hold. Given $p\in \Omega$, there holds $f_k\to f$ in $C^{1,\alpha}_{loc}(V)$ for every $V\subset S$ relatively open with $p\in V$ such that $V\subset \Omega$ and $V\subset \Omega_k$ for all but finitely many $k$. Given $p\in \tilde \partial \Omega$ that is not an intersection point, shrinking $\varepsilon>0$ if necessary, there holds $ f^1_k\to f^1$ in  $C_{loc}^{1,\alpha}((-\varepsilon,\varepsilon)\times(-\varepsilon,0])$ where $f^1_k,\,f^1\in C_{loc}^{1,\alpha}((-\varepsilon,\varepsilon)\times(-\varepsilon,0])$ are given by
 \begin{align} \label{f1} 
 f^1_k(a,b)=f_k(\Phi((a,b+h^1_k(a))))\qquad\text{and}\qquad f^1(a,b)=f(\Phi((a,b+h^1(a)))).
 \end{align} 
Likewise, given an intersection point $p\in \tilde \partial \Omega$,  shrinking $\varepsilon>0$ if necessary, there holds $ f^1_k\to f^1$  in  $C_{loc}^{1,\alpha}((-\varepsilon,\varepsilon)\times(-\varepsilon,0])$ and $f^2_k\to f^2$in $C_{loc}^{1,\alpha}((-\varepsilon,\varepsilon)\times[0,\varepsilon))$ where $f^1_k,\,f^1\in C_{loc}^{1,\alpha}((-\varepsilon,\varepsilon)\times(-\varepsilon,0])$ are given by \eqref{f1}
and $f^2_k,\,f^2$ in  $C_{loc}^{1,\alpha}((-\varepsilon,\varepsilon)\times[0,\varepsilon))$ are given by
$$
f^2_k(a,b)=f_k(\Phi((a,b+h^2_k(a))))\qquad\text{and}\qquad f^2(a,b)=f(\Phi((a,b+h^2_k(a)))).
$$
Here, given $a\in(-\varepsilon,\varepsilon)$, we identify $\Phi(a,h^1_k(a))$ and $\Phi(a,h^2_k(a))$ with the corresponding points in $\tilde \partial \Omega$. 
 
These notions of convergence generalize naturally to the case where $S_1,\,S_2,\,\ldots\subset \mathbb{R}^3$ is a sequence of support surfaces that converges in $C^{2,\alpha}_{loc}$ to $S$,  $\Omega_k\subset S_k$ are domains in $U\cap S_k$, and $f_k\in C_{loc}^{1,\alpha}(\Omega_k)$.

\begin{lem} \label{compactness} 
	Let $S_1,\,S_2,\,\ldots\subset \mathbb{R}^3$ be support surfaces that converge in $C^{2,\alpha}_{loc}$ in $U$ to a support surface $S$ and $\Omega_k\subset S_k$ domains in $U\cap S_k$. Assume that, for each $\ell\geq 1$ fixed,
	\begin{align*}
		&\circ \qquad |U_\ell\cap\partial \Omega_k|=O(1),
		\\&\circ \qquad 	\sup\{|k( \partial \Omega_k)(y)|:y\in U_\ell\cap \partial \Omega_k\}=O(1),
	\end{align*}
	and that there is 	 $s(\ell)>0$ such that $$\exp^{S_k}_y(-s\,\mu( \Omega_k)(y))\in  \Omega_k$$ for  all $k$ sufficiently large,   $y\in U_{\ell}\cap \partial \Omega_k,$ and $s\in (0,s(\ell))$. Then there is a generalized domain $\Omega$ in $U\cap S$ such that, passing to a subsequence, $\Omega_k$ converges to $\Omega$ in $C^{1,\alpha}_{loc}$ in $U$ for every $0<\alpha<1$. 
\end{lem}
\begin{proof}
	This follows from a standard argument based on the Arzelà–Ascoli theorem.
\end{proof}

\section{Complete embedded minimal surfaces with finite total curvature}
\label{appendix: cemswft}

Let $S\subset \mathbb{R}^3$ be a connected complete embedded minimal surface with Gauss curvature $K(S)$. Recall that $S$ is said to be of finite total curvature if 
$$
-\int_S K(S)<\infty.
$$

Recall from \cite[Proposition 1]{Schoen} that, rotating $S$ if necessary, there are $\lambda>1$,   properly embedded annuli $S^1,\,\ldots, \,S^m\subset \mathbb{R}^3$, $\psi^1,\,\ldots,\, \psi^m\in C^\infty(\mathbb{R}^2)$, $a^1,\,\ldots,\, a^m\in \mathbb{R},$ and $b^1,\,\ldots,\, b^m\in \mathbb{R}$  such that the following hold.
\begin{align*}
	&\circ \qquad  \text{$S$ intersects $\{x\in \mathbb{R}^3:x_1^2+x_2^2=\lambda^2\}$ transversely.} \\
	&\circ\qquad 	S\setminus \{x\in \mathbb{R}^3:x_1^2+x_2^2=\lambda^2\}=S^1\cup\ldots\cup S^m\\
	&\circ \qquad \text{$S^i\subset \{(y,\psi^i(y)):y\in \mathbb{R}^2\}$ for all $i=1,\,\ldots,\,m$} \\
	&\circ\qquad \text{$\psi^1(y)>\ldots >\psi^m(y)$ for all $y\in \mathbb{R}^2$} \\
	&\circ \qquad \psi^i(y)=a^i+b^i\,\log|y|+O(1)\sfrac{1}{|y|}\text{ for all $i=1,\,\ldots,\,m$}\\
	&\circ \qquad |\nabla^{\mathbb{R}^2}\psi^i(y)|+|y|\,|\nabla^{\mathbb{R}^2}\nabla^{\mathbb{R}^2}\psi^i(y)|=O(1)\,\sfrac{1}{|y|}\text{ for all $i=1,\,\ldots,\,m$}
\end{align*}
We fix such a number $\lambda>1$ and  orient $S$ by the unit normal $\nu(S)$ that points down  on $S^1$. The surfaces $S^1,\,\ldots,\,S^m$ are called the ends of $S$. Note that $m$, $a^1,\,\ldots, a^m$, and $b^1,\,\ldots, b^m$ are independent of the choice of $\lambda$. We call $b^i$ the logarithmic growth rate of $S^i$.  By the half space theorem \cite{HoffmanMeekshalfspace},  either $m=1$ and $S$ is congruent to $\mathbb{R}^2\times\{0\}$ or  $m\geq 2$, $b^1>0$, and $b^m<0$; see also \cite[\S2]{Schoen}. We say that  $S$ is horizontal to emphasize that $S$ has been rotated this way. Note that $S$ is horizontal if and only if the ends of $S$ have vertical flux; see, e.g., \cite{hoffmanmeeks}.

Recall from Appendix \ref{appendix support surfaces} the definition of a support surface and of the domain bounded by a support surface. Note that $S$ is a support surface. Moreover,
note that 
\begin{align} \label{components} 
	D(S)\setminus \{x\in \mathbb{R}^3:x_1^2+x_2^2\leq \lambda^2\}
\end{align} has $\ceil{\sfrac{m}{2}+\sfrac12}$ 
components. Let $N\subset \mathbb{R}^3$ be such a component. Note that its closure $\operatorname{cl} N$ either intersects exactly one or exactly two ends of $S$. In the latter case, we call $N$ an intermediate layer of $S$. Note that $S$ has $\floor{\sfrac{m}{2}-\sfrac12}$ 
intermediate layers. In particular,  \eqref{components} has two components that are not intermediate layers if $m$ is even and \eqref{components} has one component that is not an intermediate layer if $m$ is odd. If $N$ is an intermediate layer, let $1\leq i<m$ be such that $N$ intersects $S^i$ and $S^{i+1}$. We say that $N$ has logarithmic height if $b^{i}>b^{i+1}$ and that $N$ has bounded height if $b^i=b^{i+1}$. 

\end{appendices}


\begin{bibdiv}
	\begin{biblist}
		
		\bib{AltCafarelli}{article}{
			author={Alt, Hans},
			author={Caffarelli, Luis},
			title={Existence and regularity for a minimum problem with free
				boundary},
			date={1981},
			ISSN={0075-4102,1435-5345},
			journal={J. Reine Angew. Math.},
			volume={325},
			pages={105\ndash 144},
			url={https://doi.org/10.1515/crll.1981.325.105},
			review={\MR{618549}},
		}
		
		\bib{BarbosadoCarmo}{article}{
			author={Barbosa, Jo\~ao~Lucas},
			author={do~Carmo, Manfredo},
			title={Stability of hypersurfaces with constant mean curvature},
			date={1984},
			ISSN={0025-5874,1432-1823},
			journal={Math. Z.},
			volume={185},
			number={3},
			pages={339\ndash 353},
			url={https://doi.org/10.1007/BF01215045},
			review={\MR{731682}},
		}
		
		\bib{Bocher}{article}{
			author={B\^ocher, Maxime},
			title={Singular points of functions which satisfy partial differential
				equations of the elliptic type},
			date={1903},
			ISSN={0002-9904},
			journal={Bull. Amer. Math. Soc.},
			volume={9},
			number={9},
			pages={455\ndash 465},
			url={https://doi.org/10.1090/S0002-9904-1903-01017-9},
			review={\MR{1558016}},
		}
		
		\bib{Brendle}{article}{
			author={Brendle, Simon},
			title={Constant mean curvature surfaces in warped product manifolds},
			date={2013},
			ISSN={0073-8301,1618-1913},
			journal={Publ. Math. Inst. Hautes \'Etudes Sci.},
			volume={117},
			pages={247\ndash 269},
			url={https://doi.org/10.1007/s10240-012-0047-5},
			review={\MR{3090261}},
		}
		
		\bib{Brendleiso}{article}{
			author={Brendle, Simon},
			title={The isoperimetric inequality for a minimal submanifold in
				{E}uclidean space},
			date={2021},
			ISSN={0894-0347,1088-6834},
			journal={J. Amer. Math. Soc.},
			volume={34},
			number={2},
			pages={595\ndash 603},
			url={https://doi.org/10.1090/jams/969},
			review={\MR{4280868}},
		}
		
		\bib{EichmairBrendle}{article}{
			author={Brendle, Simon},
			author={Eichmair, Michael},
			title={Large outlying stable constant mean curvature spheres in initial
				data sets},
			date={2014},
			ISSN={0020-9910,1432-1297},
			journal={Invent. Math.},
			volume={197},
			number={3},
			pages={663\ndash 682},
			url={https://doi.org/10.1007/s00222-013-0494-8},
			review={\MR{3251832}},
		}
		
		\bib{BrendleEichmair}{article}{
			author={Brendle, Simon},
			author={Eichmair, Michael},
			title={Large outlying stable constant mean curvature spheres in initial
				data sets},
			date={2014},
			ISSN={0020-9910,1432-1297},
			journal={Invent. Math.},
			volume={197},
			number={3},
			pages={663\ndash 682},
			url={https://doi.org/10.1007/s00222-013-0494-8},
			review={\MR{3251832}},
		}
		
		\bib{EichmairBrendleIso}{article}{
			author={Brendle, Simon},
			author={Eichmair, Michael},
			title={Proof of the {M}ichael-{S}imon-{S}obolev inequality using optimal
				transport},
			date={2023},
			ISSN={0075-4102,1435-5345},
			journal={J. Reine Angew. Math.},
			volume={804},
			pages={1\ndash 10},
			url={https://doi.org/10.1515/crelle-2023-0054},
			review={\MR{4661529}},
		}
		
		\bib{CafarelliFriedman}{article}{
			author={Caffarelli, Luis},
			author={Friedman, Avner},
			title={Regularity of the boundary of a capillary drop on an
				inhomogeneous plane and related variational problems},
			date={1985},
			ISSN={0213-2230},
			journal={Rev. Mat. Iberoamericana},
			volume={1},
			number={1},
			pages={61\ndash 84},
			url={https://doi.org/10.4171/RMI/3},
			review={\MR{834357}},
		}
		
		\bib{Carleman}{article}{
			author={Carleman, Torsten},
			title={Zur {T}heorie der {M}inimalfl\"achen},
			date={1921},
			ISSN={0025-5874,1432-1823},
			journal={Math. Z.},
			volume={9},
			number={1-2},
			pages={154\ndash 160},
			url={https://doi.org/10.1007/BF01378342},
			review={\MR{1544458}},
		}
		
		\bib{chodoshedelenli}{article}{
			author={Chodosh, Otis},
			author={Edelen, Nick},
			author={Li, Chao},
			title={Improved regularity for minimizing capillary hypersurfaces},
			date={2025},
			ISSN={2769-8505},
			journal={Ars Inven. Anal.},
			pages={Paper No. 2, 27},
			review={\MR{4888466}},
		}
		
		\bib{chodosh2025weiss}{article}{
			author={Chodosh, Otis},
			author={Edelen, Nick},
			author={Li, Chao},
			title={Weiss monotonicity and capillary hypersurfaces},
			date={2025},
			journal={preprint, arXiv:2506.02146},
		}
		
		\bib{EichmairChodosh}{article}{
			author={Chodosh, Otis},
			author={Eichmair, Michael},
			title={Global uniqueness of large stable {CMC} spheres in asymptotically
				flat {R}iemannian 3-manifolds},
			date={2022},
			ISSN={0012-7094,1547-7398},
			journal={Duke Math. J.},
			volume={171},
			number={1},
			pages={1\ndash 31},
			url={https://doi.org/10.1215/00127094-2021-0043},
			review={\MR{4364730}},
		}
		
		\bib{ColdingMinicozzi}{book}{
			author={Colding, Tobias},
			author={Minicozzi, William},
			title={A course in minimal surfaces},
			series={Graduate Studies in Mathematics},
			publisher={American Mathematical Society, Providence, RI},
			date={2011},
			volume={121},
			ISBN={978-0-8218-5323-8},
			url={https://doi.org/10.1090/gsm/121},
			review={\MR{2780140}},
		}
		
		\bib{Costa}{article}{
			author={Costa, Celso},
			title={Example of a complete minimal immersion in {${\mathbf{R}}^3$} of
				genus one and three embedded ends},
			date={1984},
			ISSN={0100-3569},
			journal={Bol. Soc. Brasil. Mat.},
			volume={15},
			number={1-2},
			pages={47\ndash 54},
			url={https://doi.org/10.1007/BF02584707},
			review={\MR{794728}},
		}
		
		\bib{DeMasietal}{article}{
			author={De~Masi, Luigi},
			author={Edelen, Nick},
			author={Gasparetto, Carlo},
			author={Li, Chao},
			title={Regularity of minimal surfaces with capillary boundary
				conditions},
			date={2025},
			ISSN={0010-3640,1097-0312},
			journal={Comm. Pure Appl. Math.},
			volume={78},
			number={12},
			pages={2436\ndash 2502},
			url={https://doi.org/10.1002/cpa.70008},
			review={\MR{4969079}},
		}
		
		\bib{dePhilippis2}{article}{
			author={De~Philippis, Guido},
			author={Fusco, Nicola},
			author={Morini, Massimiliano},
			title={Regularity of capillarity droplets with obstacle},
			date={2024},
			ISSN={0002-9947,1088-6850},
			journal={Trans. Amer. Math. Soc.},
			volume={377},
			number={8},
			pages={5787\ndash 5835},
			url={https://doi.org/10.1090/tran/9152},
			review={\MR{4771237}},
		}
		
		\bib{DePhillipis}{article}{
			author={De~Philippis, Guido},
			author={Maggi, Francesco},
			title={Regularity of free boundaries in anisotropic capillarity problems
				and the validity of {Y}oung's law},
			date={2015},
			ISSN={0003-9527,1432-0673},
			journal={Arch. Ration. Mech. Anal.},
			volume={216},
			number={2},
			pages={473\ndash 568},
			url={https://doi.org/10.1007/s00205-014-0813-2},
			review={\MR{3317808}},
		}
		
		\bib{DeRosa}{article}{
			author={De~Rosa, Antonio},
			author={Neumayer, Robin},
			author={Resende, Reinaldo},
			title={Rigidity of critical points of hydrophobic capillary
				functionals},
			date={2025},
			journal={preprint, arXiv:2509.22532},
		}
		
		\bib{Bernstein1}{article}{
			author={do~Carmo, Manfredo},
			author={Peng, Chiakuei},
			title={Stable complete minimal surfaces in {${\bf R}\sp{3}$}\ are
				planes},
			date={1979},
			ISSN={0273-0979,1088-9485},
			journal={Bull. Amer. Math. Soc. (N.S.)},
			volume={1},
			number={6},
			pages={903\ndash 906},
			url={https://doi.org/10.1090/S0273-0979-1979-14689-5},
			review={\MR{546314}},
		}
		
		\bib{cmc}{article}{
			author={Eichmair, Michael},
			author={Koerber, Thomas},
			title={Foliations of asymptotically flat manifolds by stable constant
				mean curvature spheres},
			date={2024},
			ISSN={0022-040X,1945-743X},
			journal={J. Differential Geom.},
			volume={128},
			number={3},
			pages={1037\ndash 1083},
			url={https://doi.org/10.4310/jdg/1729092454},
			review={\MR{4810218}},
		}
		
		\bib{largearea}{article}{
			author={Eichmair, Michael},
			author={Koerber, Thomas},
			title={Large area-constrained {W}illmore surfaces in asymptotically
				{S}chwarzschild 3-manifolds},
			date={2024},
			ISSN={0022-040X,1945-743X},
			journal={J. Differential Geom.},
			volume={127},
			number={1},
			pages={105\ndash 160},
			url={https://doi.org/10.4310/jdg/1717356156},
			review={\MR{4753500}},
		}
		
		\bib{eichmair2024penrose}{article}{
			author={Eichmair, Michael},
			author={Koerber, Thomas},
			title={The {P}enrose inequality in extrinsic geometry},
			date={2024},
			journal={preprint, arXiv:2411.02113},
		}
		
		\bib{HuiskenYauUniqueness}{article}{
			author={Eichmair, Michael},
			author={Koerber, Thomas},
			author={Metzger, Jan},
			author={Schulze, Felix},
			title={Huisken-{Y}au-type uniqueness for area-constrained {W}illmore
				spheres},
			date={2024},
			ISSN={0012-7094,1547-7398},
			journal={Duke Math. J.},
			volume={173},
			number={9},
			pages={1677\ndash 1730},
			url={https://doi.org/10.1215/00127094-2023-0045},
			review={\MR{4766841}},
		}
		
		\bib{Tsiamis2}{article}{
			author={Firester, Benjy},
			author={Tsiamis, Raphael},
			title={New minimal surfaces in the sphere from capillary minimal cones},
			date={2026},
			journal={preprint, arXiv:2602.20124},
		}
		
		\bib{Tsiamis1}{article}{
			author={Firester, Benjy},
			author={Tsiamis, Raphael},
			author={Wang, Yipeng},
			title={Area-minimizing capillary cones},
			date={2026},
			journal={preprint, arXiv:2601.18794},
		}
		
		\bib{Bernstein2}{article}{
			author={Fischer-Colbrie, Doris},
			author={Schoen, Richard},
			title={The structure of complete stable minimal surfaces in
				{$3$}-manifolds of nonnegative scalar curvature},
			date={1980},
			ISSN={0010-3640,1097-0312},
			journal={Comm. Pure Appl. Math.},
			volume={33},
			number={2},
			pages={199\ndash 211},
			url={https://doi.org/10.1002/cpa.3160330206},
			review={\MR{562550}},
		}
		
		\bib{GilbargTrudinger}{book}{
			author={Gilbarg, David},
			author={Trudinger, Neil},
			title={Elliptic partial differential equations of second order},
			edition={Second},
			series={Grundlehren der mathematischen Wissenschaften [Fundamental
				Principles of Mathematical Sciences]},
			publisher={Springer-Verlag, Berlin},
			date={1983},
			volume={224},
			ISBN={3-540-13025-X},
			url={https://doi.org/10.1007/978-3-642-61798-0},
			review={\MR{737190}},
		}
		
		\bib{Spectralgeometry}{article}{
			author={Girouard, Alexandre},
			author={Polterovich, Iosif},
			title={Spectral geometry of the {S}teklov problem (survey article)},
			date={2017},
			ISSN={1664-039X,1664-0403},
			journal={J. Spectr. Theory},
			volume={7},
			number={2},
			pages={321\ndash 359},
			url={https://doi.org/10.4171/JST/164},
			review={\MR{3662010}},
		}
		
		\bib{Gromov}{article}{
			author={Gromov, Misha},
			title={Dirac and {P}lateau billiards in domains with corners},
			date={2014},
			ISSN={1895-1074,1644-3616},
			journal={Cent. Eur. J. Math.},
			volume={12},
			number={8},
			pages={1109\ndash 1156},
			url={https://doi.org/10.2478/s11533-013-0399-1},
			review={\MR{3201312}},
		}
		
		\bib{HanLin}{book}{
			author={Han, Qing},
			author={Lin, Fanghua},
			title={Elliptic partial differential equations},
			series={Courant Lecture Notes in Mathematics},
			publisher={New York University, Courant Institute of Mathematical Sciences,
				New York; American Mathematical Society, Providence, RI},
			date={1997},
			volume={1},
			ISBN={0-9658703-0-8; 0-8218-2691-3},
			review={\MR{1669352}},
		}
		
		\bib{hoffmanmeeks}{article}{
			author={Hoffman, David},
			author={Meeks, William},
			title={Complete embedded minimal surfaces of finite total curvature},
			date={1985},
			ISSN={0273-0979,1088-9485},
			journal={Bull. Amer. Math. Soc. (N.S.)},
			volume={12},
			number={1},
			pages={134\ndash 136},
			url={https://doi.org/10.1090/S0273-0979-1985-15318-2},
			review={\MR{766971}},
		}
		
		\bib{HoffmanMeekshalfspace}{article}{
			author={Hoffman, David},
			author={Meeks, William},
			title={The strong halfspace theorem for minimal surfaces},
			date={1990},
			ISSN={0020-9910,1432-1297},
			journal={Invent. Math.},
			volume={101},
			number={2},
			pages={373\ndash 377},
			url={https://doi.org/10.1007/BF01231506},
			review={\MR{1062966}},
		}
		
		\bib{HongSaturnino}{article}{
			author={Hong, Han},
			author={Saturnino, Artur},
			title={Capillary surfaces: stability, index and curvature estimates},
			date={2023},
			ISSN={0075-4102,1435-5345},
			journal={J. Reine Angew. Math.},
			volume={803},
			pages={233\ndash 265},
			url={https://doi.org/10.1515/crelle-2023-0050},
			review={\MR{4649183}},
		}
		
		\bib{HuiskenYau}{article}{
			author={Huisken, Gerhard},
			author={Yau, Shing-Tung},
			title={Definition of center of mass for isolated physical systems and
				unique foliations by stable spheres with constant mean curvature},
			date={1996},
			ISSN={0020-9910,1432-1297},
			journal={Invent. Math.},
			volume={124},
			number={1-3},
			pages={281\ndash 311},
			url={https://doi.org/10.1007/s002220050054},
			review={\MR{1369419}},
		}
		
		\bib{KamburovWang}{article}{
			author={Kamburov, Nikola},
			author={Wang, Kelei},
			title={Nondegeneracy for stable solutions to the one-phase free boundary
				problem},
			date={2024},
			ISSN={0025-5831,1432-1807},
			journal={Math. Ann.},
			volume={388},
			number={3},
			pages={2705\ndash 2726},
			url={https://doi.org/10.1007/s00208-023-02591-0},
			review={\MR{4705750}},
		}
		
		\bib{ko1}{article}{
			author={Ko, Dongyeong},
			author={Yao, Xuan},
			title={Scalar curvature comparison and rigidity of $3 $-dimensional
				weakly convex domains},
			date={2024},
			journal={preprint, arXiv:2410.20548},
		}
		
		\bib{ko2}{article}{
			author={Ko, Dongyeong},
			author={Yao, Xuan},
			title={Capillary minimal slicing and scalar curvature rigidity},
			date={2026},
			journal={preprint, arXiv:2602.21071},
		}
		
		\bib{LMS1}{article}{
			author={Lamm, Tobias},
			author={Metzger, Jan},
			author={Schulze, Felix},
			title={Foliations of asymptotically flat manifolds by surfaces of
				{W}illmore type},
			date={2011},
			ISSN={0025-5831,1432-1807},
			journal={Math. Ann.},
			volume={350},
			number={1},
			pages={1\ndash 78},
			url={https://doi.org/10.1007/s00208-010-0550-2},
			review={\MR{2785762}},
		}
		
		\bib{LMS2}{article}{
			author={Lamm, Tobias},
			author={Metzger, Jan},
			author={Schulze, Felix},
			title={Local foliation of manifolds by surfaces of {W}illmore type},
			date={2020},
			ISSN={0373-0956,1777-5310},
			journal={Ann. Inst. Fourier (Grenoble)},
			volume={70},
			number={4},
			pages={1639\ndash 1662},
			url={https://doi.org/10.5802/aif.3375},
			review={\MR{4245583}},
		}
		
		\bib{Li}{article}{
			author={Li, Chao},
			title={A polyhedron comparison theorem for 3-manifolds with positive
				scalar curvature},
			date={2020},
			ISSN={0020-9910,1432-1297},
			journal={Invent. Math.},
			volume={219},
			number={1},
			pages={1\ndash 37},
			url={https://doi.org/10.1007/s00222-019-00895-0},
			review={\MR{4050100}},
		}
		
		\bib{LiZhouZhu}{article}{
			author={Li, Chao},
			author={Zhou, Xin},
			author={Zhu, Jonathan~J.},
			title={Min-max theory for capillary surfaces},
			date={2025},
			ISSN={0075-4102,1435-5345},
			journal={J. Reine Angew. Math.},
			volume={818},
			pages={215\ndash 262},
			url={https://doi.org/10.1515/crelle-2024-0075},
			review={\MR{4846024}},
		}
		
		\bib{Pacard}{article}{
			author={Pacard, Frank},
			title={Constant mean curvature hypersurfaces in {R}iemannian manifolds},
			date={2005},
			ISSN={0035-6298},
			journal={Riv. Mat. Univ. Parma (7)},
			volume={4*},
			pages={141\ndash 162},
			review={\MR{2198126}},
		}
		
		\bib{pacati2025some}{article}{
			author={Pacati, Alberto},
			author={Tortone, Giorgio},
			author={Velichkov, Bozhidar},
			title={Some remarks on singular capillary cones with free boundary},
			date={2025},
			journal={preprint, arXiv:2502.07697},
		}
		
		\bib{Schoen}{article}{
			author={Schoen, Richard},
			title={Uniqueness, symmetry, and embeddedness of minimal surfaces},
			date={1983},
			ISSN={0022-040X,1945-743X},
			journal={J. Differential Geom.},
			volume={18},
			number={4},
			pages={791\ndash 809},
			url={http://projecteuclid.org/euclid.jdg/1214438183},
			review={\MR{730928}},
		}
		
		\bib{SerrinHarnack}{article}{
			author={Serrin, James},
			title={On the {H}arnack inequality for linear elliptic equations},
			date={1955/56},
			ISSN={0021-7670,1565-8538},
			journal={J. Analyse Math.},
			volume={4},
			pages={292\ndash 308},
			url={https://doi.org/10.1007/BF02787725},
			review={\MR{81415}},
		}
		
		\bib{JamesSerrin}{article}{
			author={Serrin, James},
			title={A symmetry problem in potential theory},
			date={1971},
			ISSN={0003-9527},
			journal={Arch. Rational Mech. Anal.},
			volume={43},
			pages={304\ndash 318},
			url={https://doi.org/10.1007/BF00250468},
			review={\MR{333220}},
		}
		
		\bib{Simon}{article}{
			author={Simon, Leon},
			title={A strict maximum principle for area minimizing hypersurfaces},
			date={1987},
			ISSN={0022-040X,1945-743X},
			journal={J. Differential Geom.},
			volume={26},
			number={2},
			pages={327\ndash 335},
			url={http://projecteuclid.org/euclid.jdg/1214441373},
			review={\MR{906394}},
		}
		
		\bib{Taylor}{article}{
			author={Taylor, Jean},
			title={Boundary regularity for solutions to various capillarity and free
				boundary problems},
			date={1977},
			ISSN={0360-5302,1532-4133},
			journal={Comm. Partial Differential Equations},
			volume={2},
			number={4},
			pages={323\ndash 357},
			url={https://doi.org/10.1080/03605307708820033},
			review={\MR{487721}},
		}
		
		\bib{Traizet}{article}{
			author={Traizet, Martin},
			title={An embedded minimal surface with no symmetries},
			date={2002},
			ISSN={0022-040X,1945-743X},
			journal={J. Differential Geom.},
			volume={60},
			number={1},
			pages={103\ndash 153},
			url={http://projecteuclid.org/euclid.jdg/1090351085},
			review={\MR{1924593}},
		}
		
		\bib{volkmann}{thesis}{
			author={Volkmann, Alexander},
			title={Free boundary problems governed by mean curvature},
			type={Ph.D. Thesis},
			date={2015},
			note={available at
				\url{https://refubium.fu-berlin.de/bitstream/handle/fub188/11688/Thesis_Volkmann.pdf}},
		}
		
		\bib{WangXia}{article}{
			author={Wang, Guofang},
			author={Xia, Chao},
			title={Uniqueness of stable capillary hypersurfaces in a ball},
			date={2019},
			ISSN={0025-5831,1432-1807},
			journal={Math. Ann.},
			volume={374},
			number={3-4},
			pages={1845\ndash 1882},
			url={https://doi.org/10.1007/s00208-019-01845-0},
			review={\MR{3985125}},
		}
		
		\bib{Weinberger}{article}{
			author={Weinberger, Hans},
			title={Remark on the preceding paper of {S}errin},
			date={1971},
			ISSN={0003-9527},
			journal={Arch. Rational Mech. Anal.},
			volume={43},
			pages={319\ndash 320},
			url={https://doi.org/10.1007/BF00250469},
			review={\MR{333221}},
		}
		
		\bib{Wu}{article}{
			author={Wu, Yujie},
			title={Capillary surfaces in manifolds with nonnegative scalar curvature
				and strictly mean convex boundary},
			date={2025},
			ISSN={1073-7928,1687-0247},
			journal={Int. Math. Res. Not. IMRN},
			number={9},
			pages={Paper No. 14},
			url={https://doi.org/10.1093/imrn/rnaf106},
			review={\MR{4903125}},
		}
		
		\bib{Ye}{article}{
			author={Ye, Rugang},
			title={Foliation by constant mean curvature spheres},
			date={1991},
			ISSN={0030-8730,1945-5844},
			journal={Pacific J. Math.},
			volume={147},
			number={2},
			pages={381\ndash 396},
			url={http://projecteuclid.org/euclid.pjm/1102644918},
			review={\MR{1084717}},
		}
		
	\end{biblist}
\end{bibdiv}
\end{document}